\documentclass[reqno,11pt]{amsart}
\usepackage{xcolor}
\definecolor{refkey}{gray}{.8}   %   1=white  0=black
\definecolor{labelkey}{rgb}{0.4,0.4,0.8} % x,y,z \in [0,1]
 \usepackage{amsmath, amsthm, a4, latexsym, amssymb}
\usepackage[unicode]{hyperref}
\hypersetup{
    colorlinks=true, %set true if you want colored links
    linktoc=all,     %set to all if you want both sections and subsections linked
    linkcolor=blue,  %choose some color if you want links to stand out
}
\usepackage{srcltx}

\usepackage{geometry}
\usepackage{mathrsfs}  
\usepackage{enumitem}

\usepackage{amstext,bbm,nicefrac,physics, tikz}
\usepackage{subcaption}

\usetikzlibrary{calc}
\usetikzlibrary{decorations.pathmorphing}

\def\@rmrk#1#2{\refstepcounter
  {#1}\@ifnextchar[{\@yrmrk{#1}{#2}}{\@xrmrk{#1}{#2}}}

\makeatletter\@addtoreset{equation}{section}\makeatother

 \newfont{\bfit}{cmbxti10 scaled 1200}

\def\@rmrk#1#2{\refstepcounter
    {#1}\@ifnextchar[{\@yrmrk{#1}{#2}}{\@xrmrk{#1}{#2}}}

 \makeatletter\@addtoreset{equation}{section}\makeatother

\renewcommand{\d}{{\rm d}}

\newcommand{\cE}{{\mathcal E}}

\newcommand{\bE}{\mathbb{E}}
\newcommand{\e}{{\rm e} }

\newcommand{\eps}{\varepsilon}
\newcommand{\supp}{{\rm supp}}

\newcommand{\diam}{{\rm diam}}

\newcommand{\R}{\mathbb{R}}
\newcommand{\N}{\mathbb{N}}
\newcommand{\Z}{\mathbb{Z}}

\newcommand{\E}{\mathbb{E}}
\renewcommand{\P}{\mathbb{P}}
 \def\1{{\mathchoice {1\mskip-4mu\mathrm l} 
{1\mskip-4mu\mathrm l}
{1\mskip-4.5mu\mathrm l} {1\mskip-5mu\mathrm l}}}

\newcommand{\uu}{{u_{\mathrm{hom}}}} 

\usepackage{todonotes}

\newcommand{\bigset}[2]{ \big\{\,  #1 \,\big| \, #2 \,\big\} }

\DeclareMathOperator*{\esssup}{\mathrm{ess\!\;sup}}

\newcommand{\heap}[2]{\genfrac{}{}{0pt}{}{#1}{#2}}

\newcommand{\ssup}[1] {{\scriptscriptstyle{({#1}})}}

\renewcommand{\subsection}{\secdef \subsct\sbsect}
\newcommand{\subsct}[2][default]{%
  \refstepcounter{subsection}%
  \addcontentsline{toc}{subsection}{\protect\tocsubsection{}{\thesubsection}{#1}}%
  \vspace{0.15cm}
   {\flushleft\bf \arabic{section}.\arabic{subsection}~\bf #1  }
    \nopagebreak\nopagebreak}
\newcommand{\sbsect}[1]{\vspace{0.1cm}\noindent{\bf #1}\vspace{0.1cm}}

             {\nopagebreak {\hfill\rule{2mm}{2mm}}\\ }

\newtheorem{theorem}{Theorem}[section]
\newtheorem{lemma}[theorem]{Lemma}
\newtheorem{cor}[theorem]{Corollary}
\newtheorem{prop}[theorem]{Proposition}

\newtheorem{definition}[theorem]{Definition}

\newtheoremstyle{thm}{1.5ex}{1.5ex}{\itshape\rmfamily}{}
                      {\bfseries\rmfamily}{}{2ex}{}
\newtheoremstyle{rem}{1.3ex}{1.3ex}{\rmfamily}{}
                      {\itshape\rmfamily}{}{1.5ex}{}
\theoremstyle{rem}
\newtheorem{remark}{{\slshape\sffamily Remark}}[]

\refstepcounter{subsubsection}

\def\thebibliography#1{\section*{References}
  \list%
  {\arabic{enumi}.}%                          {\star}{\star}{\star} style of reference number {\star}{\star}{\star}
    {\settowidth\labelwidth{[#1]}\leftmargin\labelwidth
    \advance\leftmargin\labelsep
    \parsep0pt\itemsep0pt
    \usecounter{enumi}}
    \def\newblock{\hskip .11em plus .33em minus .07em}
    \sloppy                   % \clubpenalty4000\widowpenalty4000
    \sfcode`\.=1000\relax}

\let\oldtocsection=\tocsection

\let\oldtocsubsection=\tocsubsection

\let\oldtocsubsubsection=\tocsubsubsection

\renewcommand{\tocsection}[2]{\hspace{0em}\oldtocsection{#1}{#2}}
\renewcommand{\tocsubsection}[2]{\hspace{1em}\oldtocsubsection{#1}{#2}}
\renewcommand{\tocsubsubsection}[2]{\hspace{2em}\oldtocsubsubsection{#1}{#2}}

\begin{document}

\title[HJB in Random Geometries: Homogenization on Continuum Percolation Clusters]
     {\large Hamilton--Jacobi--Bellman Equations in Random Geometries: Homogenization on Continuum Percolation Clusters}

\maketitle

%\thispagestyle{empty}

%\vspace{-0.5cm}

\begin{center}
{\sc Rodrigo Bazaes}\footnote{https://rodrigobazaes.com, {\tt rodrigo@rodrigobazaes.com}},
{\sc Alexander Mielke}\footnote{WIAS Berlin and Humboldt Univerit\"at Berlin, Mohrenstra\ss{}e 39,  10117 Berlin,
    {\tt alexander.mielke@wias-berlin.de}}  
and  
{\sc Chiranjib Mukherjee}\footnote{Universit\"at M\"unster, Einsteinstrasse 62,
      48149 M\"unster, Germany, {\tt chiranjib.mukherjee@uni-muenster.de}}
\\[0.5em]
\textit{Universit\"at M\"unster, WIAS Berlin and HU Berlin, Universit\"at
  M\"unster}
\\[0.5em]
\today
\end{center}

\renewcommand{\thefootnote}{}
\footnote{\textit{AMS Subject
Classification:} {35F21, 35B27, 49L25, 60F10, 78A48, 82B43}}
\footnote{\textit{Keywords:} Hamilton-Jacobi-Bellman equations, stochastic homogenization, continuum percolation, convex variational analysis, entropy, min-max theorems}

\renewcommand{\thefootnote}{\arabic{footnote}}

\begin{quote}{\small {Abstract: We develop a quenched homogenization theory for optimal control problems related to Hamilton--Jacobi--Bellman equations
on random geometries arising from continuum percolation. The underlying state
space is the infinite connected component of a continuum percolation model
conditioned to contain the origin. As a consequence, the relevant law of the
environment is no longer translation invariant, and the geometry of the state
space itself becomes part of the homogenization problem. The associated
controlled diffusion is allowed to degenerate according to the distance to the
random boundary of the cluster. The admissible degeneracy regime is determined jointly by a balance between the sharp negative-moment threshold associated with the distance-to-boundary function of the continuum cluster and the coercivity of the underlying Hamiltonian. The framework applies to a broad
class of continuum percolation models, including models with long-range
correlations.

We prove that the corresponding rescaled value functions converge, locally in $L^p$ for every $p\ge 1$ on the rescaled random domains, almost surely
to a deterministic limit governed by an effective Hamiltonian. The
effective Hamiltonian admits dual variational characterizations involving
a class of curl-free gradients satisfying an induced
mean-zero condition determined by the geometry of the infinite cluster. The
resulting effective theory retains information about the continuum percolation
geometry, the degeneracy of the diffusion, and the nonstationarity induced by
conditioning on the infinite component.

The proof introduces a variational framework for homogenization under
nonstationary conditioned laws. Its main ingredients are random shifts
adapted to the geometry of the continuum infinite cluster, a two-step min--max construction for admissible gradients, and a novel relative entropy structure intrinsic to the stochastic
control representation. The latter provides the coercivity needed for the variational analysis and reveals a new connection between relative entropy and the construction of effective theories for nonlinear stochastic control problems. Although developed here for continuum percolation, this coercive-entropic variational framework applies equally well in the general setting of stationary ergodic random media on $\R^d$ and is therefore of independent interest. 
Through the stochastic representation established in a companion
work, the results yield quenched homogenization of the corresponding
Hamilton--Jacobi--Bellman equations.
}
}

\end{quote}

%\maketitle

\tableofcontents 
\vspace{-1em}

\section{Introduction}\label{sec introduction}

\subsection{Background}\label{subsec background}

In this work we study quenched stochastic homogenization for
Hamilton--Jacobi--Bellman equations posed on genuine continuum percolation
clusters. In contrast to classical stochastic homogenization, the randomness in
the present setting is carried not only by the coefficients, but also by the
geometry of the state space itself. The underlying random medium is given by
the infinite connected component of a continuum percolation model conditioned
to contain the origin, and the associated diffusion degenerates according to
the distance to the random boundary of the cluster. Concretely, the relevant environment law is not the original stationary
ergodic law $\P$ of the environment, but rather the conditioned law
\begin{equation}\label{P0}
\P_0(\cdot)=\P(\cdot|0\in\mathcal C_\infty),
\end{equation}
which is no longer translation invariant.

Continuum percolation models form a fundamental class of random media in
statistical mechanics and the study of disordered systems. Their large-scale
geometry exhibits a rich interplay between probabilistic connectivity,
irregular spatial structure, and analytic degeneracy, making them a natural
but highly challenging setting for stochastic homogenization. While stochastic
homogenization on discrete percolation structures
\cite{SS04,BB07,MP07,PRS15,K12,BMO16,D21} (see also the survey of Biskup \cite{B11}) and in stationary ergodic random media on the full space $\R^d$ 
(\cite{LPV87,So99,RT00,LS05,LS10,KRV06,KV08}; see also the survey of
Kosygina~\cite{K08}) has been studied extensively, substantially less is known
in genuine continuum percolation settings; see, for instance,
\cite{DG23}. To the best of our knowledge, the
present work is the first on homogenization of Hamilton--Jacobi--Bellman
equations on genuine continuum percolation clusters.

Let us emphasize that, in contrast to classical homogenization in random
media, here the random geometry of the state space plays a determining role in the
large-scale limit. Concretely, probabilistic, analytic, and geometric effects
manifest themselves through the nonstationarity of the conditioned law
$\P_0$, the diffusion degenerating according to the distance to the random
boundary, and the integrability properties of the distance-to-the-boundary
function induced by the geometry of the infinite cluster, respectively. In
particular, the geometric quantity governing the distance to the boundary of
the cluster, together with the necessary coercivity assumptions imposed on the
Hamiltonian, determines the admissible integrability regime of the
degeneracy required for homogenization.

A central theme of this work is that these three aspects are not independent.
Rather, they are reflected jointly in the constitutent variational theory and,
ultimately, in the homogenized limit. Consequently, the percolation geometry
enters directly into the effective Hamiltonian and the associated variational
formulas. In particular, the homogenized limit retains quantitative
information about the large-scale geometry of the infinite cluster, the
degeneracy of the diffusion, and the nonstationarity induced by conditioning
on the infinite component.

We now turn to a more concrete description of the setting and the main
homogenization result.

\subsection{Outline of the result}\label{subsec outline}

Let $\Omega$ denote the space of locally finite point configurations in
$\R^d$, $d\ge2$, equipped with a probability measure $\P$. The translation
group $\{\tau_x\}_{x\in\R^d}$ acts on $\Omega$ by
$\tau_x\omega=\omega-x:=\{y-x:y\in\omega\}$.
Given $\omega\in\Omega$, define
\[
\mathcal C(\omega):=\bigcup_{y\in\omega} B_{1/2}(y).
\]
We assume that $\P$ is stationary and ergodic under
$\{\tau_x\}_{x\in\R^d}$ and that, $\P$-almost surely,
$\mathcal C(\omega)$ contains a unique open infinite unbounded connected
component $\mathcal C_\infty(\omega)$. The event
$\Omega_0:=\{0\in\mathcal C_\infty\}$
then has strictly positive probability, and we define the conditioned measure
$\P_0$ as in \eqref{P0}; cf. Section~\ref{sec-assume-models}.

For $\omega\in\Omega_0$ and $x\in\mathcal C_\infty(\omega)$, consider the
controlled diffusion
\[
X_t
=
x+\int_0^t \sigma(X_s,\omega)\,dB_s
+\int_0^t (\div a)(X_s,\omega)\,ds
+\int_0^t a(X_s,\omega)c(s)\,ds.
\]
The diffusion matrix $a(\omega)$ is assumed to be symmetric, positive
semidefinite, and degenerately elliptic. More precisely, there exists a
measurable map $\xi:\Omega_0\to(0,\infty)$ such that
\begin{equation}\label{a deg}
\begin{aligned}
\xi(\omega)|v|^2
\le \langle a(\omega)v,v\rangle
\lesssim |v|^2,
\qquad \forall v\in\R^d,
\qquad\mbox{and}\qquad
\E_0[\xi^{-\chi}]<\infty,
\end{aligned}
\end{equation}
for a suitable exponent $\chi\in(\frac12,1)$. The quantity $\xi$ is tied to
the geometry of the cluster through the distance-to-the-boundary function and
encodes the admissible level of degeneracy compatible with the homogenization
theory (see below). 

For $x\in\R^d$ we define
\[
a(x,\omega):=a(\tau_x\omega),
\]
so that $\{a(x,\cdot)\}_{x\in\R^d}$ is stationary with respect to $\P$ and
$x\mapsto a(x,\omega)$ is sufficiently regular for any fixed $\omega$; we
refer to \ref{f1}-\ref{f1'} for details. It was proved
in~\cite[Theorem~2.1]{BMM26} that, for any fixed
$\omega\in\Omega_0$, the diffusion almost surely does not hit the boundary
$\partial\mathcal C_\infty(\omega)$ and that the above stochastic differential
equation admits a unique global strong solution.

For any fixed $\omega\in\Omega_0$, control $c$, and starting point
$x\in \mathcal C_\infty(\omega)$, let $P^{c,\omega}_x$ denote the law of the
quenched diffusion. Associated with this diffusion is the value function
\[
u_\eps(t,x,\omega)
=
\sup_c
\eps J^\omega_{f_\eps}
\left(
\frac t\eps,
\frac x\eps,
c
\right),
\]
where
\[
J_f^\omega(t,x,c)
=
E^{P_x^{c,\omega}}
\bigg[
f(X_t)-\int_0^t L(X_s,c(s),\omega)\,\d s
\bigg]\quad\mbox{and} \quad f_\eps(x)= \frac 1\eps f(\eps x),
\]
with $f$ being a uniformly continuous initial condition. Also 
$L(q,\omega)
=\sup_{p\in\R^d} \big[\langle p,q\rangle_a-H(p,\omega)\big]$ and $p\mapsto H(p,\omega)$ is a convex Hamiltonian satisfying suitable growth assumptions in $p$ with respect to the degenerate matrix $a(\omega)$ and continuity assumptions on $x\mapsto H(x,p,\omega):=H(p,\tau_x \omega)$; see \ref{f2}-\ref{f4} for
details. The value function solves \cite[Theorem~2.3]{BMM26} a
Hamilton--Jacobi--Bellman equation (HJB) posed on the rescaled random cluster
$\eps\mathcal C_\infty(\omega)$, without imposing boundary conditions on
$\partial(\eps\mathcal C_\infty(\omega))$; that is, for any fixed $\eps>0$
and $\omega\in\Omega_0$, $u_\eps(t,x,\omega)$ solves
\begin{equation}\label{eq-HJB'}
\begin{cases}
\partial_t u_\eps
=
\frac{\eps}{2}
\mathrm{div}\big(a\big(\frac x\eps,\omega\big)\nabla u_\eps\big)
+
H\big(\frac x\eps,\nabla u_\eps,\omega\big),
&\text{in } (0,T)\times\eps\mathcal C_\infty(\omega),
\\
u_\eps(0,x,\omega)=f(x),
&\text{on } \eps\mathcal C_\infty(\omega),
\end{cases}
\end{equation}
in the class of viscosity solutions of at most linear growth. The absence of
boundary conditions is a consequence of the above confinement property proved
in~\cite{BMM26}: the diffusion $X_t$ started at $x\in \mathcal C_\infty$ almost surely does not reach the boundary of
the cluster.

With this background, the main goal of the present article is to develop a
homogenization theory for $u_\eps$ almost surely under the conditioned law
$\P_0$ on the random domains
\[
D(\eps,R,\omega)
=
\eps\mathcal C_\infty(\omega)\cap B_R(0).
\]
which have asymptotically positive density for every fixed $R>0$: by the spatial ergodic theorem under $\P$, 
$|D(\eps,R,\omega)|
\to
|B_R(0)|\,\P[0\in\mathcal C_\infty]
>0$
as $\eps\to 0$ almost surely with respect to $\P_0$. Our main result is the following quenched homogenization theorem, establishing local $L^p$ convergence in the random domains $D(\eps,R,\omega)$, uniformly on compact time intervals: 
\begin{theorem}\label{theorem 1'}
We assume \ref{assump:est-erg}-\ref{assump:dist}, \ref{f1}-\ref{f1'} and
\ref{f2}-\ref{f4} stated below. Then for every $T,R>0$ and
$p\ge1$, and for $\P_0$-almost every $\omega\in\Omega_0$,
\[
\lim_{\eps\to0} \sup_{0\leq t \leq T} 
\frac1{|D(\eps,R,\omega)|}
\int_{D(\eps,R,\omega)}
|u_\eps(t,x,\omega)-u_{\rm hom}(t,x)|^p\,dx
=
0.
\]
Here $u_{\rm hom}$ is the unique viscosity solution of the homogenized
equation
\[
\partial_t u_{\rm hom}
=
\overline H(\nabla u_{\rm hom})
\qquad\text{in }(0,\infty)\times\R^d, \quad u_{\rm hom}(0,x)=f(x) \qquad\forall\, x\in\R^d. 
\]
The effective Hamiltonian $\overline H$ and the homogenized limit
$u_{\mathrm{hom}}$ admit variational representations
\begin{equation}\label{var Hbar'}
\overline H(\theta)
=
\inf_{G\in \mathcal G_\delta}
\bigg(
\esssup_{\P_0}
\bigg[
\frac12
\div\big(a(G+\theta)\big)
+
H(G+\theta)
\bigg]
\bigg),
\qquad
\theta\in\R^d,
\end{equation}
and
\[
u_{\mathrm{hom}}(t,x)
=
\sup_{y\in\R^d}
\Big[
f(y)-t\mathcal I\Big(\frac{y-x}{t}\Big)
\Big],\quad\mbox{where}\quad \mathcal I(y)
=
\sup_{\theta\in\R^d}
\big[
\langle\theta,y\rangle-\overline H(\theta)
\big].
\]
\end{theorem}
See Theorem \ref{thm}. In other words, despite 
the nonstationarity of the conditioned law $\P_0$, and the degeneracy of the
diffusion, the large-scale behavior of the control problem is governed by a
deterministic effective Hamiltonian. The resulting effective theory retains quantitative information about the geometry of the cluster, the degeneracy of the diffusion, and the
nonstationarity induced by the conditioned law $\P_0$.  
Indeed, the infimum in the variational representation \eqref{var Hbar'} is formulated
in terms of a class of admissible gradients $G\in\mathcal G_\delta$ (see 
Section~\ref{classG}) that are (i) curl-free on the
random cluster $\mathcal C_\infty$, (ii) belong to $L^{1+\delta}(\P_0)$,
where the exponent $\delta>0$ is determined jointly by the ellipticity
condition~\ref{f1'} on the degenerate matrix $a$ and by the availability of
suitable negative moments of the distance-to-the-boundary function on the
continuum percolation cluster (see below), and (iii) satisfy an {\it induced
mean-zero condition} compatible with the geometry of the infinite cluster and
the nonstationarity of $\P_0$. Unlike the classical stationary ergodic
setting, the latter condition is not imposed with respect to a translation-invariant
environment law, but is instead adapted to a large-scale geometric quantity encoding
arbitrarily long excursions of the infinite cluster in different directions.
Consequently, the geometric, probabilistic and analytic aspects are all 
encoded directly into the effective variational theory and hence into the homogenized limit itself.

It is worth emphasizing that the confinement of the diffusion away from the boundary (and consequently the absence of boundary conditions in the HJB equation \eqref{eq-HJB'}) does not remove the influence of the boundary geometry from the homogenization problem. Indeed, the regularity and degeneracy assumptions imposed on the diffusion coefficient must remain compatible with the available negative moments of the distance-to-boundary function. The latter moment condition is sharp and constrained by the geometry of the infinite cluster. At the same time, the integrability regime of the degeneracy is linked to the coercivity assumptions imposed on the Hamiltonian. Consequently, the admissible degeneracy regime emerges from a balance between the geometry of the continuum percolation cluster and the nonlinear structure of the Hamiltonian; see Remark~\ref{remark P6} for a more detailed discussion.

The approach developed here is based on a variational framework formulated
for continuum percolation under the conditioned law $\P_0$. It is
inspired by the seminal work of
Kosygina--Rezakhanlou--Varadhan~\cite{KRV06} on homogenization of
Hamilton--Jacobi equations in stationary ergodic random media on $\R^d$.
The conceptual challenges posed by continuum percolation necessitate the
development of new geometric, probabilistic, and variational tools. At the
same time, the coercive-entropic variational mechanism underlying the two-step
min--max construction is intrinsic to the stochastic control representation
and applies equally well in the classical stationary ergodic setting on
$\R^d$. The resulting framework incorporates geometry, degeneracy, and nonstationarity
simultaneously into the effective theory. We discuss the main ingredients of the
proof in Section~\ref{sec proof ideas}.

\subsection{Continuum percolation and assumptions}
\label{sec-assume-models}

In this section we recall the continuum percolation framework
from~\cite[Sec.~1.3]{BMM26}. Since the homogenization arguments rely
essentially on the point process structure, Palm distributions, and the
conditioned law $\P_0$, we state the framework in a self-contained form. 

\subsubsection{Basic definitions.}
\label{subsec pointPalm}

Let $\Omega$ be the space of all locally finite point subsets of $\R^d$
endowed with the smallest $\sigma$-algebra $\mathcal{G}$ that makes the maps
$\omega\mapsto \#(\omega\cap A)$
measurable for all Borel sets
$A\in \mathcal{B}(\R^d)$. A \textit{point process} is a probability measure
$\P$ on $(\Omega,\mathcal{G})$. On $\Omega$, there is a natural action of
$(\tau_x)_{x\in \R^d}$ given by $\tau_x \omega
:=
\omega-x
=
\big\{ y-x  : y\in \omega \big\}$.
We say that a point process is {\it stationary} if
$\P\circ \tau_x = \P$ for all $x\in \R^d$.
A stationary point process is ergodic with respect to
$(\tau_x)_{x\in \R^d}$ if, for every $A\in\mathcal G$ and every
$x\in \R^d$, the invariance $\tau_x A=A$ implies
$\P(A)\in\{0,1\}$. The \textit{intensity measure} of $\P$ is the measure on
$(\R^d,\mathcal{B}(\R^d))$ defined by
\begin{equation}\label{eq:Theta-def}
 \Theta(A)
 :=
 \int \#(\omega \cap A)\,\P(\d\omega)
 =
 \E[\#(\omega\cap A)].
\end{equation}
Here and throughout the sequel, $\E$ denotes expectation with respect to
$\P$. Observe that whenever $\P$ is stationary and $\Theta$ is locally finite,
there exists $\zeta\in (0,\infty)$ such that
$\Theta=\zeta \lambda$,
where $\lambda$ denotes Lebesgue measure on $\R^d$. The constant $\zeta$ is
called the \textit{intensity} of the point process.
\subsubsection{Palm measures.}
On an intuitive level, Palm measures formalize the idea of the distribution of
a point process conditioned on containing a fixed point $x\in \R^d$. Define
the measure $\mathfrak{C}$ on $\R^d\times \Omega$ by
\begin{equation}\label{eq:camp-meas-def}
  \mathfrak{C}(A)
  :=
  \E\bigg[
  \sum_{x\in \omega}\1_A(x,\tau_x\omega)
  \bigg],
  \qquad
  A\in \mathcal{B}(\R^d)\otimes \mathcal{G}.
\end{equation}

The measure $\mathfrak{C}$ admits a decomposition whenever $\P$ is stationary.
Indeed, by \cite[Theorem 3.3.1]{SW08}, if $\P$ is a stationary point process
with intensity $\zeta\in (0,\infty)$, then there exists a unique probability
measure $\P^{\ssup 0}$ on $(\Omega,\mathcal{G})$ such that
\begin{equation}\label{eq:palm-decomp}
\mathfrak{C}
=
\zeta \lambda\otimes \P^{\ssup 0}.
\end{equation}
We call $\P^{\ssup 0}$ the \textit{Palm measure} corresponding to $\P$. It may
be interpreted as the distribution of the point process conditioned on
containing the origin; see \cite[Proposition 9.5]{LP18}. In particular,
$\P^{\ssup 0}(0\notin \omega)=0$;
see \cite[Eq.~(9.7)]{LP18}. More generally, for every $x\in \R^d$, define
$\P^{\ssup x}
:=\P^{\ssup 0}\circ \tau_x$.

The decomposition \eqref{eq:palm-decomp} allows us to disintegrate $\P$ in
terms of $(\P^{\ssup x})_{x\in \R^d}$. Indeed, by
\cite[Theorem 3.3.3]{SW08}, if $\P$ is a stationary point process with
intensity $\zeta\in (0,\infty)$, then for every
$f\in L^1(\R^d\times \Omega)$, the map
$\omega\mapsto \sum_{x\in \omega}f(x,\omega)$
is measurable, and
\begin{equation}\label{eq:refined-camp-decomp}
  \E\bigg[
  \sum_{x\in \omega}f(x,\omega)
  \bigg]
  =
  \zeta
  \int_{\R^d}
  \E^{\ssup 0}[f(x,\tau_{-x}\omega)]
  \,\d x
  =
  \zeta
  \int_{\R^d}
  \E^{\ssup x}[f(x,\omega)]
  \,\d x.
\end{equation}

Similarly, one can define the $n$-fold Palm distributions
$\P^{\ssup{x_1,\cdots,x_n}}$ for
$x_1,\cdots,x_n\in \R^d$. In this case,
\begin{equation}\label{eq:multidim-campbell-eq}
  \E\bigg[
  \sum_{\heap{x_1,\cdots,x_n\in \omega}{\neq}}
  f(x_1,\cdots,x_n,\omega)
  \bigg]
  =
  \zeta^n
  \int_{(\R^d)^n}
  \E^{\ssup{x_1,\cdots,x_n}}
  [f(x_1,\cdots,x_n,\omega)]
  \,\d x_1\cdots \d x_n
\end{equation}
for all $f\in L^1((\R^d)^n\times \Omega)$, where $\neq$ indicates
that the sum is taken over pairwise distinct elements.

\subsubsection{Continuum percolation.}

Let $B_r(x)
=
\{y\in \R^d:|y-x|<r\}$
denote the open Euclidean ball centered at $x$ of radius $r>0$. For any
locally finite point set $\omega\in \Omega$, define the random open set
\begin{equation}
  \label{eq:Def.Cinfty}
  \mathcal{C}(\omega)
  :=
  \bigcup_{x\in \omega}B_{\frac 12}(x)
  \subset \R^d.
\end{equation}
The set $\mathcal C(\omega)$ decomposes into connected components. Whenever
there exists a unique open, connected, unbounded component, we denote it by
$\mathcal C_\infty(\omega)
\subset
\mathcal C(\omega)$.
The boundary of $\mathcal C_\infty(\omega)$ is denoted by
$\partial \mathcal C_\infty(\omega)$. We further define
\begin{equation}
  \label{eq:Omega.0}
  \Omega_0
  :=
  \bigg\{
  \omega\in \Omega:
  \mathcal C_\infty(\omega)
  \text{ exists and }
  0\in \mathcal C_\infty(\omega)
  \bigg\}.
\end{equation}
If $\P(\Omega_0)>0$, which we shall assume in
Assumption~\ref{assump:inf-comp}, then we define the conditional probability
measure $\P_0$ on $\Omega_0$ by
\[
\P_0(A)
:=
\P(A\mid \Omega_0)
=
\frac{\P(A\cap \Omega_0)}{\P(\Omega_0)},
\qquad
A\in \mathcal G.
\]

Since $\mathcal C_\infty(\omega)$ is open and connected whenever it exists,
every pair of points
$x,y\in \mathcal C_\infty(\omega)$ can be connected by a curve in
$C^1([0,1];\R^d)$. The intrinsic distance $\d_\omega$ is defined on
$\mathcal C_\infty(\omega)$ by
\begin{align}\label{def domega}
  \d_\omega(x,y)
  =
  \inf\bigg\{
  \int_0^1|\dot r(s)|\,\d s :
  &\ r\in C^1([0,1];\R^d),
  \
  r(0)=x,
  \
  r(1)=y,
  \\
  &
  r(s)\in \mathcal C_\infty(\omega)
  \text{ for all }s\in [0,1]
  \bigg\}.
\end{align}

To state Assumption~\ref{assump:exp-dec-dist-indshift} below, define
$n(\omega,e)\in \N$ for every $e\in \Z^d$ satisfying $|e|_1=1$ and every
$\omega\in\Omega_0$ by
\begin{equation}
  \label{def-n}
  n(\omega,e)
  :=
  \min\{k\in\N:k e\in{\mathcal C_\infty}(\omega)\}.
\end{equation}

\subsubsection{Assumptions on percolation.}
\label{subsec assump perc}

\begin{enumerate}[label=\textbf{(P\arabic*)}]
\itemsep0.2em

\item
$\P$ is stationary and ergodic with respect to
$(\tau_x)_{x\in \R^d}$. Moreover, $\P$ is ergodic with respect to $\tau_e$
for every $e\in \Z^d$ satisfying $|e|_1=1$; namely, every
$A\in \mathcal{G}$ satisfying $\tau_e A = A$ also satisfies
$\P(A)\in \{0,1\}$.
\label{assump:est-erg}

\item
The intensity measure $\Theta$ defined in \eqref{eq:Theta-def} satisfies
$\Theta(A)<\infty$ for every compact set $A\subset \R^d$. In particular,
$\Theta=\zeta \lambda$ for some $\zeta\in (0,\infty)$.
\label{assump:intensity}

\item
Recall the definitions of $\mathcal C(\omega)$ from
\eqref{eq:Def.Cinfty} and $\Omega_0$ from \eqref{eq:Omega.0}. We assume that
$\P(\Omega_0)>0$. Equivalently, with positive $\P$-probability, the random set
$\mathcal C(\omega)\subset \R^d$ possesses a unique open, connected, unbounded
component $\mathcal C_\infty(\omega)$ containing the origin.
\label{assump:inf-comp}

\item
\begin{enumerate}

\item
The Palm distribution $\P^{\ssup{x,y}}$ defined in
\eqref{eq:multidim-campbell-eq} and the intrinsic distance
$\d_\omega(x,y)$ defined in \eqref{def domega} satisfy, for some
$c_0,c_1,c_2>0$,
\begin{equation}
\label{eq:chem_dist_ineq}
\P^{\ssup{x,y}}
\big(
\d_\omega(x,y)\geq
c_0|x-y|_{\infty};
\
0,x,y\in {\mathcal C_\infty}
\big)
\leq
c_1 \e^{-c_2 |x-y|_{\infty}}
\qquad
\forall x,y\in \R^d.
\end{equation}

\item
There exist constants $c_3,c_4>0$ such that, for every $\varrho>0$,
\begin{equation}\label{eq:exp-dec-dist-indshift}
\P_0\big(
|\mathfrak{v}_e(\omega)|>\varrho
\big)
\leq
c_3\e^{- c_4 \varrho},
\qquad
\mathfrak v_e:=n(\omega,e) e,
\qquad
\forall\, e\in \Z^d \text{ with }|e|_1=1.
\end{equation}
\label{assump:exp-dec-dist-indshift}

\end{enumerate}
\label{assump:chem-dist}

\item
The FKG inequality holds. Namely, if $A_1,A_2\subset \Omega$ are increasing
events, then $\P(A_1\cap A_2)\geq \P(A_1)\P(A_2)$. Here increasing means whenever
$\omega\subset \omega^\prime$ and $\omega\in A_i$, then
$\omega^\prime\in A_i$ for $i=1,2$.
\label{assump:fkg}

\item
If $\d(0,\partial\mathcal C_\infty)$ denotes the Euclidean distance from the
origin to the boundary of the infinite cluster, then
\[
\E_0[\d(0,\partial\mathcal C_\infty)^{-\chi}]<\infty
\qquad\text{if and only if}\qquad
\chi\in(0,1).
\]
\label{assump:dist}
\end{enumerate}
In~\cite[Theorem~2.4]{BMM26} we verify these assumptions for concrete examples
of continuum percolation models. We emphasize that while
Assumptions~\ref{assump:est-erg}--\ref{assump:fkg} are primarily structural,
Assumption~\ref{assump:dist} is not a technical input into the homogenization
proof itself, but instead plays a more conceptual role. Indeed, it identifies
the geometric integrability regime within which the admissible degeneracy of
the diffusion and the coercivity assumptions imposed on the Hamiltonian remain
compatible. In this sense, Assumption~\ref{assump:dist} serves as a bridge
between the geometry of the continuum percolation cluster and the analytic
requirements of the homogenization theory; see
Remark~\ref{remark P6} below.

\subsection{Assumptions on the HJB equation.}
\label{sec-HJB}

Denote by $\mathcal{S}_d$ the space of $d\times d$ symmetric matrices. There
is a natural partial order on $\mathcal{S}_d$: for $A,B\in\mathcal{S}_d$, we
write $A\leq B$ if $B-A$ is positive semidefinite. For any symmetric positive
semidefinite matrix $a$ (defined below in Assumption~\ref{f1}), denote by
$\sigma\in \mathcal{S}_d$ the unique symmetric positive semidefinite matrix
satisfying $a=\frac12\sigma\sigma$.

We further define the inner product
$\langle\cdot,\cdot\rangle_a=\langle\cdot,\cdot\rangle_{a(\omega)}$
and the associated seminorm by
\begin{equation}\label{eq:a-inner-prod}
\langle v,w\rangle_a
:=
\langle a(\omega)v,w\rangle
=
\langle v,a(\omega)w\rangle,
\qquad
\|v\|_a
:=
\sqrt{\langle v,v\rangle_a},
\qquad
v,w\in\R^d.
\end{equation}

We are now ready to state the assumptions on the diffusion coefficient $a$,
the Hamiltonian $H$, and the initial condition $f$.

\subsubsection{Assumptions on the diffusion coefficient}

\begin{enumerate}[label=\textbf{(A\arabic*)}]
\itemsep0.2em

\item
\label{f1}

\begin{enumerate}

\item
The map $a:\Omega\to\mathcal S_d$
is positive semidefinite, and
$a(x,\omega):=a(\tau_x\omega)$
defines a stationary process with respect to the action of
$\{\tau_x\}_{x\in\R^d}$ on $(\Omega,\mathcal G,\P)$; recall
\eqref{def-stationary}. Moreover, for every
$\omega\in\Omega_0=\{0\in\mathcal C_\infty\}\subset\Omega$,
\[
\supp(a(\cdot,\omega))
\subset
\overline{\mathcal C_\infty(\omega)}.
\]

\item
The maps
\[
x\mapsto a(x,\omega)=a(\tau_x\omega)
\qquad\text{and}\qquad
x\mapsto \xi(\tau_x\omega)
\]
are globally Lipschitz continuous, and the square-root map
\[
x\mapsto \sigma(x,\omega)
\]
is locally Lipschitz continuous away from the boundary.

\item
Moreover,
\[
\mathcal C_\infty(\omega)\ni x
\mapsto
\div a(x,\omega)\in\R^d
\]
is locally Lipschitz continuous away from the boundary, and
$|\div a|$ is uniformly bounded.

\item
The restriction of $a$ to $\Omega_0$ satisfies the following ellipticity
bounds: there exist $c_5\in(0,\infty)$ and a measurable function
$\xi:\Omega_0\to(0,\infty)$ such that $\P_0$-almost surely,
\begin{equation}\label{eq:ellip-bounds}
\xi(\omega)|v|^2
\leq
\langle a(\omega)v,v\rangle
\leq
c_5|v|^2,
\qquad
\forall\, v\in\R^d.
\end{equation}

\end{enumerate}

\item
\label{f1'}

We assume that there exist $\delta>0$ and
\begin{equation}\label{alpha delta relation}
\alpha
>
2\left(\frac{1+\delta}{1-\delta}\right)
>
1+\delta
\end{equation}
(see also \eqref{eq:H1-H} below) such that the function $\xi(\cdot)$ in
\eqref{eq:ellip-bounds} satisfies
\begin{equation}\label{eq:xi-mom-bound}
\E_0\big[\xi^{-\chi}\big]<\infty,
\qquad
\chi=\chi(\alpha,\delta)
:=
\frac{\alpha}{2}
\frac{1+\delta}{\alpha-(1+\delta)}.
\end{equation}

\end{enumerate}

A concrete example satisfying the above assumptions is given by an explicit construction of a regularized
distance function to the boundary of the infinite cluster in
\cite[Sec.~3.3]{BMM26}.

\subsubsection{Assumptions on the Hamiltonian and initial condition}

\begin{enumerate}[label=\textbf{(H\arabic*)}]
\itemsep0.2em

\item
\label{f2}

The Hamiltonian
$H:\R^d\times\Omega\to\R$
satisfies, for every $\omega\in\Omega$, that the map
$p\mapsto H(p,\omega)$ is convex. Moreover, there exist constants
$c_6,\dots,c_9>0$ such that for all
$(p,\omega)\in\R^d\times\Omega_0$,
\begin{equation}\label{eq:H1-H}
c_6\|p\|_a^\alpha-c_7
\leq
H(p,\omega)
\leq
c_8\|p\|_a^\alpha+c_9.
\end{equation}
In addition, $H(p,\omega)=0$ for every $\omega\notin\Omega_0$.
Here $\alpha>1+\delta$ and $\delta>0$ are as in
\eqref{alpha delta relation}. Equivalently, defining
$\alpha^\prime:=\frac{\alpha}{\alpha-1}$,
there exist constants $c_{10},\dots,c_{13}>0$ such that
\begin{equation}\label{eq:H1-L}
c_{10}\|q\|_a^{\alpha^\prime}-c_{11}
\leq
L(q,\omega)
\leq
c_{12}\|q\|_a^{\alpha^\prime}+c_{13},
\end{equation}
where
\begin{equation}\label{def L}
L(q,\omega)
:=
\sup_{p\in\R^d}
\big[
\langle p,q\rangle_a-H(p,\omega)
\big].
\end{equation}
Moreover, the map
$x\mapsto L(q,\tau_x\omega)$
is continuous and
$
L(q,\omega)=0
\qquad
\text{for every }\omega\notin\Omega_0.
$

\item
\label{f3}

For every $x\in\R^d$, define
$H(x,p,\omega):=H(p,\tau_x\omega)$ and
$L(x,q,\omega):=L(q,\tau_x\omega)$.
By Assumption~\ref{assump:est-erg}, $\P$ is stationary with respect to
$\{\tau_x\}_{x\in\R^d}$. Consequently, the maps
\[
x\mapsto H(x,p,\omega)=H(p,\tau_x\omega),
\qquad
x\mapsto L(x,q,\omega)=L(q,\tau_x\omega)
\]
define stationary processes with respect to
$\{\tau_x\}_{x\in\R^d}$ and $\P$. We further assume that there exist constants $c_{14},c_{15}>0$ such that, for
every $\omega\in\Omega_0$, every $x,y\in\R^d$, and every $p\in\R^d$,
\begin{equation}\label{eq:H2}
|H(x,p,\omega)-H(y,p,\omega)|
\leq
(c_{14}|p|^\alpha+c_{15})|x-y|.
\end{equation}

\item
\label{f4}

The initial condition $f:\R^d\to\R$
is uniformly continuous. In particular, for every $\delta>0$, there exists
$K_\delta>0$ such that for all $x,y\in\R^d$,
\begin{equation}
\label{eq:H6}
|f(x)-f(y)|
\leq
K_\delta |x-y|+\delta.
\end{equation}

\end{enumerate}

A basic example of a Hamiltonian satisfying
Assumptions~\ref{f2}--\ref{f3} is
\[
H(x,p,\omega)
=
\bigl(1+\langle p,a(x,\omega)p\rangle\bigr)^{\alpha/2}
\asymp
1+\|p\|_{a(x,\omega)}^\alpha,
\]
for $\alpha\ge2$; see~\cite[Example~2.2]{BMM26}. The associated Lagrangian
then satisfies \eqref{eq:H1-L}.

\begin{remark}[Geometry, coercivity and admissible degeneracy]\label{remark P6}
We underline the following interplay between \ref{assump:dist}, \ref{f1}-\ref{f1'} and \ref{f2}: the geometry of the 
percolation cluster determines through \ref{assump:dist} the admissible range of negative moments of the
distance-to-the-boundary function, which in turn constrains the allowable
degeneracy of the diffusion coefficient in \ref{f1'}. At the same time, the coercivity of
the Hamiltonian in \ref{f2} dictates how much degeneracy can be accommodated within the
homogenization theory.

Indeed, \eqref{alpha delta relation} and \eqref{eq:xi-mom-bound} imply
that $\chi<1$. Moreover, the lower bound in
\eqref{eq:ellip-bounds} and the Lipschitz continuity assumption in
\ref{f1}(b) imply that
$\xi(\omega)\lesssim d(0,\partial\mathcal C_\infty(\omega))$.
Consequently, the moment bound in \eqref{eq:xi-mom-bound} yields
\[
\E_0\Big[
d(0,\partial\mathcal C_\infty)^{-\chi}
\Big]
<\infty,
\]
which is consistent with Assumption~\ref{assump:dist}. More generally, Assumption~\ref{assump:dist} shows that the
distance-to-the-boundary function admits negative moments only up to
exponents {\it strictly} smaller than one. This threshold already appears in the supercritical Boolean continuum
percolation model (see \cite[Theorem~2.4]{BMM26}) and therefore reflects a
genuine geometric feature of the underlying random medium rather than a
technical artifact of the analysis. On the other hand, the exponent $\chi$
required in the homogenization theory depends explicitly on the coercivity
exponent $\alpha$ of the Hamiltonian in \ref{f2}: stronger coercivity permits weaker
integrability assumptions on the ellipticity function $\xi$, while the
geometry of the continuum cluster imposes the intrinsic restriction
$\chi<1$. Consequently, the admissible degeneracy regime emerges from a balanced 
interplay between the geometry of the random cluster and the coercive
structure of the Hamiltonian. In this sense, the moment condition imposed on
$\xi$ is plausibly close to optimal.\qed 
\end{remark}

\section{The main result and proof outline}\label{sec-results}

In Section \ref{sec-main-results} we state the main homogenization theorem and record several
consequences and interpretations. The principal ingredients of the proof are
discussed in Section~\ref{sec proof ideas}.

\subsection{Main results.}\label{sec-main-results}

Here $(\mathscr{X},\mathcal{F},\{\mathcal F_t\}_{t\geq 0},P)$ is a filtered
probability space carrying an auxiliary $d$-dimensional Brownian motion
$(B_t)_{t\geq 0}$ adapted to the filtration $(\mathcal{F}_t)_{t\geq 0}$. We
assume that the law $P$ of $(B_t)_{t\geq 0}$ is independent of the law $\P$ of
the point process discussed in Section~\ref{sec-assume-models}. Let
\begin{equation}\label{def-class-C}
\begin{aligned}
\mathbf C_T
=
\bigg\{
c:[0,T]\times \mathscr X \to \R^d:
c \mbox{ is progressively measurable and }
E^P\bigg[\int_0^T|c(s)|^2\,\d s\bigg]<\infty
\bigg\}.
\end{aligned}
\end{equation}

Under Assumptions \ref{assump:est-erg}--\ref{assump:inf-comp} and \ref{f1},
for every $c\in \mathbf C_T$, every $\omega\in \Omega_0$, and every starting
point $x\in \mathcal C_\infty(\omega)$, consider the controlled diffusion
process
\begin{equation}\label{eq:controlled-sde}
	X_t = x + \int_0^t \sigma(X_s)\d {B}_s+ \int_0^t
        (\mathrm{div}\;\! a )(X_s)\d s+\int_0^ta(X_s)c(s)\d s
        \quad\mbox{a.s.}\quad \forall t\geq 0,
\end{equation}
written alternatively as
\begin{equation}\label{SDE}
   \d X_t= \mathbf{b}(X_t,c_t) \d t + \sigma(X_t) \d B_t,
   \qquad\mbox{where}\quad
   \mathbf b(y,c)= a(y)c+ \mathrm{div}\,a(y).
 \end{equation}
By \cite[Theorem~2.1]{BMM26}, the diffusion almost surely does not hit
$\partial\mathcal C_\infty(\omega)$ and \eqref{eq:controlled-sde} therefore admits a unique global strong
solution whose law is denoted by $P^{c,\omega}_x$. Define 
\begin{equation}\label{eq:u-eps-def}
	u_{\eps}(t,x,\omega):= \sup_{c\in \mathbf{C}_T} \eps\, J^\omega_{f_\eps}
        \bigg(\frac t\eps, \frac x \eps, c\bigg)
\end{equation}
with
\begin{align}
J_f^\omega(t,x,c)
:=
E^{P_x^{c,\omega}}
\bigg[
f(X_t)-\int_0^t L(X_s,c(s),\omega)\,\d s
\bigg],
\qquad
f_\eps(x):=\frac{1}{\eps} f(\eps x).
\label{eq:cost-fun-def}
\end{align}
It is shown in \cite[Theorem~2.3]{BMM26} that under Assumptions \ref{assump:est-erg}--\ref{assump:inf-comp}, \ref{f1} and \ref{f2}--\ref{f4}, for any fixed $\eps>0$ and $\omega\in \Omega_0$, $u_\eps$ is a viscosity solution of the rescaled
Hamilton--Jacobi--Bellman equation \eqref{eq-HJB'} of at most linear growth,
without requiring any boundary condition. 

Given $\eps,R>0$ and $\omega\in\Omega_0$, define the random domain 
\begin{equation}
\label{def:D-eps}
D(\eps,R,\omega):=\eps \mathcal C_\infty(\omega)\cap B_R(0).
\end{equation}
Here is the main result of the paper: 
\begin{theorem}
\label{thm}
Assume \ref{assump:est-erg}--\ref{assump:dist}, \ref{f1}--\ref{f1'}, and
\ref{f2}--\ref{f4}. Then $\P_0$-almost surely, for every $T,R>0$ and every
$p\geq 1$,
\begin{equation}
\lim_{\eps\to 0} \sup_{0\leq t \leq T} 
\frac 1 {|D(\eps,R,\omega)|}
\int_{D(\eps,R,\omega)}
\big|u_\eps(t,x,\omega)-u_{\rm hom}(t,x)\big|^p\,\d x
=0,
\end{equation}
where $u_{\rm hom}$ is the unique viscosity solution of
\begin{equation}\label{eq:hom-pde}
  \begin{cases}
	\partial_tu_{\rm hom}= \overline H(\nabla u_{\rm hom})&
        \text{ in } (0,\infty)\times \R^d,\\
           u_{\rm hom}(0,x)= f(x)&\text{ on } \R^d.
\end{cases}
\end{equation}
Here, the effective Hamiltonian $\overline H$ is given by the dual variational
formula
\begin{equation}
  \label{eq:effective.hamiltonian}
  \begin{aligned}
    \overline{H}(\theta)
    &=
    \sup_{(b,\phi)\in \mathcal{E}}
    \bigg(
    \int
    \Big[
    \frac12\mathrm{div}(a\theta)
    +
    \langle \theta,b\rangle_a
    -
    L(b,\omega)
    \Big]
    \phi(\omega)\,\P_0(\d\omega)
    \bigg)
\\
    &=
    \inf_{G\in \mathcal{G}_{\delta}}
    \left(
    \mathrm{ess\,sup}_{\P_0}
    \left[
    \frac12\mathrm{div}(a(G+\theta))
    +
    H(G+\theta)
    \right]
    \right),
	\end{aligned}
\end{equation}
where the classes $\mathcal E$ and $\mathcal G_\delta$ are defined in
\eqref{class-E2} and Section~\ref{classG}, respectively. Moreover,
$u_{\rm hom}$ is given by the Hopf--Lax formula
\begin{equation}\label{eq:lax-oleinik}
  	u_{\rm hom}(t,x)
    =
    \sup_{y\in \R^d}
    \bigg(
    f(y)-t\,\mathcal I\Big(\frac{y-x}{t}\Big)
    \bigg),
  \qquad
 \mathcal I(y)
 :=
 \sup_{\theta\in \R^d}
 \big[
 \langle \theta,y\rangle-
 \overline H(\theta)
 \big].
\end{equation}
\end{theorem}

Theorem~\ref{thm} shows that, despite the lack of stationarity of the
conditioned law $\P_0$ and the presence of a degenerate diffusion on a random
continuum percolation cluster, the rescaled value functions admit a
deterministic homogenized limit. The effective Hamiltonian retains information
about both the geometry of the cluster and the degeneracy of the diffusion
through the admissible gradient class $\mathcal G_\delta$ and the variational
formula \eqref{eq:effective.hamiltonian}. %The proof of Theorem~\ref{thm} is given in Section~\ref{subsec proof thm}.

\begin{remark}[Local $L^p$ convergence in $D(\eps,R,\omega)$]\label{remark-spatial-average}
A distinctive feature of Theorem~\ref{thm} is that homogenization occurs locally in $L^p$ for all $p\geq 1$ after spatial averaging over the random domains
\[
D(\eps,R,\omega)
=
\eps\mathcal C_\infty(\omega)\cap B_R(0).
\]
This is in contrast to the classical stochastic homogenization theory for
HJ equations on the full space $\R^d$, where 
locally uniform convergence is shown e.g. in Kosygina-Rezakhanlou-Varadhan \cite{KRV06}.

This difference is intrinsic to the continuum percolation geometry. Indeed,
the equations are posed on the random domains
$\eps\mathcal C_\infty(\omega)$, whose complements contain microscopic holes
throughout space. To see why locally uniform convergence is not the natural
notion here, suppose one extends $u_\eps$ to all of $\R^d$ by setting
\[
\widetilde u_\eps(t,x,\omega)
=
f(x),
\qquad
x\notin \eps\mathcal C_\infty(\omega).
\]
If $\widetilde u_\eps$ converged locally uniformly to the continuous limit
$u_{\rm hom}$, then the density of microscopic holes in the macroscopic limit
would force $u_{\rm hom}(t,x)=f(x)$ for all
$t\ge0$ and $x\in\R^d$. Indeed, for every $x\in\R^d$ one can choose
$x_\eps\notin \eps\mathcal C_\infty(\omega)$ with $x_\eps\to x$; hence
\[
u_{\rm hom}(t,x)
=
\lim_{\eps\to0}\widetilde u_\eps(t,x_\eps,\omega)
=
\lim_{\eps\to0} f(x_\eps)
=
f(x), 
\]
which is a trivial homogenized evolution which is incompatible with the nontrivial effective equation
\eqref{eq:hom-pde}. Theorem~\ref{thm} hence uses the natural mode of convergence compatible
with continuum percolation geometry. 
\qed
\end{remark}

\begin{remark}[Non-divergence form]\label{remark-divergence-form}
The equation \eqref{eq-HJB'} can also be rewritten in non-divergence form as
$$
\begin{aligned}
&\partial_t u_\eps=\frac 12 \mathrm{Trace}\Big(a\big(\frac x \eps,\omega\big)
\mathrm{Hess}_x u_\eps\Big)
+
\hat H\Big(\frac x \eps,\nabla u_\eps,\omega\Big),
\quad \mbox{with}
\quad \hat H(x,p,\omega)
=
H(x,p,\omega)
+
\frac 12 \mathrm{div}(a(x,\omega))\cdot p .
\end{aligned}
$$
Then our assumptions on $H$ translate to assumptions on $\hat H$ if we
additionally assume that $|p\cdot \mathrm{div}\,a|\leq C\|p\|_a$, which is
stronger than our current assumption $|\mathrm{div}\,a|\leq C^\prime$ in
Assumption~\ref{f1}. Consequently, homogenization of the above equation is
covered by Theorem~\ref{thm}. Also note that since the homogenized equation
\eqref{eq:hom-pde} is defined in $(0,\infty)\times \R^d$, the choice of $T$ in
\eqref{eq-HJB'} does not play any role as $\eps\to0$.
\qed
\end{remark}

\begin{remark}[Quenched large deviations]\label{remark LDP}
A particular case of $H$ which is appealing from a probabilistic viewpoint is
the quadratic Hamiltonian
\begin{equation}
  \label{Hb}
 H_b(p,\omega):= \frac 12 \|p\|_{a}^2+ \langle b(\omega),p\rangle_a.
\end{equation}
For any $\omega\in \Omega_0$, let $P^\omega_0$ denote the law of the diffusion
\begin{equation}\label{eq-diffusion}
  \d X_t= \sigma(X_t,\omega) \d W_t
  + \mathrm{div}\!\; a(X_t,\omega)\d t
  +a(X_t,\omega)b(X_t,\omega)\,\d t
\end{equation}
starting at $0$ in the environment $\omega$, where $(W_t)_{t\geq 0}$ is a
standard Brownian motion in $\R^d$ whose law is independent of $\P$. Let
$b:\Omega\to \R^d$ be such that $x\mapsto b(\tau_x\omega)$ defines a
stationary process with respect to translations, $H_b$ defined in \eqref{Hb}
satisfies \eqref{eq:H2} (for instance, it suffices to assume that
$x\mapsto b(x,\omega)$ is bounded and Lipschitz), and the diffusion process
$X_t$ above does not hit the boundary of the cluster in finite time almost
surely. Then, under the assumptions of Theorem~\ref{thm}, for
$\P_0$-almost every realization $\omega\in \Omega_0$, the distributions
$P^\omega_0[X_t/t \in \cdot]$ satisfy a quenched large deviation principle with
rate function
\begin{align}
&I(x)= \sup_{\theta\in \R^d} \{\langle \theta, x\rangle- \overline H(\theta)\},
\label{def-rate} \quad\mbox{with}\\
&\overline H_b (\theta)=
\inf_{G\in \mathcal{G}_{\delta}}
\,\mathrm{ess\,sup}_{\P_0}
\bigg[
\frac 12 \mathrm{div} (a\,G)
+
\langle b,G{+}\theta\rangle_a
+
\frac 12 \|G{+}\theta\|_{a}^2
\bigg],
\label{def-conj-rate}
 \end{align}
where $\overline H_b$ is defined as in \eqref{eq:effective.hamiltonian}. In
other words, for $\P_0$-a.e. $\omega\in \Omega_0$, every open set
$G\subset \R^d$, and every closed set $F\subset \R^d$,
\[
\liminf_{t\to\infty}
\frac1t\log P^\omega_0\big[X_t/t \in G\big]
\geq
-\inf_{x\in G} I(x),
\qquad
\limsup_{t\to\infty}
\frac1t\log P^\omega_0\big[X_t/t \in F\big]
\leq
-\inf_{x\in F} I(x).
\]
We refer to Remark~\ref{remark proof LDP} for details.
\qed
\end{remark}

\begin{remark}[Quenched diffusion in a random potential]\label{remark-potential}
We can also consider Hamiltonians of the type
\begin{equation*}
	H_{b,V}(p,\omega):=\frac{1}{2}\|p\|^2_a+\langle b(\omega),p\rangle_a-V(\omega),
\end{equation*}
and establish a $\P_0$-almost sure large deviation principle for the
distribution of $X_t/t$ under the measure
$\d Q_{0}^\omega \propto \e^{-\int_0^t V(X_s,\omega)\,\d s}\,\d P_{0}^{\omega}$
under suitable moment assumptions on the potential $V$ with respect to $\P_0$.
This corresponds to an absorbing random environment; see the
fundamental work of Sznitman~\cite{S94} on Brownian motion in a Poissonian
potential; see also the survey of Kosygina~\cite[Section~7]{K08}.
\qed
\end{remark}

\begin{remark}[Boundary conditions]\label{remark Neumann}
As noted in Remark \ref{remark P6}, the geometry of the random boundary enters the present theory through the degeneracy of the diffusion coefficient and the moment condition of Assumption~\ref{assump:dist}. Also, the effective limit retains quantitative information about
the distance to the boundary of the continuum percolation cluster. One may ask whether the present framework can be extended to uniformly elliptic diffusions equipped with reflecting (Neumann or oblique) boundary conditions on the random boundary. Such an extension would require controlling the associated boundary local time on a highly irregular random boundary and incorporating its contribution into the variational theory. Since the geometric regularity available in continuum
percolation is substantially weaker than that typically assumed in the
classical theory of reflected diffusions, this appears to introduce additional analytical and probabilistic difficulties. We leave this question for future work.\qed 
\end{remark}

\subsection{Ingredients of the proof.}
\label{sec proof ideas}

The goal of this section is to highlight the main ingredients of the proof.
As a guiding philosophy, we draw inspiration from a novel method developed by
Kosygina--Rezakhanlou--Varadhan \cite{KRV06} for treating viscous
Hamilton--Jacobi equations in stationary ergodic environments; see also
Kosygina \cite[Sec.~6]{K08} for a survey of this approach and
Kosygina--Varadhan \cite{KV08} for an extension to time-dependent settings.
The roots of this method go back to the pioneering work of
Lions--Papanicolaou--Varadhan \cite{LPV87} and to the framework of the
{\it environment seen from the particle} developed by Papanicolau- Varadhan and Kozlov in
\cite{PV81,PV82,K85,KV86}.

While the variational philosophy of \cite{KRV06} serves as an important source
of inspiration, the continuum percolation framework here combines
several sources of difficulty for which new ideas need to be developed. Before outlining those, we briefly recall the main ingredients
of the method of \cite{KRV06}.

\medskip

\noindent{\bf The approach of \cite{KRV06}.}
Consider a stationary ergodic environment
$(\Omega,\mathcal F,\P)$ and a viscous Hamilton--Jacobi equation on the full
space $\R^d$ with $a(\omega)\equiv \mathrm{id}$ and convex Hamiltonian
satisfying $H(p,\omega)\asymp |p|^\alpha$. We briefly recall the three main
ingredients of the method developed in \cite{KRV06}.

\smallskip

\noindent{\it Lower bound.}
The starting point is the optimal-control representation of the solution. One
restricts the controls to stationary controls of the form
$c(x,\omega)=b(\tau_x\omega)$ and studies the corresponding environment process
seen from the particle. Using invariant densities for the generator of the
environment process together with ergodic properties of the stationary
environment, one obtains a deterministic variational lower bound for the
effective Hamiltonian.

\smallskip

\noindent{\it Convex variational analysis.}
The lower bound naturally leads to a variational problem involving drift fields
and invariant densities. By introducing suitable Lagrange multipliers and
applying min--max arguments, one constructs approximate gradients whose weak
limits yield stationary mean-zero gradients
$v\in L^\alpha(\P)$. A crucial feature of the stationary ergodic setting is
that both stationarity and the mean-zero property are inherited directly from
the translation invariance of $\P$.

The success of this min--max approach relies, among other requirements, on
suitable compactness properties of the underlying variational problems. In the
stationary ergodic setting of \cite{KRV06}, this compactness becomes available
by restricting the relevant variational problems to bounded regions and then
passing to the limit. Combined with the coercivity assumptions on the
Hamiltonian, this leads to sufficient moment bounds to extract weak limits of
the approximate gradients and ultimately construct the admissible stationary
mean-zero gradient fields entering the variational characterization of the
effective Hamiltonian.

\smallskip

\noindent{\it Upper bound.}
The stationary gradient $v$ gives rise to a corrector
$V(x,\omega)
=
\int_{0\to x}\langle v,\d z\rangle$.
A key step is to establish the sublinear growth property
$V(x,\omega)=o(|x|)$ as $|x|\to\infty$.
This follows from the mean-zero property of $v$, ergodicity, and the coercivity
assumptions imposed on the Hamiltonian. One then perturbs the affine function
$\langle p,x\rangle+t\overline H(p)$
by the corrector $\eps V(x/\eps,\omega)$ and uses comparison arguments to
obtain a matching upper bound, thereby completing the homogenization proof.

\medskip

\noindent{\bf The current method.}
We now explain the main ideas underlying the proof in the continuum
percolation setting. For the lower bound, we also work with the 
environment seen from the particle. Here, since $\P_0$ is not invariant
under the usual translations $\tau_x$, the classical stationary ergodic theory
cannot be applied. To overcome this difficulty, we introduce an
induced shift adapted to the geometry of the infinite cluster. This shift is
defined through successive arrivals of $\mathcal C_\infty$ along the coordinate
directions and therefore depends intrinsically on the underlying percolation
configuration. We establish the ergodic properties of the environment process for the controlled diffusion with respect to this random shift and $\P_0$ in
Section~\ref{sec ergodic induced}.

Next we start with the stochastic representation formula together with square-integrable
progressively measurable controls on an auxiliary probability space
$(\mathscr X,\mathcal F,P)$. We derive a variational lower bound by working
with Lipschitz drift fields
$b\in L^1_a(\phi\,d\P_0)$; see Section~\ref{sec-b-phi}. The ergodicity of the
above induced shift under $\P_0$ can then be used to obtain uniform asymptotic lower bound at the origin;
these are done via a priori estimates shown in Lemmas~\ref{lemma:control_restriction}-\ref{lemma:lower-bound-lemma1-extended}. A further substantial step is then required to
upgrade this pointwise statement to the locally averaged $L^p$ lower bound
appearing in Theorem~\ref{thm lower bound}. The argument for its proof is actually quite subtle -- to invoke the a priori estimates developed in Section \ref{subsec proof thm}, one has to trade carefully with the supremum over the invariant pairs $(b,\phi)\in\mathcal E$ in the variational representation of $u_{\rm{hom}}$ (see Lemma \ref{lemma2}) and the necessity to take spatial averages over the randomly evolving domains $D(\eps,R,\omega)=\eps\mathcal C_\infty(\omega)\cap B_R(0)$. As noted in Remark~\ref{remark-spatial-average}, this part is also closely tied to the geometry of the continuum percolation cluster and to the imperative role of spatial averaging in the homogenization process. The detailed analysis is carried out in Section~\ref{sec proof lb}.

The variational analysis in Section~\ref{sec-equivalence} constitutes one of
the main conceptual and technical innovations of the paper. Here again, the
lack of translation invariance of $\P_0$ enters in an essential way and
significantly influences the structure of the variational theory. Concretely, one cannot expect limiting gradients to satisfy the classical mean-zero condition appearing in stationary ergodic
homogenization. Instead, the geometry of the infinite cluster naturally leads
to an induced mean-zero condition formulated in terms of the induced shifts
introduced above.

To construct gradients satisfying this condition, we develop a new min--max
framework. Unlike the stationary ergodic setting, the restriction of the
variational problems to bounded regions is not compatible with the geometric
information encoded by the induced shifts. Indeed, the induced mean-zero
condition requires keeping track of arbitrarily long random excursions of the
continuum cluster, while gradients are still defined through the usual
translations $\tau_x$, which do not preserve $\P_0$. A different compactness
mechanism is therefore required.

The first min--max step exploits the intrinsic coercivity of the Hamiltonian
with respect to the degenerate metric induced by $a$. This coercivity
propagates to the accompanying variational problem and yields the compactness
needed to pass to weak limits. At this stage, the choice of the admissible
class $(b,\phi)\in\mathcal E$ in \eqref{class-E2} is crucial. We work with
Lipschitz drift fields $b\in L^1_a(\phi\,d\P_0)$. Had one restricted to
uniformly bounded drifts $b\in L^\infty(\P_0)$, the resulting class would not
be closed in $L^p(\P_0)$ for $p\ge1$, thereby obstructing the weak compactness
required in the first min--max step; see Lemma~\ref{lemma-minmax1}.

For the second min--max step, we also develop a novel subtractive relative entropy term.
This entropy structure is naturally compatible with the stochastic-control
representation and provides the additional coercivity required to complete the
variational argument; see Lemma~\ref{lemma-minmax2}. Combined with the moment
assumption \eqref{eq:xi-mom-bound},
$$
\chi=
\frac{\alpha(1+\delta)}
{2(\alpha-(1+\delta))} \qquad\mbox{yields a weak limit}\qquad G\in L^{1+\delta}(\P_0).
$$
A noteworthy feature of the present approach is that the limiting gradient
field $G$ automatically inherits the structural properties dictated by the
underlying medium. In particular, $G$ is both curl-free and satisfies the
induced mean-zero condition compatible with the conditioned law $\P_0$ and the
geometry of the infinite cluster; see
Lemmas~\ref{lemma1-lemma-minmax3}--\ref{lemma3-lemma-minmax3} for the detailed
arguments. The resulting variational characterization of the effective
Hamiltonian therefore reflects directly the interaction between the geometry
of the continuum cluster and the degeneracy of the diffusion. Furthermore, the
coercive-entropic variational framework underlying the two-step min--max
construction is not specific to continuum percolation. Indeed, it applies
equally well in the classical setting of stationary ergodic random media on
$\R^d$ (for instance, in the framework of \cite{KRV06}) and is therefore of
independent interest.

The upper bound requires establishing sublinearity of the path integral
\[
V_G(x,\omega)
=
\int_{0\to x}\langle G,\d z\rangle
=
o(|x|)
\qquad\text{as }|x|\to\infty,
\quad \P_0\text{-a.s.}
\]
This step is technically quite involved. For this part of the proof we draw inspiration from ideas developed in~\cite{SS04,BB07} in the analysis of the Kipnis--Varadhan corrector \cite{KV86} for simple
random walks on discrete percolation clusters. At a conceptual level, both
problems require establishing sublinearity of an object obtained by
integrating a gradient-like field along the underlying random medium. Beyond
this analogy, however, there are some key conceptual differences between the two settings. 

Indeed, the quantity $V_G$ arises from the variational structure of a
nonlinear Hamilton--Jacobi--Bellman problem and is tied to a large-deviation
type effective theory rather than to a reversible Markov process. In
particular, the effective behavior is governed by an optimal tilt of the
underlying dynamics, a mechanism which is inherently different from the
reversible framework of \cite{KV86}. Consequently, the object
whose sublinearity must be established is of a fundamentally different nature
from the classical Kipnis--Varadhan corrector. Additional difficulties emerge from lack of uniform ellipticity available for simple random walks inside the percolation cluster, as in our setting the diffusion degenerates according to the distance to the random boundary of the
cluster, and also the continuum geometry of the infinite component renders the
combinatorial counting arguments available in the discrete setting much more subtle. As a
result, the proof instead relies on using the geometric and probabilistic
properties encoded in Assumptions~\ref{assump:est-erg}--\ref{assump:fkg} in a careful manner, together with the induced structure of the limiting gradients established in the variational part in Section \ref{sec-prop-minmax3}. Using these ingredients, we establish the required sublinearity in
Section~\ref{sec-proof-sublinear}.

Having established sublinearity, the upper bound is established in Section~\ref{sec-ub}. Here we use a perturbation and mollification argument and combine these with the techniques developed in the variational part in Section \ref{sec-equivalence} and the aforementioned sublinear bound. We refer to Section \ref{subsec proof ub lin f}- \ref{subsec ProofThmLDP} for details where the subtleties emerging from the geometric constraints imposed by the random cluster need to be circumveneted again.  

We close by mentioning that there is an orthogonal route to
\cite{KRV06} to study homogenization in the stationary ergodic setting based on subadditivity; see
\cite{So99,RT00,LS05,LS10,AT14} which allow very
degenerate situations, including the possibility that $a\equiv0$.

In summary, to the best of our knowledge, the present work provides the first homogenization result for a nonlinear Hamilton--Jacobi--Bellman equation on a genuine continuum percolation cluster. More broadly, it shows that the variational approach to stochastic homogenization can be extended to random geometries governed by nonstationary conditioned laws. The resulting effective theory retains quantitative information about the geometry of the continuum percolation cluster, the degeneracy of the diffusion, and the conditioning induced by the infinite component. In addition, the coercive-entropic variational framework developed here appears to be intrinsic to the stochastic control representation itself and is therefore of interest beyond the present continuum percolation setting.

\noindent{\bf Roadmap of the article.} Section~\ref{sec ergodic induced} establishes the ergodicity of the environment process seen from the particle and the controlled diffusion process under the induced shift and the conditioned measure $\P_0$. The proof of the lower bound is constitutes Section~\ref{sec-lb}, while the associated entropic variational analysis is developed in Section~\ref{sec-equivalence}. Assuming the sublinearity estimate, the upper bound is established in Section~\ref{sec-ub}. The proof of the required sublinearity property is then proved in Section~\ref{sec-proof-sublinear}.

\section{Ergodicity and Diffusion Processes on Continuum Percolation}\label{sec ergodic induced}

\subsection{Environment process on continuum percolation.}\label{su:ErgodicThmEnviron} 
In this section we will prove ergodic theorems in Proposition
\ref{thm-ergodic} - Proposition \ref{prop 1}  for the so-called {\it environment process} which, for homogenization of stationary ergodic random media (at least in the elliptic setting), goes back to the works of Kozlov \cite{K85} and Papanicolau-Varadhan \cite{PV81,PV82}.  In our context, this environment process is a diffusion taking values in the space of {\it conditioned} environments $\Omega_0$. As applications, we will subsequently obtain Corollary \ref{lemma-ergodic} and Corollary \ref{lemma2-ergodic}. 
\subsubsection{The environment seen from the particle.}
\label{sec:env-process}
Recall that the group $\{\tau_x\}_{\R^d}$ acts on $(\Omega,\mathcal G, \P)$ via
translations.  This action allows us to define, for any $u:\Omega\to \R$, its
{\it weak gradient} via
$$
(\nabla_i u)(\omega): = \lim_{\eps\to 0} \frac { u(\tau_{\eps e_i}(\omega)) - u(\omega)}{\eps}, \qquad i=1,\dots, d.
$$
Likewise, we also define the corresponding {\it divergence}. 
Now for $a:\Omega\to \R^{d\times d}$ satisfying \ref{f1}, we set 
\begin{equation}\label{sL}
\begin{aligned}
(\mathscr L^{\ssup b} u)(\omega) : =  \frac{1}{2}\mathrm{div}\big(a(\omega)\nabla u(\omega)\big)+ \langle b(\omega),  \nabla u(\omega)\rangle_a\qquad
\forall \omega \in \Omega_0. 
\end{aligned}
\end{equation} 
For a reasonable class of maps $b:\Omega\to \R^d$ (which does not depend on the
probability space $(\mathscr{X},\mathcal{F},P)$) 
and test functions $u$, $\mathscr L^{\ssup b}$ is the generator of a Markov process
taking values on $\Omega_0$ which can be defined as follows.  Set
$b(x,\omega):=b(\tau_x\omega)$ and let $X_t$ denote the $\R^d$-valued diffusion
solving the SDE
\begin{equation}
  \label{eq:SDE.beta}
	X_t=\int_0^t \sigma(X_s)\d {B}_s+ \int_0^t 
        (\mathrm{div}\;\! a )(X_s)\d s+\int_0^ta(X_s)b(X_s)\d
        s\quad\mbox{a.s.}\quad \forall t\geq 0, 
\end{equation}
with quenched law $P^{{b,\omega}}_0$ and generator 
\begin{equation}\label{eq:operator_def}
	(\mathcal L^{\ssup{b,\omega}}u)(x)= \frac{1}{2} 
  \mathrm{div}\big(a(x,\omega)\nabla u(x))+ \langle b(x,\omega), \nabla u(x)\rangle_a. 
\end{equation}
Then
\begin{equation}\label{def-env-pr}
\overline \omega_t: =\tau_{X_t} \omega
\end{equation}
is the $\Omega_0$-valued diffusion process with generator $\mathscr L^{(b)}$ defined in \eqref{sL}. We call $(\overline\omega_t)_{t\geq 0}$ the {\it environment process} with generator $\mathscr L^{\ssup b}$, and its 
 law with initial condition $\delta_\omega$ is denoted by $Q^{{b,\omega}}$. 
 
\subsubsection{Invariant density for the environment process.}\label{sec-b-phi}
Recall that $\P_0=\P(\cdot| \Omega_0)$. We write $L^1_+(\P_0)$ for the space of all non-negative and $\P_0$-integrable functions on $\Omega$. Any probability density $\phi\in L^1_+(\P_0)$ with $\int \phi \!\;\d\P_0=1$ is an invariant density with respect to $Q^{b,\omega}$ if
\begin{equation}\label{invariance}
\frac{1}{2}\mathrm{div}(a\nabla\phi)=\mathrm{div}(\phi (ab)), \quad\text{i.e., }\, (\mathscr L^{\ssup b})^\star \phi=0,\,\, 
\text{ in } \Omega_0, 
\end{equation}
with the generator $\mathscr L^{\ssup b}$ defined in \eqref{sL}. For any probability density $\phi$, we also
 set \begin{equation}\label{eq:L1a-def}
	L_a^1(\phi\d\P_0):=\bigg\{b:\Omega_0\to \R^d\text{ measurable: } \int \d\P_0 \phi \,\, \|b\|_a< \infty\bigg\}.
\end{equation} 
where we remind the reader from \eqref{eq:a-inner-prod} that 
\begin{equation}\label{eq:L1a-norm-def}
	\|b\|_{L_a^1}(\phi\d\P_0):=\int \d\P_0 \phi\,\,   \|b\|_a= \int
        \P_0(\d\omega) \phi(\omega) \sqrt{\big|\langle b(\omega), a(\omega)
          b(\omega)\rangle \big|}.  
\end{equation}
As usual, $L_a^1(\phi\d\P_0)$ can be turned into a Banach space with the norm
defined in \eqref{eq:L1a-norm-def} by taking the quotient w.r.t the subspace of
functions with zero $L^1_a$-norm.  Finally, for a suitable space $X$ (which
will be specified later on depending on the context), we will denote by
$\mathrm{Lip}=\mathrm{Lip}(X)$ the set of $1$-Lipschitz functions from
$X \to \R^d$. With this background, we define the class
\begin{equation}\label{class-E2}
 \begin{aligned}
   \mathcal E=\bigg\{ ( b, \phi)\in L^{1}_a(\phi \d\P_0)\times L^{1}_+(\P_0)
   &\colon \R^d\ni x\mapsto b(x,\omega) = b(\tau_x\omega) \in
   \mathrm{Lip}(\R^d)\  \forall \, \omega\in \Omega_0, \\ 
   &\int \phi \d\P_0=1,\,\, (\mathscr L^{\ssup b})^* \phi=0\bigg\}.
 \end{aligned}
\end{equation}

\subsection{The ergodic theorems and induced shift.}
\label{suu:Ergodic}
We are now ready to state the main result of this subsection:

\begin{prop}
\label{thm-ergodic}
Suppose that there exists $\phi$ such that  $(b,\phi)\in\mathcal{E}$. Let 
$\mathbb{Q}(\d\omega):=\phi(\omega)\P_0(\d\omega)$. If $\mathbb Q\ll \P_0$,
then the following three implications hold: 
\begin{itemize}
\item $\mathbb Q\sim \P_0$.
\item $\mathbb Q$ is ergodic with respect to the Markov process $Q^{ b,\omega}$.
\item There can be at most one such measure $\mathbb Q$.
\end{itemize}
\end{prop}
The proof of the above result will need a simple fact, for which we recall that $\Omega_0=\{\omega\in \Omega\colon 0\in \mathcal C_\infty(\omega)\}$, and also from \eqref{def-n} that $n(\omega,e)= \min \{k\in \N\colon k e \in \mathcal C_\infty(\omega)\}$. We then define the \textit{induced shift} $\sigma_e:\Omega_0\rightarrow\Omega_0$ by setting 
\begin{align}
\label{def-induced}
\sigma_e(\omega)=\tau_{n(\omega,e)}\omega.
\end{align}
Then $\sigma_e$ satisfies the following property:
\begin{prop}
	\label{prop 1}
	For every $e\in \Z^d$ with $|e|_1=1$, the induced shift $\sigma_e:\Omega_0\rightarrow\Omega_0$ is measure preserving and ergodic with respect to ${\P}_0$.
\end{prop}
We defer the proof to Section \ref{sec-appendix-ergodic} and complete the proof of Proposition \ref{thm-ergodic}.

\begin{proof}[{\bf Proof of Proposition \ref{thm-ergodic}:}] 
We first show that $\mathbb{Q}$ is equivalent to $\P_0$. Let $A:=\{\phi>0\}$. We need to show that $\P_0(A)=1$. Since $\phi$ is a density, we know that $\P_0(A)>0$. As $\phi \d\P_0$ is invariant with respect to the environmental process, we have \begin{align*}
	0=\int_{A^c}\phi \d\P_0=\int \1_{A^c}\phi \d\P_0&=\int E^{b,\omega}(\1_{A^c}(\tau_{X_1}\omega))\phi(\omega)\d\P_0\\
	&=\int_{A} E^{b,\omega}(\1_{A^c}(\tau_{X_1}\omega))\phi(\omega)\d\P_0.\\
\end{align*}

Thus, for $\P_0(|A)$-a.s $\omega$,
$E^{b,\omega}(\1_{A^c}(\tau_{X_1}\omega))=0$. Equivalently, for $\P_0(|A)$-a.s
$\omega$, $E^{b,\omega}(\1_{A}(\tau_{X_1}\omega))=1$. In particular, for
$\P_0(|A)$-a.s $\omega$, $\1_{A}(\tau_{X_1}\omega)=1$ $Q^{b,\omega}$-a.s. We
claim that this implies that $A$ is $\P_0$-a.s. invariant under the induced
shift, so $\P_0(A)\in\{0,1\}$. Since $\P_0(A)>0$, the equivalence between
$\mathbb{Q}$ and $\P_0$ would be complete. To show the claim, notice that for
$\omega$ as above, $\tau_x\omega\in A$ for almost all
$x\in {\mathcal C_\infty}(\omega)$. Indeed, if there is a subset $V$ of
${\mathcal C_\infty}(\omega)$ of positive Lebesgue measure satisfying
$\tau_x\omega\notin A$ for $x\in V$, then since the diffusion visits every set
of positive Lebesgue measure inside ${\mathcal C_\infty}(\omega)$, we would
have $P^{b,\omega}(X_1\in V)>0$, so that
$Q^{b,\omega}(\tau_{X_1}\omega\notin A)>0$, which would be a contradiction.
Thus, for $\P_0$-a.s. $\omega\in A$ and almost all
$x\in {\mathcal C_\infty}(\omega)$, we have $\tau_x\omega\in A$. In other
words,
\begin{equation*}
	\int_{A}\int_{\R^d}\1_{\{x:\tau_x\omega\notin A\}}\d x\d\P_0=0.
\end{equation*}
By Fubini's theorem, $\int_{\R^d}\int_{A}\1_{A^c}(\tau_x\omega)\d\P_0\d x=0$.
Hence, for almost all $x\in \R^d$,\begin{equation*}
	\int_{A}\1_{A^c}(\tau_x\omega)\d\P_0=\frac{1}{\P(0\in {\mathcal C_\infty})}\int_{\Omega}\1_{A}(\omega)\1_{A^c}(\tau_x\omega)\d\P=0.
\end{equation*}
By the continuity of the map
$\R^d\ni y\mapsto \1_{A}(\omega)\1_{A^c}(\tau_x\omega)\in L^1(\P) $, we deduce
that for all $x\in \R^d, \1_{A}(\omega)\1_{A^c}(\tau_x\omega)=0$ $\P_0$-a.s. In
particular, $\P_0$-a.s., for all $x\in \mathbb{Q}^d$ we have
$\1_{A}(\omega)\1_{A^c}(\tau_x\omega)=0$. By definition of the induced shift
(see \ref{def-induced}), $n(\omega,e)\in \mathbb{Q}^d$ and we conclude that
$\P_0$-a.s., $\1_{A}(\omega)\1_{A^c}(\sigma_e(\omega))=0$. In other words, $A$
is invariant under the induced shift $\P_0$-a.s., which proves that
$\mathbb Q\sim \P_0$. The other two assertions follow from standard arguments. 
\end{proof}

The following consequences of Proposition \ref{thm-ergodic} are laws of large numbers for the
trajectory of the diffusion and the associated functional. These results will be used heavily in the proof of the lower bound namely for proving Theorem \ref{thm lower bound} in Section \ref{sec-lb}. 

\begin{cor}\label{lemma-ergodic}
  Fix $( b,\phi)\in \mathcal E$. Then $\P_0\times P_0^{b,\omega}$-a.s.,
\begin{equation}
  \label{eq:lln}
  \lim_{t\to\infty} \frac {X_t} t=
  \E_0\left[\phi(\omega)\left(\frac{1}{2} \mathrm{div} \;\! a
      (\omega)+a(\omega)b(\omega)\right)\right].\end{equation}
\end{cor}
\begin{proof} 
By definition, $X_t$ satisfies 
\begin{equation}
 \label{eq:SDE.BC}
	X_t = \int_0^t \sigma(X_s)\,\d\mathcal{B}_s+ \int_0^t ( \frac12
        \mathrm{div} \;\! a  + ab)(X_s)\,\d s.
\end{equation}
Since $\sigma$ is bounded, the stochastic integral divided by $t$ goes to $0$ \
$\P_0{\times} P_{0}^{b,\omega}$-a.s. Moreover, Proposition \ref{thm-ergodic} yields 
\begin{equation}
  \label{eq:Drift.Aver}
	\lim_{t\to\infty}\frac{1}{t} \int_0^t  \Big( \frac12
        \mathrm{div} \;\! a  + ab\Big)(X_s)\,\d s 
  =\E_0\left[\Big(\frac{1}{2}\mathrm{div} \;\! a+ab\Big) \phi\right]\quad 
    \P_0{\times} P_{0}^{b,\omega}\text{-}a.s.
\end{equation}
This finishes the proof.
\end{proof}

The following immediate consequence of Proposition \ref{thm-ergodic} and
Corollary \ref{lemma-ergodic} will be used several times in the sequel:

\begin{cor}
\label{lemma2-ergodic}
Fix $(b,\phi)\in \mathcal{E}$. Then for $\P_0$-a.e.\ $\omega\in \Omega_0$ and
$P_0^{b,\omega}$-a.s. and in $L^1(P_0^{b,\omega})$, we have, uniformly on $[0,T]$, 
\begin{equation}\label{eq:ergodic-limits}
  \begin{aligned}
    &\lim_{\eps\to 0}\eps \int_0^{t/\eps}b(X_s,\omega)\d s=t\int \P_0(\d\omega)
    \phi(\omega) b(\omega), \\ 
    &\lim_{\eps\to 0}\eps \int_0^{t/\eps}L(X_s,b(X_s,\omega),\omega)\d s= t h
    (b,\phi),\quad \mbox{and} 
    \quad\lim_{\eps\to 0}\eps X_{t/\eps}=  t m(b,\phi), \qquad\mbox{where}\\
    &h(b,\phi):=
    \int \P_0(\d\omega) \phi(\omega) L(b(\omega),\omega), \quad\mbox{and}\quad  m(b,\phi):= \int\P_0(\d\omega)
    \phi(\omega) \Big(\frac{1}{2} \mathrm{div}(a(\omega))+b(\omega)\Big),
  \end{aligned}
\end{equation}

\end{cor}\qed

\subsubsection{Proof of Proposition \ref{prop 1}.}\label{sec-appendix-ergodic} 
 Proposition \ref{prop 1} is a consequence of the following known result from ergodic theory (see e.g. \cite{P89,BB07}). 
\begin{lemma}
\label{lemma 3.3}
Let $(X,\mathscr F,\mu)$ be a probability space and let $T:X\to X$ be
invertible, measure preserving and ergodic with respect to $\mu$. Let
$A\in\mathscr F$ with $\mu(A)>0$. If $n:A\rightarrow \N\cup\{\infty\}$ is
defined by
$$
	n(x)=\min\{k>0:T^k(x)\in A\}
$$
and $S:A\rightarrow A$ by $S(x)=T^{n(x)}(x)$ for $x\in A$, then $S$ is measure
preserving and ergodic with respect to $\mu(\cdot|A)$ and almost surely
invertible with respect to the same measure.
\end{lemma}
\begin{proof}
We first prove that $S$ is measure preserving. By the Poincar\'{e} Theorem,
$n(x)<\infty$ almost surely. For any $j\geq 1$ we define
$A_j=\{x\in A:n(x)=j\}$. By definition, the $A_j$ are disjoint and as
$n(x)<\infty$ almost surely, $\mu(A\setminus\cup_{j\geq 1}A_j)=0$. As the
restriction of $S$ to $A_j$ is $T^j$ and since $T^j$ is measure preserving, $S$
is measure preserving on $A_j$. We claim that $S(A_i)\cap
S(A_j)=\emptyset$. This, together with the fact that $S$ is measure preserving
on $A_j$, proves that $S$ is measure preserving on the disjoint union
$\cup_{j\geq 1}A_j$ and therefore on $A$.

Thus, we only owe the claim $S(A_i)\cap S(A_j)=\emptyset$. We assume that there
exists $x\in S(A_i)\cap S(A_j)$ for $1\leq i<j$. This requires the existence of
$y,z\in A$ with $n(y)=i$, $n(z)=j$ and $x=T^i(y)=T^j(z)$. As $T$ is invertible,
$y=T^{j-i}(z)$. Thus, $n(z)\leq j-i<j$, which is a contradiction to $n(z)=j$
and the desired claim follows.

Next, we note that $T$ is invertible. Thus, $S$ is almost surely invertible, as
the intersection $S^{-1}(\{x\})\cap \{S\text{ is well defined}\}$ is a
one-point set.

We finally want to show that $S$ is ergodic. Let $B\in \mathscr F$ such that
$B\subseteq A$ is $S$-invariant. Then if $x\in B$ and $n\geq 1$, it follows
that $S^n(x)\notin A\setminus B$. This implies that for any $x\in B$ and
$k\geq 1$, if $T^k(x)\in B$, then $T^k(x)\notin A\setminus B$. We conclude that
$C=\cup_{k\geq 1}T^k(B)$ is $T$-invariant and
$B\subseteq C\subseteq (X\setminus A)\cup B$. In particular,
$\mu(B)\leq \mu(C)\leq 1+\mu(B)-\mu(A)$. Therefore, ergodicity of $T$ implies
$\mu(C)\in\{0,1\}$, which forces $\mu(B)\in\{0,\mu(A)\}$ and thus, the
ergodicity of $S$ with respect to $\mu(\cdot|A)$.
\end{proof}

\noindent\textbf{Proof of Proposition \ref{prop 1}.}
% As the balls around elements of $H_{\lambda}$ have radius
% $1/2$, % and $n(\omega)\geq 1$,
% the events $\{x\in G(H_\lambda(\omega),1/2)\}$ and
% $\{x+e\in G(H_\lambda(\omega),1/2)\}$ are i.i.d.
The shift $\tau_e$ is invertible, measure preserving and ergodic with respect
to ${\P}$. It follows from Lemma \ref{lemma 3.3} that the induced shift
$\sigma_e$ is ${\P}_0$-preserving, almost surely invertible and ergodic with
respect to ${\P}_0$.  \qed

\section{Lower Bounds.}
\label{sec-lb}

The goal of this section is to prove Theorem \ref{thm lower bound} in Section \ref{subsec proof thm} showing the lower bound of Theorem \ref{thm}. Its proof will be based on Lemmas \ref{lemma:control_restriction}- \ref{lemma2} established in Section \ref{subsec proof thm}. First we need some preliminaries. 
 \subsection{Preliminaries}\label{subsec prelim lb} 
We recall from Section \ref{sec-main-results} the definition of the space $\mathbf{C}_T$ of
progressively measurable functions
$c: [0,T]\times \mathscr X \mapsto
\R^d$ with $ E^P[\int_0^T |c(s)|^2 \d s]<\infty$. For every $a$ satisfying \ref{f1} and $c\in \mathcal C_T$ and for every fixed $\omega\in \Omega_0$ we have the controlled SDE $X_t= x+ \int_0^t \sigma(X_s) \d B_s+ \int_0^t (\mathrm{div}a)(X_s) \d s + \int_0^t a(X_s) c(s) \d s$ which admits a unique strong solution (\cite[Theorem 2.1]{BMM26}).

For any fixed $\omega\in \Omega_0$, the law of the diffusion $X_\cdot$ starting at $x\in \R^d$ is denoted by $P^{c,\omega}_x$. Recall from \ref{f1} that for every $x\in \R^d$ and $\omega\in \Omega_0$, we have $\sigma(x,\omega)=\sigma(\tau_x\omega)$, $a(x,\omega)=a(\tau_x\omega)$ and $(\mathrm{div} \, a)(x,\omega)= (\mathrm{div} \,a)(\tau_x\omega)$. Then for any $y\in \R^d$, the map $\hat\tau_y: C([0,T],\R^d) \to C([0,T],\R^d)$ with $X_\cdot \mapsto X_\cdot + y$ satisfies 
\begin{equation}\label{shift1}
P^{c,\omega}_{x+y}= P^{c,\tau_y\omega}_x \, \hat\tau_y^{-1}\qquad \forall\,\,\omega\in \Omega_0, \,\, x,y \in \R^d,\,\, c\in \mathbf C_T.
\end{equation}

Next, we recall %such that \eqref{eq:prog-meas-L2} holds.
from \eqref{eq:cost-fun-def}
 
\begin{equation}  
\begin{aligned}
&u(t,x,\omega):=\sup_{c\in \mathbf{C}_T} J^\omega_f(t,x,c)=	\sup_{c\in \mathbf{C}_T}E^{P_x^{c,\omega}}\bigg[f(X_t)-\int_0^t L(X_s,c(s))\d s\bigg ] \label{eq:visc-sol-var-form 2}
\end{aligned}
\end{equation}

 We note that \eqref{shift1} implies the following simple identity. For any $y\in \R^d$, let $f^y(\cdot)= f(\cdot + y)$ be the $y$-translate of $f(\cdot)$ and let $u^y(t,x,\omega)$ be defined as in $u(t,x,\omega)$ above with initial condition $f^y(\cdot)$ replacing $f(\cdot)$. Since $L(y,q,\omega)=L(q,\tau_y\omega)$ (recall \ref{f3}), then  
\begin{equation}\label{shift2}
\begin{aligned}
u^y(t,x,\tau_y\omega)&= \sup_{c\in \mathbf C_T} E^{P_x^{c,\tau_y\omega}}\bigg[f(y+X_t)-\int_0^t L(X_s,c(s),\tau_y \omega)\d s\bigg ]\\
&=\sup_{c\in \mathbf C_T} E^{P_x^{c,\tau_y\omega}}\bigg[f(y+X_t)-\int_0^t L(c(s),\tau_{X_s}\tau_y\omega)\d s\bigg ] \\
&= \sup_{c\in \mathbf C_T} E^{P_{x+y}^{c,\omega}}\bigg[f(X_t)-\int_0^t L(c(s),\tau_{X_s}\omega)\d s\bigg ] \\
&= \sup_{c\in \mathbf C_T} E^{P_{x+y}^{c,\omega}}\bigg[f(X_t)-\int_0^t L(X_s, c(s),\omega)\d s\bigg ] 
= u(t,x+y,\omega)
\end{aligned}
\end{equation}
where we used \eqref{shift1} in the third identity above. Next we recall the rescaled version of $u(t,x,\omega)$ from \eqref{eq:u-eps-def}: For any $\eps>0$,  $f_\eps(x):=\frac{1}{\eps} f(\eps x)$ and for any $\omega\in \Omega_0$ 
\begin{equation}\label{eq:u-eps}
\begin{aligned}
%v_\eps(t,x,\omega):=\sup_{c\in \mathbf{C}_T}J_{f_{\eps}}(t,x,c), \qquad\mbox{and}\quad 
	u_{\eps}(t,x,\omega) &:= \sup_{c\in \mathbf{C}_T} \eps J^\omega_{f_\eps}\bigg(\frac t\eps, \frac x \eps, c\bigg)
    %\eps v_{\eps}\bigg(\frac{t}{\eps},\frac{x}{\eps},\omega\bigg). 
    =\sup_{c\in \mathbf C_T} E^{P^{c,\omega}_{x/\eps}}\bigg[f(\eps X_{t/\eps})-\eps\int_0^{t/\eps}L(X_s,c(s))\d s\bigg]
\end{aligned}
\end{equation}
An alternative representation of the above expression is given by 

\begin{equation}\label{eq:u-eps-2} u_\eps(t,x,\omega)=\sup_{c\in \mathbf
    C_T} E^{P_{x}^{\eps,c,\omega}}\bigg[f(
  Y_{t})-\int_0^{t}L\bigg(\frac{1}{\eps}Y_s,c\big(\frac{s}{\eps}\big)\bigg)\d
  s\bigg],
\end{equation}
where $P_x^{\eps,c,\omega}$ is the law of the diffusion satisfying the rescaled SDE 
\begin{equation}
  \label{eq:sde-rescaled}
  Y_t=x + \sqrt{\eps} \int_0^t\sigma\left(\frac1\eps Y_s\right) \d{B}_s 
  +\int_0^t(\mathrm{div}\;\! a )\left(\frac1\eps Y_s\right)\d
  s+\int_0^ta\left(\frac1\eps Y_s\right)c\left(\frac{s}{\eps}\right)\d s. 
\end{equation}

\subsection{A priori estimates}
\label{subsec proof thm}

In this section, we use constants $C,C''$ independent of $\omega,t,\eps$ that
may change from line to line. Here using Lemma \ref{lemma:control_restriction}- Lemma \ref{lemma2} below, in Theorem \ref{thm lower bound} we will prove the lower bound of Theorem \ref{thm}.   

 The first lemma below shows that it is sufficient to restrict the control set
$\mathbf C_T$ to a smaller set $\mathbf C^*_T$, which follows from
comparing supremum property with the special choice $c \equiv 0$. 

\begin{lemma}
\label{lemma:control_restriction}
Assume \ref{assump:est-erg}-\ref{assump:inf-comp}, $\ref{f1}, \ref{f2}$- \ref{f4}. Then we can replace the supremum of
$c\in \mathbf C_T$ in \eqref{eq:u-eps-2} by a supremum over
$c\in \mathbf C^*_T\subset \mathbf C_T$ of functions satisfying the following:
for each $\delta>0$, there exists a constant $C_\delta$ depending only on
$\delta$ and the constants $\alpha, \alpha^\prime$ appearing in \ref{f2} and
\ref{f4} such that for all $\omega \in \Omega_0$, 
\begin{equation}
  \label{eq:control_restriction_eq1} \sup_{x\in
    \eps \mathcal{C}_\infty}\eps
  E^{P_{x/\eps}^{c,\omega}}\bigg[ \int_0^{t/\eps} 
  L(X_s,c(s)) \d s  \bigg]\leq C_\delta(t+\sqrt{\eps
    t})+2\alpha\delta.
\end{equation}
In particular, for all $c\in \mathbf C^*_T$,
\begin{equation}
\label{eq:control_restriction_eq2} 
\sup_{x \in
    \eps \mathcal{C}_\infty}\eps E^{P_{x/\eps}^{c,\omega}}\bigg[\int_0^{t/\eps} \big\|
c(s)\big\|_a ^{\alpha'}\d s\bigg]\leq C_\delta(t+\sqrt{\eps t})+2\alpha\delta.
\end{equation}
\end{lemma}
\begin{proof} 
First, we remark that \eqref{eq:control_restriction_eq1} implies
\eqref{eq:control_restriction_eq2} by using \eqref{eq:H1-L} giving 
\begin{align}
 \label{eq:ineq_c_norm2} 
  \eps E^{P_{x/\eps}^{c,\omega}}\bigg[\int_0^{t/\eps}\|c(s)\|_a^{\alpha'} 
  \, \d s\bigg]&\leq \frac1{c_{10}}\bigg(\eps
  E^{P_{x/\eps}^{c,\omega}}\bigg[\int_0^{t/\eps} L(X_s,c(s)) 
 \: \d s\bigg]+c_{11}t\bigg)=: \frac1{c_{10}} \Theta(t).
\end{align}
Thus,  we only need to prove \eqref{eq:control_restriction_eq1}. 

 The general strategy is the following: Denoting the expectation under
$\sup_{c\in \mathbf C_T}$ in \eqref{eq:u-eps} by $J_\eps(t,x,\omega, c)$, the
choice $c\equiv 0$ gives the estimate
$u_\eps(t,x,\omega) \geq J_\eps(t,x,\omega, 0)$. Hence, it is sufficient to
shrink the set $\mathbf C_T$ to a set $\mathbf C_T^*$ containing all functions
in $\mathbf C_T$ satisfying $J_\eps(t,x,\omega, c) \geq J_\eps(t,x,\omega,
0)$. This restriction will lead to condition
\eqref{eq:control_restriction_eq1}.

To provide a lower bound on $J_\eps(t,x,\omega,0)$ we next derive a control for
the expectation of $\eps X_{t/\eps}-x$ under the law $P_{x/\eps}^{c,\omega}$,
corresponding to  the diffusion
\begin{equation}
  \label{eq:diff_repr} 
  \eps X_{t/\eps} = x +
  \eps\int_0^{t/\eps} \sigma(X_s)\d{B}_s+ \eps\int_0^{t/\eps} (\mathrm{div}\;\!
  a)(X_s)\d s+\eps\int_0^{t/\eps} a(X_s)c(s)\d s. 
\end{equation}
We will now use the upper bound \eqref{eq:ellip-bounds} from \ref{f1} to deduce that 
\begin{equation*}
 	|a(X_s)c(s)|\leq C|\sigma(X_s)c(s)|=C\|c(s)\|_a.
\end{equation*}
Note that the norm above implicitly depends on $X_s$.  Also using \ref{f1} we
have uniformly $|\mathrm{div}\;\! a|\leq C^\prime$ for some
$C^\prime<\infty$. Using these two bounds,
\begin{equation*}
  E^{P_{x/\eps}^{c,\omega}}\big[|\eps X_{t/\eps}{-}x| \big]\leq \eps
  E^{P_{x/\eps}^{c,\omega}}\bigg[\bigg(\int_0^{t/\eps} \sigma(X_s)\d
  {B}_s\bigg)^2\bigg]^{1/2}+C^\prime t+\eps C
  E^{P_{x/\eps}^{c,\omega}}\bigg[\int_0^{t/\eps}\|c(s)\|_a\d s\bigg]. 
\end{equation*}
Using It\^o isometry, followed by employing the upper bound from
\eqref{eq:ellip-bounds}, we have  
\begin{equation}\label{Ito}
E^{P_{x/\eps}^{c,\omega}} \bigg[\bigg(\int_0^{t/\eps} \sigma(X_s)\d
{B}_s\bigg)^2\bigg] \leq C^{\prime\prime} t/\eps, \quad\mbox{hence,}
\end{equation}
 \begin{equation}
\label{eq:norm_X-t-bound2}
E^{P_{x/\eps}^{c,\omega}} \big[ |\eps X_{t/\eps}{-}x| \big]  \leq  
C\big(t+ \sqrt{\eps t} \big)+\eps
CE^{P_{x/\eps}^{c,\omega}}\bigg[\int_0^{t/\eps}\|c(s)\|_a\d s\bigg]. 
\end{equation}
	
Using the formula \eqref{eq:u-eps} with $c\equiv 0$, by 
\eqref{eq:H6} and \eqref{eq:norm_X-t-bound2}, for any $\delta>0$ we obtain the
lower  bound 
\begin{equation}\label{eq:u-eps_trivial_lower_bound}
  \begin{aligned}
   J_\eps(t,x,\omega,0)&= f(x) + 
   E^{P_{x/\eps}^{0,\omega}}\bigg[f(\eps
    X_{t/\eps})-f(x) -\eps \int_0^{t/\eps} \!\!L\left(X_s,0\right)\d s\bigg]\\
    &\geq f(x) -K_\delta E^{P_x^{\eps,0,\omega}} \big[|X_t{-}x|\big] - \delta - \eps
    c_{13}\frac t\eps\\ 
    &\geq f(x) -K_{\delta} C\big( t + \sqrt{\eps t}\big)-c_{13}t-\delta.
  \end{aligned}
\end{equation}
	
Finally, we use $-J_\eps(t,x,\omega,c) \leq - J_\eps(t,x,\omega,0)$ to derive
\eqref{eq:control_restriction_eq1}. Inserting \eqref{eq:u-eps} and using the
bound on $J_\eps(t,x,\omega,0)$ leads to  
\begin{align*}
  \eps E^{P_{x/\eps}^{c,\omega}}\bigg[\int_0^{t/\eps}
  L\big(X_s,c(s)\big)\d s \bigg] 
  &\leq  E^{P_{x/\eps}^{c,\omega}}\big[ f(\eps X_{t/\eps}){-}f(x) \big] +
  C_{\delta}(\sqrt{\eps t}+ t)+ c_{13}t+\delta,\\
  &\leq  K_\delta E^{P_{x/\eps}^{c,\omega}}\Big[ \big|\eps X_{t/\eps}{-}x\big|) \Big] +
  C K_{\delta}(\sqrt{\eps t}+ t)+ c_{13}t+2\delta
\\
& \leq  \eps C E^{P_{x/\eps}^{c,\omega}}\bigg[\int_0^{t/\eps}\!\|c\|_a \:\d s \bigg]+
  C_{\delta}(\sqrt{\eps t}+ t) +2\delta ,
\end{align*}
where we used \eqref{eq:H6} for the second estimate and
\eqref{eq:norm_X-t-bound2} for the last. By H\"{o}lder's inequality and  \eqref{eq:ineq_c_norm2}, 
\begin{align}
 \label{eq:ineq_c_norm}
  \eps E^{P_{x/\eps}^{c,\omega}}\bigg[\int_0^{t/\eps}\|c(s)\|_a\d s\bigg]&\leq
  t^{1/\alpha}\bigg(\eps
  E^{P_{x/\eps}^{c,\omega}}\bigg[\int_0^{t/\eps}\|c(s)\|_a^{\alpha'}\d
  s\bigg]\bigg)^{1/\alpha'}  \leq \frac{t^{1/\alpha}}{c_{10}^{1/\alpha'}} 
 \Theta(t)^{1/\alpha'}, 
\end{align}
 where $\Theta(t)$ is defined in \eqref{eq:ineq_c_norm2} in terms of $\int_0^{t/\eps}
  L\big(X_s,c(s)\big)\d s$. Together, we have 
\[
  \Theta(t) \leq C t^{1/\alpha} \Theta(t)^{1/\alpha'} + C'_\delta\big(
  \sqrt{\eps t} + t\big) + 2 \delta \leq \frac12\Theta(t) + C't^1 +
  C'_\delta\big( \sqrt{\eps t} + t\big) + 2 \delta ,
\]
where we used Young's inequality with $1/\alpha + 1/\alpha'=1$. Thus,
$\Theta(t)\leq 2(C'_\delta{+}C)\big( \sqrt{\eps t} + t\big) + 4
\delta$. Replacing $\delta$ by $\delta/4$ gives the assertiion
\eqref{eq:control_restriction_eq1} with $C_\delta:= 2(C'_{\delta/4}{+}C)$.

\end{proof}

\begin{lemma}
\label{lemma:lower-bound-lemma1}
Assume \ref{assump:est-erg}-\ref{assump:inf-comp}, \ref{f1}, \ref{f2} and \ref{f4} and fix $(b,\phi)\in \mathcal{E}$ (recall the definition of $\mathcal E$ from \eqref{class-E2}). Then, for 
all $\eta>0$ there exists $N_\eta\subset \Omega_0$ with $\P_0(N_\eta) 
 > 1-\eta$ such that for all $T>0$ we have  
\begin{equation*}
 \liminf_{\eps\to 0}\inf_{\omega\in N_\eta}\inf_{0\leq t\leq T} \Big( 
  u_\eps(t,0,\omega)- \big(f(m(b,\phi) t)-th(b,\phi) \big) \Big)\geq 0,
\end{equation*}
where $h(b,\phi)$ and $m(b,\phi)$ are defined in \eqref{eq:ergodic-limits}.
\end{lemma}
\begin{proof}
We fix $(b,\phi)\in \mathcal E$. Then from \eqref{eq:u-eps}, for any $\eps>0$, $t \in [0,T]$ and for $\P_0$a-e. $\omega \in \Omega_0$, 
\begin{equation}\label{lemma4.2 eq1}
u_\eps(t,0,\omega)= \sup_{c \in \mathbf C_T} E^{P^{c,\omega}_0}\bigg[ f\big(\eps X_{t/\eps}\big) - \eps \int_0^{t/\eps} L(X_s, c(s)) \d s\bigg]
\end{equation}
Now the supremum over $c\in \mathbf C_T$ is chosen for a percolation configuration $\omega \in \Omega_0$. So we can and will choose $c$ to depend on $\omega$. Moreover, the dependence on $s$ will occur only through $X_s$, which is progressively measurable.\footnote{We recall that any $c\in \mathbf C_T$ which is adapted to the natural filtration of the associated process $(X_t)_t$ is called a {\it feedback control}. Moreover, a process $c\in \mathbf C_T$ which can be written in the form $c(s)=\tilde c(s,X_s)$ for a suitable measurable map $\tilde c :[0,T] \times \R^d \to \R^d$ is called a {\it Markovian control}. Note that any Markovian control is also a feedback control, see \cite[Remark 3.1]{T13}.} Indeed, given any $(b,\phi) \in \mathcal E$, we consider the 
Markovian control 
\begin{equation}
 \label{Markovian control2}
c(s,\omega)=  \hat c_b(s,\omega) :=  
 b(\tau_{X_s}\omega).
\end{equation}
 To see that $ \hat c_b$ lies in the control set $\mathbf C_T$
it remains to check the fact that, for each $T>0$ and $\omega$ there is a constant $C(\omega,T)<\infty$ such that
\begin{equation}
  \label{starstar}
E^{P^{b,\omega}_0}\bigg[\int_0^T |b(\tau_{X_s}\omega)|^2 \d s \bigg]= C(\omega,T)<\infty.
\end{equation}
To verify the above bound, we recall from \eqref{eq:SDE.beta} that $P^{b,\omega}_0$ refers to the law of the diffusion 
\begin{align}
    \label{eq:Xt.diffusion}
X_t=\int_0^t \sigma(X_s) \d B_s + \int_0^t \mathrm{div}(a(X_s)) \d s + \int_0^t a(X_s) b(X_s) \d s. 
\end{align}
Since $x \mapsto b(x,\omega)$ is Lipschitz (recall \eqref{class-E2}), we have $|b(x,\omega)| \leq \mathrm{Lip}(b)|x| + |b(0,\omega)|$, so \eqref{starstar} follows once we show that for every $T>0$ and $\omega$, we have  
\begin{align}
 \label{eq:int.As.ds}
  \int_0^T A(s) \, \d s = C(\omega,T)<\infty, \quad 
  \text{ where } A(t):= E^{P^{b,\omega}_0} \big[ |X_t|^2 \big]. 
\end{align}
For this, we keep $b$ and $\omega$ fixed and apply It\^o's formula 
to \eqref{eq:Xt.diffusion} and find 
\[
|X_t|^2 = \int_0^t X_s {\cdot} \sigma(X_s) \d B_s + \int_0^t 
\Big( X_s{\cdot}\big(\mathrm{div}(a(X_s)) + a(X_s) b(X_s)\big) + \frac12 \mathrm{trace}(a(X_s))\Big) \d s 
\]
Taking the expectation and exploiting the independence of 
$\d B_s$ we obtain
\begin{align*}
    A(t)&= 0 + \int_0^t E^{P^{b,\omega}_0}\bigg[ \Big( X_s{\cdot}(\big(\mathrm{div}(a(X_s)) + a(X_s) b(X_s)\big) + \frac12 \mathrm{trace}(a(X_s))\Big) \bigg] \d s   
    \\
    &\leq \int_0^t   E^{P^{b,\omega}_0}\bigg[ \sup|\mathrm{div}(a)|\,|X_s| + \sup|a|\, \mathrm{Lip}(b)\,|X_s|^2 + \frac12 \sup \mathrm{trace}(a) \bigg] \d s
    \\ 
    & \leq \int_0^t \Big( C_0 + C_1 A(s) \Big) \d s,
\end{align*}
where $C_0$ and $ C_1 $ only depend on $a$, $b$, and $\omega$. Now, the Gr\"onwall lemma  yields the estimate 
\[
A(t) \leq \frac{C_0}{C_1} \Big( \mathrm{e}^{C_1 t} - 1 \Big) \quad \text{ for }t\in [0,T].
\]
Using this, it is immediate that \eqref{eq:int.As.ds} holds with $C(\omega,T) = \frac{C_0}{C_1^2} \big( \mathrm{e}^{C_1T} - 
1-C_1T\big)$.  

Thus the supremum in \eqref{lemma4.2 eq1} can be restricted to controls of the form \eqref{Markovian control2}, leading to a lower bound for each $\eps>0$, $(b,\phi)\in \mathcal{E}$, $t \in [0,T]$ and $\P_0$-a.e. $\omega \in \Omega_0$:
\begin{equation}\label{lemma4.2 eq2}
	 u_\eps(t,0,\omega)\geq E^{P_{0}^{b,\omega}}\bigg[f(\eps
         X_{t/\eps})-\eps\int_0^{t/\eps}L(X_s,b(\tau_{X_s}\omega))\d s\bigg]. 
\end{equation} 
We now recall \eqref{eq:ergodic-limits} from Corollary \ref{lemma2-ergodic} which says that, for $\P_0$-a.e. $\omega\in\Omega_0$ and $P^{b,\omega}_0$ almost surely and in $L^1(P^{b,\omega}_0)$ we have,  uniformly in $t\in [0,T]$, 
\begin{equation}\label{lemma4.2 eq3}
\lim_{\eps\to 0} \eps X_{t/\eps}= t m(b,\phi)\quad\mbox{and}\quad 
\lim_{\eps\to 0} \eps \int_0^{t/\eps} L(X_s,b(\tau_{X_s}\omega)) \d s = t h(b,\phi).
\end{equation}
 Defining the auxiliary function 
\[
g_\eps(\omega):= \sup_{0\leq t \leq T} 
   E^{P^{b,\omega}_0}  \bigg[ 
   \Big| \eps X_{t/\eps} - t m(b,\phi)\Big|\, +\, 
    \Big| \eps \int_0^{t/\eps} L(X_s,b(\tau_{X_s}\omega)) \d s - t h(b,\phi)\Big| \bigg]
\]
we have $g_\eps(\omega) \to 0$ as $\eps \to 0$ for $P^{b,\omega}_0$ almost all $\omega$. 
Applying  Egoroff's theorem  to $ g_\eps$ we find for all $\eta>0$ a subset $N_\eta \subset \Omega_0$ with $\P_0(N_\eta)\geq 1- \eta$ such that 
\begin{equation}\label{lemma4.2 eq4}
\begin{aligned}
&\lim_{\eps\to 0} \sup_{\omega \in N_\eta} \sup_{0\leq t \leq T} E^{P^{b,\omega}_0}
\big[\big| \eps X_{t/\eps} - t m(b,\phi)\big|\big]=0, \\
&\lim_{\eps\to 0} \sup_{\omega \in N_\eta} \sup_{0\leq t \leq T} E^{P^{b,\omega}_0}
\bigg[\bigg| \eps \int_0^{t/\eps} L(X_s,b(\tau_{X_s}\omega)) \d s - t h(b,\phi)\bigg|\bigg]=0.
\end{aligned}
\end{equation}
Since $f$ is assumed to be uniformly continuous in \ref{f4}, combining \eqref{lemma4.2 eq2} and \eqref{lemma4.2 eq4} proves the lemma. 
\end{proof}

The following result extends Lemma \ref{lemma:lower-bound-lemma1}. 
For any $R>0$, we can consider the family of functions 
$\{f^{y},|y|\leq R\}$, where $f^{y}(x):=f(x+y)$. Notice that 
$f^{y}$ is uniformly continuous with the same constants $K_\delta$ 
as in \ref{f4}.

\begin{lemma}
\label{lemma:lower-bound-lemma1-extended}
Assume \ref{assump:est-erg}-\ref{assump:inf-comp}, \ref{f1}, \ref{f2} and \ref{f4} and for any $y \in \R^d$ let 
$u^y_\eps$ be the representation \eqref{eq:u-eps} with initial condition 
$f^y(\cdot)= f(y + \cdot)$. 
%This is a viscosity solution of \eqref{eq-HJB} associated to the initial condition $f^{(y)}$. 
 Fix $(b,\phi)\in \mathcal{E}$. Then  for all $\eta>0$, there exists a subset 
 $N_\eta (b,\phi)  \subset \Omega_0$ with $\P_0(N_\eta)\geq 1-\eta$
such that for all $R>0$,  $T>0$, and $\omega \in N_\eta (b,\phi) $  
we have  
\begin{equation}\label{eq Lemma 4.3}
  \liminf_{\eps\to 0} \inf_{0\leq t\leq T}\inf_{  y \in B_R(0) } 
    \Big( u_\eps^{y}(t,0,\omega)-\big(f(y+m(b,\phi) t) - 
    t h(b,\phi) \big) \Big)  \geq  0. 
\end{equation}

\end{lemma}
\begin{proof}

%\noindent{\bf Proof of \eqref{eq Lemma 4.3}:}	
We proceed as in the proof of Lemma \ref{lemma:lower-bound-lemma1}. Indeed, similar to \eqref{lemma4.2 eq2}, we have for each $\eps>0$, $(b,\phi)\in \mathcal{E}$, $t \in [0,T]$, $R>0$, $y \in B_R(0)$ and $\P_0$-a.e. $\omega\in \Omega_0$, 
 \begin{equation}\label{lemma4.2 eq5}
	 u^y_\eps(t,0,\omega)\geq E^{P_{0}^{b,\omega}}\bigg[f^y(\eps X_{t/\eps})-\eps\int_0^{t/\eps}L(X_s,b(\tau_{X_s}\omega))\d s\bigg]=: A_\eps(y,t,\omega).
\end{equation} 
We now recall from \ref{f4} that for every $y\in \R^d$, $f^{y}$ is uniformly continuous with the same constant $K_\delta=K_\delta(f)$ (i.e., for every $\delta>0$, there exists $K_\delta>0$ such that for all $y \in \R^d$ and $x_1, x_2 \in \R^d$, $|f^y(x_1)- f^y(x_2)| =|f(y+x_1)- f(y+x_2)| \leq K_\delta|x_1- x_2| + \delta$).  With this, define the auxiliary function
\[
\widetilde g_\eps(y,\omega):= \sup_{0,t\leq T} \big| 
A_\eps(y,t,\omega) - (f^y(t m(b,\phi)) - t h(b,\phi)) \big|
\]
As in the proof of Lemma \ref{lemma:lower-bound-lemma1} (cf.\ \eqref{lemma4.2 eq4}), we have $ g_\eps (y,\omega) \to 0$ as $\eps \to 0$ for $\P_0$ almost all $\omega$ and all $ y \in B_R(0)$. Now, we use that $f^y$ is defined with $ f $ via shifting by $y$. 
Hence we have 
\[
\big| f^{y_1}(x) - f^{y_2} (x) \big| = \big|f(x{+}y_1) - f(x{+}y_2)\big| \leq \delta + K_\delta 
\big| y_1-y_2| .
\]
Inserting this into the definition of $A_\eps$ we find $ |A_\eps(y_1,t,\omega) - A_\eps(y_2,t,\omega)| \leq \delta + K_\delta |y_1{-}y_2|$, and hence 
\[
\big| \widetilde g_\eps (y_1,\omega) - \widetilde g_\eps(y_2, 
 \omega)\big| \leq \delta + K_\delta \big| y_1 - y_2 \big|. 
\]
This uniform continuity of $y\mapsto g_\eps( y,\omega)$ and the pointwise convergence $g_\eps(y,\omega)\to 0$ as $\eps \to 0$ implies 
\[
\widehat g_\eps(\omega):= \sup_{y \in B_R(0)} \widetilde g_\eps( y, \omega)
\ \longrightarrow \ 0 \quad \text{as } \eps \to 0 
\quad \text{ for $\P_0$ almost all } \omega \in \Omega_0. 
\]
Applying Egoroff's theorem to the family $ \widehat g_\eps$ we find, for 
all $\eta>0$, a subset  $N_\eta \subset \Omega_0$ with $\P_0(N_\eta) 
\geq 1- \eta$ such that for any $R,T>0$ and $(b,\phi)\in \mathcal E$, 
\begin{equation}
 \label{lemma4.2 eq7}
  \lim_{\eps\to 0} \sup_{\omega \in N_\eta} \widehat g_\eps(\omega) = 
  \lim_{\eps\to 0} \sup_{\omega \in N_\eta} \sup_{y\in B_R(0) } 
   \sup_{t\in [0,T]} \big| A_\eps(y,t,\omega) - 
 (f^y(t m(b,\phi)) - t h(b,\phi))\big| =0.
\end{equation}
Combining \eqref{lemma4.2 eq5} and \eqref{lemma4.2 eq7} proves \eqref{eq Lemma 4.3}. 
\end{proof}

\begin{lemma}[$\uu$ and effective Hamiltonian]
\label{lemma2}
Assume \ref{assump:est-erg}-\ref{assump:inf-comp}, \ref{f1}, \ref{f2}-\ref{f4} and recall from \eqref{eq:lax-oleinik} that $u_{\hom}(t,x)=\sup_{y\in \R^d}(f(y)-t\, 
  \mathcal{I}( \frac{y{-}x}{t}))$ where $\mathcal{I}(y):=\sup_{\theta\in \R^d} [\langle \theta,y\rangle-
 \overline{H}(\theta)]$ and $\overline{H}(\theta)	
:=\sup_{(b,\phi)\in \mathcal{E}}[\langle
     \theta,m(b,\phi)\rangle-h(b,\phi)]$ with $h(b,\phi)$ and $m(b,\phi)$ defined in \eqref{eq:ergodic-limits}. Then the following hold: (A) 
\begin{equation}
  \label{claim-uhom}
  u_{\hom}(t,x)=\sup_{(b,\phi)\in \mathcal{E}}\big[ f(x{+}m(b,\phi)t)-th(b,\phi) \big].
\end{equation}

 (B) The effective Hamiltonian $\overline{H}$ and its convex conjugate $\mathcal
I$ statisfy 
\begin{equation}
  \label{eq:olH.calI.esti}
  \overline{H}(\theta) \leq c_{16} |\theta|^\alpha + c_{17} \quad
\text{and} \quad \mathcal I(z) \leq c_{18}|z|^{\alpha'} -c_{17}. 
\end{equation}
(C) For all $\delta>0$ there exists $K'_\delta$ such that
\begin{equation}
  \label{eq:uu.upper.bound}
  \uu(t,x) \leq f(x) + \delta + t K'_\delta \quad \text{for all }t>0,\ x\in
  \R^d. 
\end{equation}
(D) There is a constant $C<\infty$ such that  
$\sup_{(t,x,\omega)\in [0,T]{\times} B_R(0) {\times} \Omega_0} \,\, |u_\eps(t,x,\omega)| \leq C(1+R+T)$. 
\end{lemma}
%{\color{blue} C 21.09.2024: We will not need item (B) and (C) above for the proof of the theorem, but it ight be worth keeping these two items for coceptual clarity.}
\begin{proof}
Note first that by definition of $\overline{H}$ and
\eqref{eq:ergodic-limits},
\begin{equation}\label{Hbar lemma2}
\begin{aligned}
\overline{H}(\theta)	
&=\sup_{(b,\phi)\in \mathcal{E}}[\langle
     \theta,m(b,\phi)\rangle-h(b,\phi)]
  =\sup_{y\in \R^d}\sup_{\substack{(b,\phi)\in \mathcal{E}:\\ m(b,\phi)=y}} 
  \:\big[\langle \theta,y\rangle-h(b,\phi) \big] 
=\sup_{y\in\R^d}\big[\langle \theta,y\rangle
   -\inf_{\substack{(b,\phi)\in \mathcal{E}:\\ m(b,\phi)=y}}h(b,\phi)\big].
\end{aligned}
\end{equation}
On the other hand, since $\mathcal{I}$ is the convex conjugate of
$\overline{H}$, we conclude that 
\begin{equation}
 \label{eq:mathcal-I-alt-formula}
  \mathcal{I}(y)=\inf_{\substack{(b,\phi)\in \mathcal{E}:\\ m(b,\phi)=y}}h(b,\phi).
\end{equation}
As a result, and using \eqref{eq:mathcal-I-alt-formula},
\begin{equation}
 \label{eq:uhom-identity}
   \begin{aligned}
    u(t,x)&=\sup_{(b,\phi)\in \mathcal{E}}[f(x{+}m(b,\phi)t)-th(b,\phi)] 
     =\sup_{y\in\R^d}\sup_{\substack{(b,\phi)\in \mathcal{E}:\\ m(b,\phi)=y}} 
         [f(x{+}yt)-th(b,\phi)]\\
    &=\sup_{y\in\R^d}\big[f(x{+}yt)- t \mathcal{I}(y) \big] =u_{\hom}(t,x),
  \end{aligned}
\end{equation}
which finishes the proof of (A).

To establish (B) we use the upper representation of $\overline H$ in
\eqref{eq:effective.hamiltonian}, the lower bound for $L$ in \ref{f2}  (which
is equivalent to the upper bound of $H$), and $\int \phi\,\d \P_0=1$ for all
$(b,\phi)\in \mathcal E$. This gives 
\begin{align*}
\overline H(\theta)&=\sup_{(b,\phi)\in \mathcal E} \int \Big( \frac12
\mathrm{div}(a\theta) + \langle b,\theta\rangle_a - L(b,\theta) \Big) \phi\,\d\P_0
\\
&\leq \sup_{(b,\phi)\in \mathcal E} \int \Big( \frac12
\big\|\mathrm{div}\;\!a\big\|_{L^\infty}|\theta| + \| b\|_a |\theta| -
c_{10} \|b\|_a^{\alpha'} + c_{11}\Big) \phi\,\d\P_0
\\
&\leq \Big(\frac12 \big\|\mathrm{div}\;\!a\big\|_{L^\infty}|\theta| + c
|\theta|^\alpha   + c_{11}\Big)\sup_{(b,\phi)\in \mathcal E} \int
\phi\,\d\P_0\ \leq c_{16} |\theta|^\alpha + c_{17}. 
\end{align*} 
From this upper boun d of $\overline H$ the lower bound for $\mathcal I$ follws
by Legrendre transformation. 

To prove (C) we use \ref{f4} and the lower bound for $\mathcal I$ as follows:
\begin{align*}
\uu(t,x)&=f(x) + \sup\nolimits_{y\in \R^d} \Big( f(x{+}yt) -f(x) - t \mathcal I(y) \Big)
\\
&\leq f(x) + \sup\nolimits_{y\in \R^d} \Big( \delta + K_\delta|ty| -
t \big(c_{18}|y|^{\alpha'} - tc_{17}\big) \Big) \ = \ f(x) + \delta +
K_\delta'\, t,
\end{align*}
which proves (C). To prove the uniform bounds on $u_\eps$ in (D), we use \ref{f2}, implying $L(q,\omega)\geq -c_{11}$ and
$L(0,\omega)\leq c_{13}$. Moreover, \ref{f4} with $\delta=1$ implies 
$|f(x){-}f(y)| \leq 1+ K_1|x{-}y]$.  Arguing as in
\eqref{eq:u-eps_trivial_lower_bound}, for all $\eps\in (0,1)$, $t>0$,
$x\in \R^d$, we obtain 
\begin{align}
  \label{eq:ueps.lower.upper}
-1-c_{13} t - K_1\,\big(\sqrt{t} + t\big)  \leq   u_\eps(t,x,\omega) - f(x)
\leq 1 +c_{11} t + K_1\,\big(\sqrt{t} + t\big)  .
\end{align}
This estimate together with
$|f(x)| \leq K_1|x| + |f(0)| + 1$ (from \ref{f4} with $\delta=1$) give
$|u_\eps(t,x,\omega)| \leq C(1+R+T)$ for all $(t,x,\omega)\in [0,T]{\times}
B_R(0) {\times} \Omega_0$, proving (D). 
\end{proof}

\subsection{Proof of lower bound.}\label{sec proof lb}

\begin{theorem}[Lower bound for $u_\eps$.]
\label{thm lower bound}
Let $u_{\eps}(t,x)$ be defined as in \eqref{eq:u-eps} and $u_{\hom}$ as in
\eqref{eq:lax-oleinik}. Assume \ref{assump:est-erg}-\ref{assump:inf-comp}, \ref{f1}, \ref{f2}-\ref{f4} and fix $T,R>0$. Then for every $0<r < R$, there is a constant $g(r)\to 0$ as $r\to 0$ such that for $y\in B_R(0)$ with $B_r(y) \subset B_R(0)$,  $\P_0$-a.s., 
\begin{equation}
  \label{eq lower bound} 
  \begin{aligned}
&\liminf_{\eps\to 0} \inf_{t\in [0,T]} \frac 1 {|D(\eps,r,y,\omega)|} \int_{D(\eps,r,y,\omega)} \d x \,  \, 
[ u_{\eps}(t,x,\omega) {-} u_{\hom}(t,y)]\geq -g(r), \quad\mbox{where}\\
&\qquad D(\eps, r,y,\omega):= \eps \mathcal C_\infty(\omega) \cap B_r(y).
\end{aligned}
\end{equation}
\end{theorem} 
\begin{proof}
Recall that by \eqref{eq Lemma 4.3} from Lemma 
\ref{lemma:lower-bound-lemma1-extended}, for any $(b,\phi)\in \mathcal E$ and given 
any $\eta>0$, there is $N_\eta(b,\phi) \subset\Omega_0$ with 
$\P_0(N_\eta(b,\phi))\geq 1- \eta$ such that for all $T, R>0$, 
$$
\liminf_{\eps\to 0} \inf_{0\leq t \leq T}  \inf_{\omega \in N_\eta} \inf_{y \in B_R(0)\cap \eps \mathcal C_\infty} \big[u^y_\eps(t,0,\omega) - \big(f^y(t m(b,\phi)) - t h(b,\phi)\big)\big]\geq 0. 
$$
Now for any $y\in B_R(0)$ with $B_r(y)\subset B_R(0)$ we have $B_r(y)\cap \eps \mathcal C_\infty \subset B_R(0) \cap \eps \mathcal C_\infty(\omega)$. Therefore, it follows from the above estimate that for any $(b,\phi)\in \mathcal E$ and  $\eta>0$, there is $N_\eta(b,\phi) \subset\Omega_0$ with 
$\P_0(N_\eta(b,\phi))\geq 1- \eta$ such that 
\begin{equation}\label{eq1 lb}
\liminf_{\eps\to 0} \inf_{0\leq t \leq T}  \inf_{\omega \in N_\eta} \inf_{x \in B_r(y)\cap \eps \mathcal C_\infty} \big[u^{x}_\eps(t,0,\omega) - \big(f^{x}(t m(b,\phi)) - t h(b,\phi)\big)\big]\geq 0. 
\end{equation}
Now let us decompose $D(\eps, r,y,\omega)= \eps \mathcal C_\infty(\omega) \cap B_r(y)$ as a disjoint union $D^1(\eps,r,y,\omega)\cup D^2(\eps,r,y,\omega)$ where $D^2(\eps,r,y,\omega)=\{y^\prime\in B_r(y)\cap \eps\mathcal C_\infty(\omega): \tau_{y^\prime/\eps}\omega\notin N_\eta\big\}$ and $D^1(\eps,r,y,\omega)$ denotes its complement in the set $D(\eps, r,y,\omega)$. Now 
$$
\begin{aligned}
\frac{|D^2(\eps,r,y,\omega)|}{|D(\eps,r,y,\omega)|}=\frac{\big|\big\{x\in B_r(y)\cap \eps\mathcal C_\infty(\omega): \tau_{x/\eps}\omega\notin N_\eta\big\}\big|}{\big|\big\{x\in B_r(y)\cap \eps\mathcal C_\infty(\omega)\big|}
&= \frac{\big|\big\{z\in B_{r/\eps}(y/\eps)\cap \mathcal C_\infty(\omega): \tau_{z}\omega\notin N_\eta\big\}\big|}{\big|\big\{z\in B_{r/\eps}(y/\eps)\cap \mathcal C_\infty(\omega)\big|} \\
&= \frac{\frac 1{|B_{\frac r\eps}|} \int_{B_{\frac r \eps}(y/\eps) \,\, F(\tau_z\omega) \d z}}{\frac 1{|B_{\frac r\eps}|} \int_{B_{\frac r \eps}(y/\eps) \,\, G(\tau_z\omega) \d z}} 
    \end{aligned}
$$
where $F(\omega)= \1_{\{0\in \mathcal C_\infty(\omega)\}} \, \1_{\{\omega\notin N_\eta\}}$ and $G(\omega)= \1_{\{0\in \mathcal C_\infty(\omega)\}}$. Now for every fixed $r>0$, the sets $A_\eps= B_{\frac r\eps}(y/\eps)$ satisfy the F{\o}lner condition, i.e., for every $x\in \R^d$, $\lim_{\eps\to 0} \frac{|\tau_x A_\eps \Delta A_\eps|}{|A_\eps|}=0$. Moreover, since $\P$ is measure preserving and ergodic w.r.t. $\{\tau_z\}_{z\in\R^d}$, by the multidimensional ergodic theorem, there exists a set $N\subset\Omega$ such that $\P(N)=1$ (and therefore $\P_0(N)=1$) such that for every $\omega\in N$, 
\begin{equation}\label{spatial erg}
\lim_{\eps\to 0} \frac{|D^2(\eps,r,y,\omega)|}{|D(\eps,r,y,\omega)|} = \frac{\E^P[F]}{ \E^{\P}[G]}= \frac{\P[0\in \mathcal C_\infty, N_\eta^c]}{\P[0\in \mathcal C_\infty]}= \P_0[N_\eta^c]\leq \eta.
\end{equation}

\noindent Next we decompose 
\begin{equation}\label{dec D1D2}
\begin{aligned}
&\frac 1 {|D(\eps,r,y,\omega)|} \int_{D(\eps,r,y,\omega)} \d x \,   
\big[ u_{\eps}(t,x,\omega) - \big(f(x+ tm(b,\phi))- t h(b,\phi)\big)\big ] \\
&= \frac 1 {|D(\eps,R,\omega)|} \int_{D^1(\eps,r,y,\omega)} \d x \,  \, 
\big[ u_{\eps}(t,x,\omega) - \big(f(x+ tm(b,\phi))- t h(b,\phi)\big)\big ] \\
&\qquad\qquad\qquad\qquad+ \frac 1 {|D(\eps,r,y,\omega)|} \int_{D^2(\eps,r,y,\omega)} \d x \,   \, 
\big[ u_{\eps}(t,x,\omega) - \big(f(x+ tm(b,\phi))- t h(b,\phi)\big)\big ]
\end{aligned}
\end{equation}
We first handle the first integral over $D^1(\eps,r,y,\omega)$. Note that, the identity \eqref{shift2} implies that $u_\eps(t,x,\omega)= u_\eps^x(t,0,\tau_{x/\eps}\omega)$, which, together with \eqref{eq1 lb} and the fact that $x\in D^1(\eps,r,y,\omega)$ implying $x\in B_r(y)\cap \eps\mathcal C_\infty(\omega)$ and  $x/\eps \in N_\eta$, dictate that for any $(b,\phi)\in \mathcal E$, $\eta>0$ and $\omega\in N_\eta(b,\phi)$, 
\begin{equation}\label{est D1}
\liminf_{\eps\to 0} \inf_{t\in [0,T]} \frac 1 {D(\eps,r,y,\omega)|} \int_{D_1(\eps,r,y,\omega)} \d x \,  
\big[ u_{\eps}(t,x,\omega) - \big(f(x+ tm(b,\phi))- t h(b,\phi)\big)\big ] \geq 0. 
\end{equation}
To estimate the second integral over $D^2(\eps,r,y,\omega)$ in \eqref{dec D1D2}, 
we will invoke uniform estimates on $u_\eps(t,x,\omega)$ and on $u_{\mathrm{hom}}(t,y)$ from Lemma \ref{lemma2} (C) and (D) and then apply \eqref{spatial erg} to exploit that the integral over $D^2(\eps,r,y,\omega)$ is taken over a set whose relative density is at most $\eta$.  
Indeed, from Lemma \ref{lemma2} (D), 
$|u_\eps(t,x,\omega)| \leq C(1+R+T)$ for all $(t,x,\omega)\in [0,T]{\times}
B_R(0) {\times} \Omega_0$ and from Lemma \ref{lemma2} (A) and (D),  $f(y+ t m(b,\phi))- t h(b,\phi) \leq \sup_{(b,\phi)\in \mathcal E} [f(y+ t m(b,\phi))- t h(b,\phi)]=\uu(t,y)\leq C(R,T)$. Hence, there is a constant $C=C(R,T)$, such that 
$$
\frac 1 {D(\eps,r,y,\omega)|} \inf_{t\in [0,T]} \int_{D^2(\eps,r,y,\omega)} \d x \,  
\big[ u_{\eps}(t,x,\omega) - \big(f(x+ tm(b,\phi))- t h(b,\phi)\big)\big ] \geq -C \frac{D^2(\eps,r,y,\omega)}{D(\eps,r,y,\omega)}.
$$
Applying \eqref{spatial erg} we have that there is $C\in(0,\infty)$ such that for every $(b,\phi)\in \mathcal E$ and every $\eta>0$, and $\P_0$-a.s., 
\begin{equation}\label{est D2}
\liminf_{\eps\to 0} \inf_{t\in [0,T]} \frac 1 {D(\eps,r,y,\omega)|} \int_{D^2(\eps,r,y,\omega)} \d x \,  \big[ u_{\eps}(t,x,\omega) - \big(f(x+ tm(b,\phi))- t h(b,\phi)\big)\big ] \geq -C\eta
\end{equation}
Combining \eqref{est D1} and \eqref{est D2} it follows from \eqref{dec D1D2} that there is $C\in (0,\infty)$ such that for every $(b,\phi)\in \mathcal E$,  $\eta>0$ and $\omega\in N_\eta(b,\phi)$, 
\begin{equation}\label{est D1D2}
\liminf_{\eps\to 0} \inf_{t\in [0,T]} \frac 1 {D(\eps,r,y,\omega)|} \int_{D(\eps,r,y,\omega)} \d x \,   
\big[ u_{\eps}(t,x,\omega) - \big(f(x+ tm(b,\phi))- t h(b,\phi)\big)\big ] \geq -C\eta
\end{equation}

From here we will now prove \eqref{eq lower bound}. Using $\P_0(N_\eta(b,\phi))>1-\eta$, we 
conclude that $\widehat N(b,\phi):=\cup_{\eta>0} N_\eta (b,\phi)$ has full 
measure, i.e.\ $ \P_0(\widehat N(b,\phi)) = 1 $. Next we convince ourselves that $u_{\hom}$ in \eqref{claim-uhom} can already be characterized by 
a countable  subset 
$ \big( (b_j,\phi_j)\big)_{j\in \N} \subset \cE$. To see this, we first observe 
that \eqref{claim-uhom} can be rewritten as 
\[
u_{\hom}(t,y)= \sup_{(m,h)\in \mathbf A} \big[ f(y{+}tm) - t h\big] \quad 
\text{ with } \mathbf A:= \bigset{ \big(m(b,\phi),h(b,\phi)\big) \in \R^{d+1} }
 { (b,\phi)\in \cE }. 
\]
Moreover, for fixed $(t,y)$ the function $(m,h) \mapsto f(y{+}tm) -th$ is 
continuous. Hence, it suffices to replace the set $\mathbf A$ in the supremum 
by any dense set $ \mathbf D \subset \mathbf A$. Because $\mathbf A \subset 
\R^{d+1}$ we can choose $\mathbf D$ to be countable, namely $\mathbf D = 
\bigset{(m_j,h_j)\in \mathbf A}{ j\in \N}$. By definition there exist $(b_j,\phi_j) 
\in \cE$ such that $(m_j,h_j)=m \big(b_j,\phi_j), h(b_j,\phi_j)\big)$. Hence, 
we arrive at the relation 
\begin{align}
 \label{eq:u-hom.dense}
   u_{\hom}(t,y)= \sup\nolimits_{ j \in \N} \, u^j(t,y), \qquad u^j(t,y):= f(y{+}tm(b_j,\phi_j))
   - t h(b_j,\phi_j).    
\end{align}

\noindent Using this, we now define the set $ \widetilde N := \cap _{j\in \N} 
\widehat N(b_j,\phi_j)$, which still has full $\P_0$ measure, and obtain from \eqref{est D1D2} that for every $y\in B_R(0)$ with $B_r(y)\subset B_R(0)$, 
\begin{equation}
  \label{eq:3.lb}
 \forall\, j\in \N\ \forall\, \omega \in \widetilde N {:} \quad 
 \liminf_{\eps\to 0} \inf_{t\in [0,T]}\frac 1 {|D(\eps,r,y,\omega)|} \int_{D(\eps,r,y,\omega)} \d x   \,
 \big[ u_\eps(t,x,\omega) - u^j(t,x)\big] \geq 0.
\end{equation}

Now for any $\eps_n\to 0$, $n$ large enough, $t\in [0,T]$, $y\in B_R(0)$ with $B_r(y)\subset B_R(0)$, $j\in \N$ and $\omega\in N$ with $\P_0(N)=1$, 
\begin{equation}\label{mod cont}
\frac 1 {|D(\eps_n,r,y,\omega)|} \int_{D(\eps_n,r,y,\omega)} \d x   \,
 u_{\eps_n}(t,x,\omega) \geq \frac 1 {|D(\eps_n,r,y,\omega)|} \int_{D(\eps_n,r,y,\omega)} \d x \,\,  u^j(t,x).
\end{equation}
Recalling \eqref{eq:u-hom.dense}, we have $u^j(t,y)= f(y+ tm_j) - t h_j$, where $m_j=m(b_j,\phi_j)$ and $h_j=h(b_j,\phi_j)$. Therefore, by \ref{f4}, for any $\delta>0$ there exists $K_\delta$ such that $|u^j(t,z)- u^j(t,x)|= |f(x+ tm_j)- f(z+ t m_j)| \leq K_\delta|x-z| + \delta$. So if we define \begin{equation}\label{def gr}
g(r):= \inf_{\delta>0}[K_\delta r + \delta]
\end{equation}
then $g(r)\to 0$ as $r\to 0$. Indeed, given $\eps>0$ choose $\delta=\eps/2$ and then pick $r< \eps/(2K_{\eps/2})$ so $ r K_{\eps/2}+ \eps/2<\eps$, so $\omega(r)\leq \eps$. Hence, $g(r)\to 0$ as $r\to 0$. Moreover, for all $j$ and $t\in [0,T]$ and $x,z\in B_R(0)$, $|u^j(t,z)- u^j(t,x)|\leq g(|x-z|)$. Therefore, we have  
$$
\frac 1 {|D(\eps_n,r,y,\omega)|} \int_{D(\eps_n,r,y,\omega)} \d x \,\,  u^j(t,x) \geq  u^j(t,y) - g(r).
$$
Since the above bound holds for every $j$, we can now pass to the supremum over $j$, use from \eqref{eq:u-hom.dense} that $\uu(t,y)= \sup_j u^j(t,y)$ and conclude from \eqref{mod cont} that  
\begin{equation}\label{mod cont 2}
\frac 1 {|D(\eps_n,r,y,\omega)|} \int_{D(\eps_n,r,y,\omega)} \d x   \,
 u_{\eps_n}(t,x,\omega) \geq \uu(t,y) - g(r).
\end{equation}
Bringing $\uu(t,y)$ to the left hand side, passing to $\inf_{t\in[0,T]}$ and then to $\liminf_{\eps_n\to 0}$, we obtain \eqref{eq lower bound}, concluding the proof of Theorem \ref{thm lower bound}.

\end{proof}

\section{Entropic Variational Analysis}
\label{sec-equivalence}

Recall from \eqref{eq:ergodic-limits} and \eqref{Hbar lemma2}, the variational formula for $\overline H(\theta)$: 
\begin{equation}\label{Hbar-2}
\overline{H}(\theta)= \sup_{(b,\phi)\in \mathcal{E}} 
   \bigg(\int  d\P_0 \, \phi \left[\frac{1}{2}\mathrm{div}(a\theta)
   +\langle \theta,b\rangle_a- L(b,\cdot))\right].
\end{equation}

The goal of this section is to prove Theorem \ref{thm-lb-geq-ub} below establishing 
the lower bound 
\begin{equation}
 \label{eq:Lambda_def}
		\overline H(\theta) \geq \overline\Lambda(\theta) \qquad\forall \theta\in \R^d, \quad\mbox{where} \quad \overline{\Lambda}(\theta):=\inf_{G\in \mathcal{G}_{\delta}}\left(\mathrm{ess\,sup}_{\P_0}\left[\frac{1}{2}\mathrm{div}(a(G+\theta))+ H(G+\theta)\right]\right).
\end{equation}
We now define the class of gradients $\mathcal G_\delta$ and 
the corresponding ``correctors". 
\subsection{Correctors.}
\label{classG}
Given any $\delta>0$, we start this section by defining the class of gradients
$G\in \mathcal{G}_\delta$ and the corresponding ``correctors"
$V_G:\R^d\times\Omega_0 \rightarrow \R^d$.  Let $\mathcal{G}_\delta$ be the
class of functions $G:\Omega_0\to \R^d$ satisfying the following properties:

\noindent$\bullet$ {\bf $L^{1+\delta}(\P_0)$-boundedness}:  
The following inequalities hold:
\begin{equation}\label{eq:unifbound-L-alpha} 
\|G\|_{L^{1+\delta}(\P_0)}<\infty,
\end{equation}	
and \begin{equation}\label{eq:unifbound-div}
	\esssup_{\P_0}\left[\frac{1}{2}\mathrm{div}(a(G+\theta))+H(G+\theta)\right]<\infty.
\end{equation}
		
\noindent$\bullet$ {\bf Curl-free property on the cluster}: Given any 
$G: \Omega_0\to \R^d$, with a slight abuse of notation we will continue to write 
\[
G: \R^d \times \Omega_0 \to \R^d, \quad\mbox{with}\quad G(x,\omega)=G(\tau_x \omega).
\]
Now, for $\P_0$-almost every $\omega\in \Omega_0$, we require that $G$ is curl-free,
meaning $\nabla \times G(\cdot,\omega)=0$ on $\mathcal C_\infty$, or simply, for 
$\P_0$-almost every $\omega\in \Omega_0$, we have 
\begin{equation}\label{eq:closed}
\int_{\mathrm C} G(\cdot,\omega)\,\cdot \d\mathbf r=0
\end{equation}
for every rectifiable simple closed path $\mathrm C$ on $\mathcal C_\infty$. For all
$G$ satisfying \eqref{eq:closed} we define $V_G:\R^d\times \Omega_0\to \R^d$ by 
\begin{equation}
 \label{def-V}
	V_G(x,\omega):=\int_{0\leadsto x} G(\cdot,\omega)\cdot\d\mathbf r,
\end{equation} 
where $0\leadsto x$ is any piecewise smooth curve contained in $\mathcal C_\infty$ 
(and 0 when $x\notin \mathcal C_\infty$). Note that the choice of the smooth curve 
is irrelevant, thanks to \eqref{eq:closed}.

\noindent$\bullet$ {\bf Zero induced mean:} Recall the definition of $n(\omega,e)$ 
from \eqref{def-n} and set $\mathfrak v_e=\mathfrak v_e(\omega)=n(\omega,e)e$. Then 
we require that 
\begin{equation}
 \label{eq:meanzero}
  \bE_0[V_G(\mathfrak v_e,\cdot)]=0.
\end{equation}

\begin{definition}\label{def-class-G}
For any $\delta>0$, we say that $G\in \mathcal G_\delta$ if 
\eqref{eq:unifbound-L-alpha}-\eqref{eq:closed} and \eqref{eq:meanzero} hold.
Similarly, we declare that $G\in \mathcal{G}_\infty$ if the above conditions hold, 
but replace $\eqref{eq:unifbound-L-alpha}$,  by 
\begin{equation}\label{eq:unifbound}
	\mathrm{ess\,sup}_{\P_0}|G(\omega)|<\infty.
\end{equation}
\end{definition}

A crucial fact about the correctors $V_G$ for $G\in \mathcal{G}_\infty$ is 
underlined by the following result that dictates their \textit{sublinear} growth 
at infinity inside $\mathcal C_\infty$:

\begin{theorem}
\label{thm:sublinear}
Assume \ref{assump:est-erg}-\ref{assump:fkg}. Fix $d\geq 2$, $G\in \mathcal{G}_\infty$ and recall that $D(\eps, r):= D(\eps,r,\omega)= \eps \mathcal C_\infty(\omega) \cap B_r(0)$ for $r>0$. Then for $\P_0$-a.e. $\omega\in \Omega_0$,
we have 
\[
\lim_{\eps\to 0 }\sup_{x \in D(\eps,1)} \eps 
   \big|V_G\big(\frac x\eps,\omega\big)\big |=0 \qquad \P_0\textrm{-a.s.} 
\]
\end{theorem}	
\noindent  The above result holds the key for the upper bound shown later in 
Proposition \ref{prop: up-bound-lin-dat}. 
The proof of Theorem \ref{thm:sublinear} is quite long and technical, and will 
be deferred to until Section \ref{sec-proof-sublinear}. 	
	
\subsection{The lower bound $\overline H(\cdot) \geq\overline\Lambda(\cdot)$}
 \label{subsec lb HLambda}
 
As mentioned before, our goal in this section is to prove  

\begin{theorem}\label{thm-lb-geq-ub}
Assume \ref{assump:est-erg}-\ref{assump:fkg}, \ref{f1}, \ref{f1'}, \ref{f2}-\ref{f3}. Then for any $\theta\in \R^d$, 
$$
\overline{H}(\theta) \geq \overline\Lambda(\theta).
$$
\end{theorem}
In Theorem \ref{thm-equivalence} we will show that in fact equality holds in the above theorem. 
The proof of Theorem \ref{thm-lb-geq-ub} is divided into several steps. In Section \ref{sec-lb-geq-ub} we establish this lower bound by combining three crucial results: Lemma \ref{lemma1-lemma-minmax3}, Lemma \ref{lemma2-lemma-minmax3} and Proposition \ref{prop-minmax3}. The latter result is proved in Section \ref{sec-prop-minmax3}. From now on we assume the same hypotheses from Theorem \ref{thm-lb-geq-ub}.

We set 
\begin{equation}\label{def-class-D}
\mathcal D:=\big\{g:C^2_c(\Omega_0): g: \Omega_0 \to\R\big\}
\end{equation}
to be the linear space of functions on $\Omega_0$ with compact support, such that their first and second weak derivatives (defined in Section \ref{sec:env-process}) exist and are continuous. For any $g\in\mathcal{D}$, define 
\begin{equation}\label{def-R-theta}
 R_{\theta} g(\omega)=\frac 12 \, \mathrm{div}\big(a(\omega) \,(\nabla g(\omega)+\theta)\big), \quad\mbox{and write}\quad R= R_0.
 \end{equation}
As already observed in \cite{KRV06}, we have that for any  $g\in \mathcal{D}$,
\begin{equation}\label{eq:eq19}
	 \int \d \P_0\phi \big( R g + \langle  b, \nabla g \rangle_a\big)
\begin{cases}
=0 \quad\forall  g\in\mathcal D \qquad\qquad\, \mbox{if} \qquad ( b, \phi)\in\mathcal E,
\\
\ne 0\quad \mbox{for some}\,\, g\in\mathcal D \quad \mbox{if} \qquad ( b, \phi)\notin\mathcal E,
\end{cases}
\end{equation}
and hence, by taking constant multiples if $(b, \phi)\notin\mathcal E$, we conclude that the infimum over $g\in\mathcal D$ on the left hand side of \eqref{eq:eq19} is $0$ if
$( b,\phi)\in\mathcal E$, and $-\infty$ otherwise. Therefore, from \eqref{Hbar-2} it follows 
\begin{equation}\label{Hbar-3}
\begin{aligned}
\overline{H}(\theta)&=\sup_{\phi\in\Phi}\,\,\sup_{ b\in B_\phi}\,\,\inf_{g\in\mathcal D} \bigg[\int \d \P_0\phi \bigg(\frac{1}{2}\mathrm{div}(a\theta)+\langle \theta,b\rangle_a- L(b,\omega))+  \big(R g + \langle b, \nabla g \rangle_a\big)\bigg)\bigg],
\end{aligned}
\end{equation} where \begin{equation}\label{eq:Phi-def}
	\Phi:=\bigg\{\phi\in L^1_+(\P_0):  \int \phi \d\P_0=1\bigg\}.
\end{equation}
Furthermore, for any given $\phi\in\Phi$, we set 
\begin{equation}\label{eq:B-phi-def}
	B_\phi:=\bigg\{b\in L^1_a(\phi \d\P_0):\forall \omega\in \Omega_0: x\mapsto b(\tau_x \omega)\in \mathrm{Lip}\bigg\},
\end{equation}
with $L^1_a(\phi\d\P_0)$ being defined in \eqref{eq:L1a-def}. 
We remark that, for any $\phi\in \Phi$, the set $B_\phi$ contains constant functions $b$.

\subsection{Entropic coercivity and min-max theorems: proof of Theorem \ref{thm-lb-geq-ub}}\label{sec-lb-geq-ub}

First, we will prove 
\begin{lemma}\label{lemma-minmax1}
Let $\overline H(\theta)$ be the variational formula defined in \eqref{Hbar-2} (or equivalently, in \eqref{Hbar-3}). Then 
\begin{equation}\label{eq-minmax1}
\overline H(\theta)= \sup_{\phi\in \Phi} \inf_{g\in \mathcal{D}} \left[\int \d \P_0\phi \big( R_{\theta}g + H(\theta+\nabla g(\omega),\omega)\big)\right].
 \end{equation}
\end{lemma}
\begin{proof}
By \eqref{Hbar-3} and using the definition of $R_\theta$ from \eqref{def-R-theta}, 
 \begin{equation}\label{Hbar-minmax1}
	\overline{H}(\theta)
=\sup_{\phi\in \Phi}\,\,\sup_{ b\in B_\phi}\,\,\inf_{g\in\mathcal D} \bigg[\int \d \P_0\phi \big(\langle \theta+\nabla g, b\rangle_a+ R_{\theta}g - L(b,\omega)\big)\bigg].
\end{equation}

For any given $\phi\in \Phi$, we would like to exchange the supremum over $b\in B_\phi$ with the infimum over $g\in \mathcal D$, for which we would like to apply the min-max theorem from \cite[Theorem 8, p. 319 ]{AE84}, the requirements for which are verified as follows. First, we fix any $\phi\in \Phi$,  note that the map
$$
B_\phi\ni b\mapsto \int \d \P_0\phi \big(\langle \theta+\nabla g, b\rangle_a+ R_{\theta}g - L(b,\omega)\big)\qquad\mbox{is concave and upper semicontinuous}, 
$$
while the map 
$$
\mathcal D\ni g\mapsto \int \d \P_0\phi \big(\langle \theta+\nabla g, b\rangle_a+ R_{\theta}g - L(b,\omega)\big)\qquad\mbox{ is convex and lower semicontinuous.}
$$
We need to verify the remaining compactness (resp.\ coercivity): we will show that for a fixed $\phi \in \Phi$ and $g\in \mathcal{D}$, the level sets  
\begin{equation}\label{Ec}
\begin{aligned}
	E_c:&=\bigg\{b \in B_\phi: \int \d \P_0\phi \big(\langle \theta+\nabla g, b\rangle_a+ R_{\theta}g - L(b,\omega)\big)
 \geq c\bigg\}	
	\\
	&=\bigg\{b\in L_a^1(\phi \d\P_0):{\mathcal C_\infty}(\omega)\ni x\mapsto b(x,\omega)\in \text{Lip and} ~\forall~ \omega\in \Omega_0,\\ 
	&\qquad\qquad\qquad\int \d \P_0\phi \big(\langle \theta+\nabla g, b\rangle_a+ R_{\theta}g - L(b,\omega)\big)
 \geq c\bigg\} \quad\mbox{are weakly compact in $L^1_a(\phi \d\P_0)$.}
\end{aligned} 
\end{equation}
Indeed, by the Eberlein-\v{S}mulian theorem (see \cite[p.430]{DS58}), checking the latter condition is equivalent to verifying that the set $E_c$ above is
\begin{enumerate}
    \item[(A)]   weakly closed, and 
    \item[(B)] sequentially weakly compact in $L^1_a(\phi \d\P_0)$. 
\end{enumerate}
For the second condition (B), it is enough to show that $E_c$ is bounded and uniformly integrable, but both these conditions follow from the coercivity of $L$. Indeed, recall \eqref{eq:H1-L} from \ref{f2}:
$$
c_{10}\|q\|_a^{\alpha'}-c_{11}\leq L(q,\omega)\leq c_{12}\|q\|_a^{\alpha'}+c_{13}, \qquad \alpha^\prime= \frac{\alpha}{\alpha-1},  \,\, 1 < \alpha  <\infty.
$$
On the other hand, using that $g\in \mathcal D$ has compact support and $\nabla g$ is continuous, $|\theta+\nabla g| \leq (|\theta|+ \|\nabla g\|_{L^\infty(\P_0)}=: C_1(\theta,g)<\infty$. Moreover, 
by \ref{f1}, $|\mathrm{div}(a)|\leq C$, so we can find a constant $C_2(\theta,g)$ such that by \eqref{def-R-theta}, 
$|R_\theta g|\leq C_2(\theta,g)$. Hence, 
\begin{equation}\label{eq:eq23}
	\int \d \P_0\phi \big(\langle \theta+\nabla g, b\rangle_a+ R_{\theta}g\big)\leq C_2(\theta,g)+C_1(\theta,g)\int \d\P_0 \phi \|b\|_a<\infty
\end{equation}since $b\in L_a^1(\phi\d\P_0)$, recall \eqref{eq:L1a-def}. This shows that $E_c$ is sequentially weakly compact in $L^1_a(\phi \d\P_0)$.

Thus, it remains to show that $E_c$ is weakly closed. Since $E_c$ is convex, it suffices to show that $E_c$ is strongly closed. Indeed, suppose that $(b_n)_n\subset E_c$ such that $b_n\to b$ in $L_a^1(\phi d \P_0)$. Passing to a subsequence, since $L$ is lower semicontinuous and by Fatou's lemma, one can easily verify that $\int \d \P_0\phi \big(\langle \theta+\nabla g, b\rangle_a+ R_{\theta}g - L(b,\omega)\big)
 \geq c$. We will construct a function $\tilde{b}$ such that $\tilde{b}=b$ $\P_0$-a.s. and for all $\omega\in \Omega_0, ~\R^d\ni x\mapsto \tilde{b}(x,\omega)\in $ Lip. Let $\Omega_0'$ with $\P_0(\Omega_0')=1$ such that $b_n(\omega)\to b(\omega)$ for all $\omega\in \Omega_0'$. For a fixed $\omega\in \Omega_0'$, we know that the family $(b_n(\cdot,\omega))_n$ is uniformly equicontinuous and on any compact set $K\subset \R^d$ and $x\in K$,\begin{equation*}
 	|b_n(x,\omega)|\leq |x|+ |b_n(0,\omega)|\leq \diam(K)+\sup_n |b_n(0,\omega)|.
  \end{equation*}
As $b_n(0,\omega)\to b(0,\omega)$, the supremum above is finite. Hence, for fixed $\omega$, the family of continuous functions $(b_n(\cdot,\omega))_n$ is globally uniformly equicontinuous and uniformly bounded on compact sets. By the Arzel\`{a}-Ascoli theorem, the sequence $(b_n(\cdot,\omega))$ converges uniformly on compact sets and therefore converges pointwise to some function $\tilde{f}(\cdot,\omega)\in$ Lip.  By definition, $\tilde{f}(0,\omega)=b(\omega)$ $\P_0$-a.s. Now let us consider the set \begin{equation*}
	\Omega_0^{''}:=\{\omega\in \Omega_0: \exists x\in \R^d, \omega'\in \Omega_0': \omega =\tau_x \omega'  \}.
\end{equation*}

Then we define \begin{equation}
	\tilde{b}(\omega):=\begin{cases}
		\tilde{f}(x,\omega')~\text{ if } \omega=\tau_x \omega' \text{ for some } x\in \R^d, \omega'\in \Omega_0',\\
		0~\text{ otherwise}.
	\end{cases}
\end{equation}
Let us first check that $\tilde{b}$ is well-defined. Indeed, suppose that $\omega=\tau_x \omega' = \tau_y \omega''$ for some $x,y\in \R^d$ and $\omega',\omega''\in \Omega_0'$. Then \begin{align*}
	\tilde{f}(x,\omega')=\lim_{n\to\infty} b_n(x,\omega')=\lim_{n\to\infty} b_n(0,\tau_x\omega')&=\lim_{n\to\infty} b_n(0,\tau_y\omega'')\\
	&=\lim_{n\to\infty} b_n(y,\omega'')\\
	&=\tilde{f}(y,\omega'').
	\end{align*}
Hence, the function $\tilde b$ is well defined. Notice that on $\Omega_0'$, $\tilde{b}(\omega)=f(0,\omega)=b(\omega)$ , so $\tilde{b}=b$ $\P_0$-a.s. Finally, let us check that for all $\omega\in \Omega_0$, ${\mathcal C_\infty}(\omega)\ni x \mapsto \tilde{b}(x,\omega)\in $ Lip. Indeed, 
\begin{enumerate}
	\item If $\omega\in \Omega_0''$, then $\omega=\tau_z\omega'$ for some $z\in \R^d$ and $\omega'\in \Omega_0$ and for $x,y\in \R^d$,
	\begin{align*}
		|\tilde{b}(x,\omega)-\tilde{b}(y,\omega)|=|\tilde{b}(x,\tau_z \omega')-\tilde{b}(y,\tau_z \omega')|&=|\tilde{f}(x+z,\omega')-\tilde{f}(y+z,\omega')|
		\\&\leq |x-y|.
	\end{align*} 
	\item On the other hand, if $\omega\in \Omega_0\setminus \Omega_0''$, then the same holds for $\tau_x\omega$ for all $x\in {\mathcal C_\infty}(\omega)$, and the Lipschitz condition is trivially satisfied. 
\end{enumerate}

This finishes the proof that $E_c$ is weakly compact in $L_a^1(\phi \d\P_0)$, and therefore, by the aforementioned min-max theorem, we can exchange the $\sup_{b\in B_\phi}$ and $\inf_{g\in \mathcal{D}}$
in \eqref{Hbar-minmax1} to obtain \begin{equation*}
	\overline{H}(\theta)=\sup_{\phi\in \Phi}\inf_{g\in \mathcal{D}}\sup_{b\in B_\phi} \left[\int \d \P_0\phi \big(\langle \theta+\nabla g, b\rangle_a+ R_{\theta}g - L(b,\omega)\big)\right].
\end{equation*}
Since the integrand depends locally in $b$, we can bring the supremum over $b$  inside the integral, and use the duality between $H$ and $L$ to conclude that 
\begin{equation}\label{Hbar-minmax2}
	\begin{aligned}
	\overline{H}(\theta)&=\sup_{\phi\in \Phi}\inf_{g\in \mathcal{D}} \left[\int \d \P_0\phi \bigg( R_{\theta}g +\sup_{b\in B_\phi}\big[\langle \theta+\nabla g, b\rangle_a- L(b,\omega)\big]\bigg)\right]\\&=\sup_{\phi\in \Phi} \inf_{g\in \mathcal{D}} \bigg[\int \d \P_0\phi \big( R_{\theta}g + H(\theta+\nabla g(\omega),\omega)\big)\bigg].	
	\end{aligned}
\end{equation}
In the last equality we used that, for any $\phi\in \Phi$, the set $B_\phi$ defined in \eqref{eq:B-phi-def} contains constants, so that \begin{align*}
	\sup_{b\in B_\phi}\big[\langle \theta+\nabla g, b\rangle_a- L(b,\omega)\big]&=\sup_{y\in \R^d}\big[\langle \theta+\nabla g, y\rangle_a- L(y,\omega)\big] 
	=H(\theta+\nabla g(\omega),\omega).
\end{align*}
\end{proof}
We would like to now swap the order of $\sup_\phi$ and $\inf_g$ in \eqref{Hbar-minmax2}.
\begin{lemma}\label{lemma-minmax2}
With $R_\theta$ defined in \eqref{def-R-theta}, let 
\begin{equation}\label{eq:S-def}
	\mathcal{S}_{\theta}(g)(\omega):=R_{\theta}g(\omega)+H(\theta+\nabla g(\omega),\omega)),\quad g\in \mathcal{D}.
\end{equation}
Then for any $\theta\in \R^d$, 
\begin{equation}\label{eq-minmax2}
\begin{aligned}
\overline H(\theta) &= \sup_{\phi\in \Phi}\,\,\inf_{g\in\mathcal{D}}\,\, \left[\int \d\P_0\phi\, \mathcal{S}_{\theta}(g)\right] \geq \liminf_{\eps\to 0}\,\,\inf_{g\in\mathcal{D}} \bigg[ \eps \log \int \d\P_0\, \exp\big[ \eps^{-1} \mathcal S_{\theta}(g)\big] \bigg].
\end{aligned}
\end{equation}
\end{lemma}
\begin{proof}
  For any probability density $\varphi \geq 0$ on $\Omega_0$ (i.e., $\int_{\Omega_0} \varphi \d\P_0=1$), let
$$
\mathrm{Ent}_{\P_0}(\varphi)= \int \varphi \log \varphi \,\d\mathbb{P}_0\geq 0
$$ 
be the entropy of $\varphi$. Its non-negativity is a consequence of the Jensen's inequality. Moreover, the map $\varphi \mapsto \mathrm{Ent}_{\P_0}(\varphi)$ is convex, weakly lower semicontinuous and has weakly compact sub-level sets, meaning, for any $\ell>0$, $\{\varphi: \mathrm{Ent}_{\P_0}(\varphi) \leq \ell\}$ is compact in the weak topology.  
Thus, for any $\eps>0$ we have the lower bound 
\begin{equation}\label{eq minmax2.5}
\begin{aligned}
\overline H(\theta) &\geq \sup_{\phi\in \Phi}\inf_{g\in\mathcal{D}}\,\,\bigg[\int \d\P_0\phi\, \mathcal{S}_{\theta}( g) -\eps \mathrm{Ent}_{\mathbb{P}_0}(\phi)\bigg] = \sup_{\phi\in \Phi}\inf_{g\in\mathcal{D}}\,\,\bigg[\int \d\P_0\phi\, (\mathcal{S}_{\theta}(g) -\eps \log \phi)\bigg].\end{aligned}
\end{equation}
 Similarly as in \eqref{eq:eq23}, we use the fact that $g\in \mathcal{D}$ together with the assumptions \ref{f1} to conclude that there is a constant $C_2(\theta,g)$ such that $|R_\theta(g)|\leq C_2(\theta,g)$, so that 
 \begin{equation*}
 \begin{aligned}
 	 \int \d\mathbb{P}_0\phi\mathcal{S}_{\theta}(g)\leq C_2(\theta,g)+ \int \d\P_0\phi  H(\theta+ \nabla g) &\leq C_2(\theta,g)+c_8 \int \d\P_0\phi\| \theta+ \nabla g \|_a^\alpha +c_9  \\
	 &\leq C_2(\theta,g) + c_8 C(\alpha, \theta,g)+ c_9<\infty.
\end{aligned}
 \end{equation*}
 where for the second inequality we used the upper bound from \eqref{eq:H1-H} in \ref{f2} and for the third inequality we used $\langle a(\omega)x,x\rangle \leq c_5|x|^2$ from \ref{f1} and again 
 that $\sup_\omega |\theta+ \nabla g(\omega)| \leq |\theta|+ \|\nabla g\|_{L^\infty(\P_0)}$ for $g\in \mathcal D$. 
 Now, for any fixed $\phi\in \Phi$, the map 
 \begin{equation*}
 	\mathcal D\ni g\mapsto \int \d\P_0\phi\, (\mathcal{S}_{\theta}(g) -\eps \log \phi) \quad\mbox{is convex and continuous,}
 \end{equation*}
  while for any fixed $g\in \mathcal D$, the map 
  \begin{equation*}
  \begin{aligned}
 	\Phi\ni \varphi \mapsto \int \d\P_0\varphi\, (\mathcal{S}_{\theta}(g) -\eps \log \varphi)
    &\quad\mbox{is concave, upper-semicontinuous and} \\
    &\quad\mbox{has compact superlevel sets in the weak} \,\, L^1_+(\mathbb{P}_0)\,\,\mbox{topology.}
    \end{aligned}
 \end{equation*}
Hence, we can again use Von-Neumann's minimax theorem to justify changing the order of $\sup_{\phi}$ and $\inf_{g\in\mathcal{D}}$ in \eqref{eq minmax2.5}, leading to 
 $$
 \overline H(\theta)\geq \inf_{g\in\mathcal{D}} \sup_{\phi}\,\,\bigg[\int \d\P_0\phi\, (\mathcal{S}_{\theta}(g) -\eps \log \phi)\bigg]. $$
 The above variational problem over $\phi$ subject to the condition $\int\phi d\mathbb{P}_0=1$ can be solved explicitly and the maximizing density is
 $$\phi= \frac{\exp[\eps^{-1} \mathcal{S}_{\theta}(g)]}{ \int \d\mathbb{P}_0\exp[\eps^{-1}  \mathcal{S}_{\theta}(g)]}. $$
 We replace this value of $\phi$ in the last lower bound for $\overline H(\theta)$  to obtain  \begin{equation*}
 	\overline{H}(\theta)\geq \,\,\inf_{g\in\mathcal{D}} \bigg[ \eps \log \int \d\P_0\, \exp\big[ \eps^{-1} \mathcal{S}_{\theta}(g)\big] \bigg]. \end{equation*}
 We let $\eps\to 0$, to deduce the lower bound claimed in \eqref{eq-minmax2}.
\end{proof}

 Given the above results, the lower bound in Theorem \ref{thm-lb-geq-ub} will now be a
consequence of the following technical result that will be established in
Section \ref{sec-prop-minmax3}. 
  
\begin{prop}
\label{prop-minmax3}
For any given $\eps>0$, there exists a sequence $\eps_n\to 0$ and a sequence of
functions $(g_n)_n\subset \mathcal D$ so that
\begin{equation}
  \label{eq-minmax3}
  \overline H(\theta) \geq \eps_n \log \E_0\bigg[ \e^{\eps_n^{-1}
    \mathcal{S}_{\theta}(g_n,\cdot)}\bigg]-\eps, 
\end{equation}
and $G_n(\omega):=\nabla g_n$ converges weakly in $L^{1+\delta}(\P_0)$ (with
$\delta>0$ as in \eqref{eq:xi-mom-bound}) and in distribution (along a subsequence) to some
$G$. Furthermore, $G\in \mathcal{G}_{\delta}$, which is defined in Section \ref{classG}. 
\end{prop}

\begin{proof}
[{\bf Proof of Theorem \ref{thm-lb-geq-ub} (assuming Proposition
  \ref{prop-minmax3})}]
By Proposition \ref{prop-minmax3}, for 
$r>0$, we pick some sequence $\eps_n\to 0$ and $g_n\in \mathcal{D}$
satisfying
\begin{equation*}
  \overline H(\theta) \geq \eps_n \log \E_0\big[ \e^{\eps_n^{-1}
    \mathcal{S}_\theta(g_n,\cdot)}\big] - r. 
\end{equation*}
For fixed $n$, the map
$\lambda\in [0,\infty)\to \frac{1}{\lambda}\log \E_0[\e^{\lambda
  \mathcal{S}_\theta(g_n,\cdot)}]$ is increasing, so for each
$\eta, \lambda>0$, if $n$ is large enough,
\begin{align*}
  \overline H(\theta) &\geq \frac{1}{\lambda} \log \E_0\big[ \e^{\lambda
    \mathcal{S}_\theta(g_n,\cdot)}\big] - r \\ 
  &=\frac{1}{\lambda}\log \E_0\bigg[\e^{\lambda
    \big(\frac{1}{2}\mathrm{div}(a(\omega)(G_n(\omega)+\theta))+
    H(\theta+\nabla
    G_n(\omega),\omega)\big)}\bigg]-r. 
\end{align*}
For any $M,\lambda>0$, the map
\[
  x\mapsto \e^{\lambda\left(M\wedge \left(\frac{1}{2}\mathrm{div}
 (a(\omega)(x+\theta))+ H(\theta+ x,\omega)\right)\right)}
\]
is continuous and bounded.  Thus, letting $n\to \infty$ and using the fact
from Proposition \ref{prop-minmax3} that $G_n$ converges to $G\in \mathcal G_\delta$ in distribution, we conclude
from the above bound that
\begin{equation*}
  \overline H(\theta) \geq \frac{1}{\lambda}\log \E_0\bigg[\e^{\lambda \big(M\wedge \big(\frac{1}{2}\mathrm{div}(a(\omega)(G(\omega)+\theta))+ H(\theta+ G(\omega),\omega)\big)\big)}\bigg]-r.
  \end{equation*}
  Now by letting $M\nearrow \infty$ and using monotone convergence theorem, we obtain 
  \begin{equation}\label{eq:eq32}
  \overline H(\theta) \geq  \log \bigg\| \e^{ \frac{1}{2}\mathrm{div}(a(\omega)(G(\omega)+\theta))+ H(\theta+ G(\omega),\omega)} \bigg\|_{L^\lambda(\P_0)}- r.
  \end{equation}
  Finally, letting $\lambda\to \infty$,  we obtain
\begin{equation}
 \label{eq:eq17}
  \begin{aligned}
  \overline H(\theta) &\geq
  \mathrm{ess\,sup}_{\P_0}\bigg[\frac{1}{2}\mathrm{div}(a(G+\theta))+
    H(G+\theta)\bigg]
  -r  
  \\ 
  &\geq \inf_{G\in \mathcal{G}_\delta} \mathrm{ess\,sup}_{\P_0}\bigg[\frac{1}{2}\mathrm{div}(a(G+\theta))+
    H(G+\theta)\bigg] - r
  \\
  &=\overline{\Lambda}(\theta)-r, 
  \end{aligned}
\end{equation}
by \eqref{eq:Lambda_def}. Since $r>0$ is arbitrary, we are done with the proof of  Theorem
\ref{thm-lb-geq-ub}. 
\end{proof}

\subsection{Gradients from entropic coercivity: Proof of Proposition \ref{prop-minmax3}.}
\label{sec-prop-minmax3}
We divide the proof of Proposition \ref{prop-minmax3} into the following three lemmas: 

\begin{lemma}\label{lemma1-lemma-minmax3}
For any given $\eps>0$, there exists a sequence $\eps_n\to 0$ and a sequence of functions $(g_n)_n\subset \mathcal D$ so that \eqref{eq-minmax3} holds, and 
$G_n(\omega):=\nabla g_n(\omega)$ converges weakly in $L^{1+\delta}(\P_0)$ (with $\delta>0$ as in \eqref{eq:xi-mom-bound}) and in distribution along a subsequence 
to some random variable $G\in L^{1+\delta}(\P_0)$. 
\end{lemma}

\begin{lemma}\label{lemma2-lemma-minmax3}
The limit $G$ of $G_n$ from Lemma \ref{lemma1-lemma-minmax3} satisfies the closed loop condition defined in \eqref{eq:closed}, i.e., for any simple closed path $\mathrm C$ contained in the infinite cluster $\mathcal C_\infty$, we have $\int_{\mathrm C} G(\cdot,\omega)\cdot \d r=0$, almost surely w.r.t. $\P_0$. 
\end{lemma}

\begin{lemma}\label{lemma3-lemma-minmax3}
The limit $G$ of $G_n$ from Lemma \ref{lemma1-lemma-minmax3} belongs to the class $\mathcal G_{\delta}$ from Definition \ref{def-class-G}.
\end{lemma}

\medskip 

\noindent We will now prove the three lemmas stated above.

 \begin{proof}[{\bf Proof of Lemma \ref{lemma1-lemma-minmax3}}]
 We start with the bound \eqref{eq-minmax2} in Lemma \ref{lemma-minmax2} which implies that there exist sequences $\eps_n\to 0$ and $(g_n)_n\subset \mathcal{D}$ satisfying \begin{equation}\label{eq1-lemma1}
 	\eps_n \log \E_0\big[ \e^{\eps_n^{-1} \mathcal{S}_{\theta}(g_n,\cdot)}\big] \leq \overline H(\theta).
 \end{equation}
 Using this we will first show that
 \begin{equation}\label{eq-alpha}
 \sup_n\|G_n\|_{L^{1+\delta}(\Omega_0)}<\infty.
 \end{equation}
 In particular, the above bound will imply that $G_n$ converges weakly in $L^{1+\delta}(\P_0)$ along a subsequence to some $G$. 
 
We now prove \eqref{eq-alpha}. Note that the map $\lambda\in [0,\infty)\mapsto \frac{1}{\lambda}\log \E_0[\e^{\lambda\mathcal{S}_{\theta}(g_n,\cdot)}]$ is increasing. Thus, recalling the definition of $\mathcal{S}_\theta$ from \eqref{eq:S-def} and using \eqref{eq1-lemma1}, we obtain
 that  for $n$ large enough, 
  \begin{equation*}
 	\log \E_0\left[\e^{R_\theta g_n(\omega)+H(\theta+\nabla g_n(\omega),\omega))}\right]\leq  \overline H(\theta).
 \end{equation*}
The lower bound on $H(\cdot,\omega)$ from \ref{f2} implies 
\begin{equation*}
	\log \E_0\left[\e^{R_\theta g_n(\omega)+c_6\|\theta+\nabla g_n\|_a^{\alpha}-c_7}\right]\leq \overline H(\theta). 
\end{equation*} 
Set $G_n:=\nabla g_n$. Then Jensen's inequality applied to the bound bound and the definition of $R_\theta g_n=\frac 12 \mathrm{div}(a(\nabla g_n+ \theta))$ from \eqref{def-R-theta} leads to 
 \begin{equation*}
	\E_0\left[\frac{1}{2}\mathrm{div}(a(G_n+\theta))+c_6\|\theta+G_n\|_a^\alpha\right]\leq \overline H(\theta) + c_7.
\end{equation*}
Since $G_n=\nabla g_n$, we have $\E_0\left[\mathrm{div}(aG_n)\right]=0$. Thus by \ref{f1} we conclude that for some constant $C=C(\theta,\eta)$, 
$$
\sup_n \E_0[\|G_n\|_a^{\alpha}]\leq C, \qquad\qquad \alpha>1.
$$
But by \eqref{eq:ellip-bounds},
\begin{equation*}
	\|G_n\|_a^\alpha=\langle a(\omega),G_n,G_n\rangle^{\alpha/2}\geq \xi(\omega)^{\alpha/2}|G_n|^\alpha.
\end{equation*}
Combining the last two displays, we have 
\begin{equation}\label{boundGna} 
\sup_n \E_0\big[\xi(\omega)^{\alpha/2}|G_n|^\alpha\big] \leq C.
\end{equation}
Hence, 
\begin{equation}\label{estimate-alpha-delta}
\begin{aligned}
	\E_0[|G_n|^{1+\delta}]&=\E_0\Big[|G_n|^{1+\delta}\xi(\omega)^{\frac{1+\delta}{2}}\xi(\omega)^{-\frac{1+\delta}{2}}\Big]\\
	&\leq \E_0\Big[|G_n|^\alpha\xi(\omega)^{\alpha/2}\Big]^{\frac{1+\delta}{\alpha}}\E_0\bigg[\xi(\omega)^{-\frac{\alpha(1+\delta)}{2(\alpha-1-\delta)}}\bigg]^{\frac{\alpha-1-\delta}{\alpha}}<\infty.
\end{aligned}
\end{equation}
In the first upper bound we used H\"{o}lder's inequality with exponents $\frac{\alpha}{1+\delta}>1$ (recall that $\alpha>1+\delta$) and $\frac{\alpha}{\alpha-1-\delta}$, 
and for the second bound we invoked \eqref{boundGna} and \eqref{eq:xi-mom-bound} with $\chi= \frac\alpha 2 \frac{1+\delta}{\alpha-(1+\delta)}$. Hence, 
 \begin{equation*}
	\sup_n \E_0[|G_n|^{1+\delta}]<\infty,
\end{equation*}
 with $\delta>0$. Consequently, $G_n$ converges weakly in $L^{1+\delta}(\P_0)$ and in distribution along a subsequence to some random variable $G\in L^{1+\delta}(\P_0)$, as claimed. 
\end{proof}

\medskip

\begin{proof}[{\bf Proof of Lemma \ref{lemma2-lemma-minmax3}}.]
We first assert that it suffices to prove that for any measurable set $A\subset \Omega_0$, \begin{equation}\label{eq:closed_loop_to_show}
	\E_0\bigg[\1_{A\cap (\mathrm C\subset{\mathcal C_\infty}) }\int_{\mathrm C} G(\cdot,\omega)\cdot \d r\bigg]=0~\text{ for each simple closed path }\mathrm C\subset \mathcal C_\infty.
\end{equation}
Indeed, since \eqref{eq:closed_loop_to_show} holds for any arbitrary measurable set $A\subset\Omega_0$, it in particular implies  that, for any fixed simple closed path $\mathrm C \subset \mathcal C_\infty$, $\P_0$-a.s., $\1_{C\subset {\mathcal C_\infty}}\int_C G(\cdot,\omega)\cdot \d r=0$. We want to show that this holds $\P_0$-a.s. uniformly on each closed loop. Since line integrals are independent of the parametrization of the path, for each $\mathrm C$, we choose any (but fixed from now) smooth function 
$$
f_{\mathrm C}:[0,1]\to \R^d\qquad\mbox{ satisfying }\,\, f(0)=f(1).
$$ The space 
$$
X:=\big\{f\in C^{\infty}[0,1]:f(0)=f(1)\big\}
$$
 is separable under the $|\cdot|_\infty$ norm, so there exists some countable dense subset $Y\subset X$. If \eqref{eq:closed_loop_to_show} holds, we can show that $\P_0$-a-s., the closed loop condition holds for each curve $\mathrm C$ such that $f_{\mathrm C}\in Y$. To extend this to all simple closed curves in $\R^d$, we can approximate each curve by a sequence $\mathrm C_n$ such that $f_{\mathrm C_n}\in Y$. Since the convergence is uniform, it is easy to deduce that $\P_0$-a.s., $\1_{\mathrm C\subset {\mathcal C_\infty}}\int_{\mathrm C} G(\cdot,\omega)\cdot r=0$ for any simple closed curve $\mathrm C$. Thus, we only need to show that \eqref{eq:closed_loop_to_show} holds for fixed $A\subset \Omega_0$ and simple closed curve $\mathrm C$.

Let $f:[0,1]\to \R^d$ be any smooth function that parametrizes $\mathrm C$. For each fixed $n\in \N$, we know that $G_n=\nabla g_n$ satisfies the closed loop condition (because it is a gradient). By Fubini's theorem we have 
\begin{align*}
	0=\E_0\bigg[\1_{A\cap (\mathrm C\subset{\mathcal C_\infty}) }\int_{\mathrm C} G_n(\cdot,\omega)\cdot \d r\bigg]&=\E_0\bigg[\1_{A\cap (\mathrm C\subset{\mathcal C_\infty}) }\int_{0}^1 G_n(f(x),\omega)\cdot f^\prime(x)\d x\bigg]\\
	&=\int_{0}^1f^\prime(x)\cdot \E_0\bigg[\1_{A\cap (\mathrm C\subset {\mathcal C_\infty})}G_n(f(x),\omega)\bigg]\d x.
\end{align*}
Since $G_n$ converges weakly to $G$ in $L^{1+\delta}(\P_0)$ (as shown in Lemma \ref{lemma1-lemma-minmax3}), for fixed $x\in [0,1]$,\begin{equation*}
	\lim_{n\to\infty}\E_0\left[\1_{A\cap (\mathrm C\subset {\mathcal C_\infty})}G_n(f(x),\omega)\right]=\E_0\left[\1_{A\cap (\mathrm C\subset {\mathcal C_\infty})}G(f(x),\omega)\right].
\end{equation*}
Using that $\sup_n\E_0[G_n^{1+\delta}]<\infty$, and that $f'$ is bounded on
$[0,1]$, we can apply dominated convergence theorem to conclude that 
\begin{equation*}
  0=\lim_{n\to\infty}\E_0\left[\1_{A\cap (\mathrm C\subset {\mathcal C_\infty})}G_n(f(x),\omega)\right]=\int_{0}^{1}f'(x)\cdot \E_0\left[\1_{A\cap (\mathrm C\subset {\mathcal C_\infty})}G(f(x),\omega)\right]\d x.
\end{equation*} 
As $G\in L^{1+\delta}(\P_0)$, we can again exchange the order of integration
using Fubini's theorem so the right-hand side in the last display is
$\E_0\left[\1_{A\cap (\mathrm C\subset{\mathcal C_\infty}) }\int_{\mathrm C}
  G(\cdot,\omega)\cdot \d r\right]$. This shows \eqref{eq:closed_loop_to_show}
and concludes the proof of the lemma.
\end{proof}

To complete the proof of Proposition \ref{prop-minmax3}, it remains to prove Lemma \ref{lemma3-lemma-minmax3}. Its proof will require two further technical estimates, Lemma \ref{lemma 4.4} and Proposition \ref{zero expectation}, stated below. Using these two facts, we will complete the proof of Lemma \ref{lemma3-lemma-minmax3}, and therefore that of Proposition \ref{prop-minmax3}.

\begin{lemma}
	\label{lemma 4.4}
	Let $\ell=\ell(\omega)=\d_\omega(0,\mathfrak v_e(\omega))$ be the graph distance between $0$ and $\mathfrak v_e=n(\omega,e)e$. Then there exist constants $a$, $C>0$ such that for any $t>0$,
	\begin{equation}\label{eq:sup_prob-exp-decay}
		\P_0\bigg(\sup_{0\leq s\leq n_{1}(\omega)}\1\{se\in {\mathcal C_\infty}(\omega)\}\d_{\omega}(0,se)>t\bigg)\leq C \e^{-at}.
	\end{equation}
In particular,
	$\P_0(\ell>t)<C\e^{-at}$. 
\end{lemma}
\begin{proof}

Let $\eps>0$. For $t>0$ we write $t_\eps:=\lfloor \eps t\rfloor$. Then
\begin{align*}
		\P_0\left(\sup_{0\leq s\leq n_{1}}\1\{se\in {\mathcal C_\infty}(\omega)\}\d_{\omega}(0,se)>t\right)&\leq \P_0\left(n_1(\omega)\geq t\eps\right)+\P_0\left(\sup_{0\leq s\leq t\eps}\1\{se\in {\mathcal C_\infty}(\omega)\} \d_{\omega}(0,se)>t\right).
	\end{align*}
	By \ref{assump:exp-dec-dist-indshift}, the claim follows once we prove that the second term goes to zero at an exponential rate. This probability is bounded above by $\sum_{i=1}^{\lfloor t\eps \rfloor}\P_0\left(\sup_{i-1\leq s\leq i}\1\{se\in {\mathcal C_\infty}(\omega)\}\d_{\omega}(0,se)>t\right)$. Since the number of summands is growing only polynomially in $t$, it suffices to show that each summand there decays exponentially in $t$. We will proceed as follows:
	
	We define 
	$$m:=\min\{l\in\N:l>t_\eps,\,-le\in{\mathcal C_\infty}\}, \quad %\mathfrak w_e=-me,\quad\mbox{and } 
	A_{x,y}=\{\d_\omega(x,y)\geq t/2,\,x,y\in{\mathcal C_\infty}\}.
	$$
	Now we observe that on the event $\{\sup_{i-1\leq s\leq i}\1\{se\in {\mathcal C_\infty}(\omega)\}\d_{\omega}(0,se)>t\}$, one of the following cases must hold:
	
	\begin{itemize}
		\item $m>2t_\eps$, or 
		\item at least one of the points $le$ with $l\in\mathbb Z$ and $|l|\leq 2t_\eps$ is in ${\mathcal C_\infty}$ and for some $i-1\leq s\leq i$, $\max\{\d_\omega(0,-le), \d_\omega(-le,se)\}\geq t/2$.
	\end{itemize} 
	In the first of the two cases above we have $|\mathfrak v_{-e}\circ\sigma_{-e}^m|>t_\eps$ for at least one $m=1,...,t_\eps$. Hence,
 
	\begin{align*}
		\P_0\big(\sup_{i-1\leq s\leq i}\1\big\{se\in {\mathcal C_\infty}(\omega)\big\}\d_{\omega}(0,se)>t\big)&\leq \sum_{m=1}^{t_\eps}\P_0\big(\sigma_{-e}^m(\{|\mathfrak v_{-e}|\geq t_\eps\})\big) \\
        &\quad +\sum_{\ell =t_\eps}^{2t_\eps}\P_0\big(\exists i-1\leq s\leq i: A_{0,-le}\cup A_{-le,se}\big).\\
	\end{align*}
	By \ref{assump:exp-dec-dist-indshift}, the probabilities of the events in the first sum  are equal and exponentially small.
	 The second sum is bounded by \begin{equation*}
		t_\eps\P_0(A_{0,-le})+\sum_{\ell =t_\eps}^{2t_\eps}\P_0\left(\exists i-1\leq s\leq i:  A_{-le,se}\right).
	\end{equation*}
	To bound the first term, we use \eqref{eq:multidim-campbell-eq} and \ref{assump:chem-dist} to obtain the bound \begin{align*}
		\P_0(A_{0,-le})&\leq \frac{1}{\P(0\in {\mathcal C_\infty})}\P\bigg(\exists x\neq y\in \mathcal C_\infty(\omega): |x|\leq \frac{1}{2}, |y-le|\leq \frac{1}{2}, \d_\omega(x,y)\geq \frac{t}{2}+1;0,x,y\in {\mathcal C_\infty}\bigg)\\
		&\leq \frac{1}{\P(0\in {\mathcal C_\infty})}\E\bigg[\sum_{x,y\in \omega}^{\neq}\1\bigg\{|x|\leq 1/2, |y-le|\leq 1/2, \d_\omega(x,y)\geq t/2+1;0,x,y\in {\mathcal C_\infty}\bigg\}\bigg]\\
		&=\frac{\zeta^2}{\P(0\in {\mathcal C_\infty})}\int_{[-1/2,1/2]^d}\int_{[le-1/2,le+1/2]^d}\P^{x,y}\bigg( \d_\omega(x,y)\geq \frac{t}{2}+1;0,x,y\in {\mathcal C_\infty}\bigg)\d x\d y\\
		&\leq C\e^{-C't_\eps}
	\end{align*}
	for some constants $C,C'>0$ which are independent of $l$, and for $\eps>0$ small enough, with $\zeta$ defined in \ref{assump:intensity}. 
	Following the same calculations as in the last display, we can also show that \begin{equation*}
		\P_0\left(\exists i-1\leq s\leq i:  A_{-le,se}\right)\leq C\e^{-C't_\eps}
	\end{equation*}
	for constants $C,C'>0$ independent of $i$ and $l$, for $\eps>0$ small enough.
	After estimating the probabilities of all events by an exponential upper bound, from the unions we get another factor that is linear in $t$, which can be absorbed by the exponential bound for $t$ large enough. Thus the proof of Lemma \ref{lemma 4.4} is complete. 
	
\end{proof}

\begin{prop}
\label{zero expectation}
For any unit coordinate vector $e$, recall that we denote the first successive
arrival in direction $e$ by $\mathfrak v_e=\mathfrak
v_e(\omega)=n(\omega,e)e$. Then for any $G\in \mathcal{G}_\delta$, we have
$\bE_0|V_G(\mathfrak v_e,\cdot)|<\infty$. More precisely, there is a constant
$C=C(d,\delta,\P_0)$ such that for any $G\in \mathcal{G}_\delta$,
$\E_0|V_G(\mathfrak v_e,\cdot)|\leq C\|G\|_{L^{1+\delta}(\P_0)}$.
\end{prop}

\begin{proof}
Let $\mathscr{B}=\mathscr{B}(\omega)$ be an enumeration of the balls that appear in the construction of $\mathscr{C}(\omega)$ (recall \ref{assump:inf-comp}). Then define the random variable \begin{equation}\label{eq:altern_graph_dis_def}
\begin{aligned}
	\tilde{\d}_\omega(x,y):=&\min\bigg\{n\in\N: \text{ $\exists$  $(B_i)_{i=1}^n\subset \mathscr{B}$ such that $x\in B_1, y\in B_n$, and}\\
	&\qquad\qquad\text{$B_{i-1}\cap B_{i}\neq \emptyset~\forall~ 1\leq i\leq n$ }\bigg\}
	\end{aligned}
\end{equation}
and set \begin{equation}\label{eq:ell-tilde-def}
 \tilde{\ell}:=\tilde{\d}_\omega(0,\mathfrak{v}_e).	
 \end{equation}
 Note that there is some constant $c>0$ such that for all $n>0$, \begin{equation}\label{eq:ell-tilde-exp-mom}
	\P_0\left(\tilde{\ell}>n\right)\leq \e^{-cn}.
\end{equation}
 For $j\in \N$, let $N_j:=\Z/2\cap [-j,j]^d$. We consider this set as a graph, where for $x,y\in N_j$, $x\sim y$ iff $|x-y|_1=\frac{1}{2}$. Note that if $\tilde{\ell}=j$, then there is a 
 nearest-neighbor path on $N_j$ of length $k\leq 3^d j$ such that for all $1\leq i\leq k-1$, the line segment between $x_{i}$ and $x_{i+1}$ is contained in the cluster. Thus, we write \begin{equation*}
	\{\tilde{\ell}=j\}=\bigcup_{k=1}^{3^dj}\bigcup_{x_1,\cdots,x_k}\bigg\{\tilde{\ell}=j\cap A(x_1,\cdots,x_k)\bigg\},
\end{equation*}
where \begin{equation}\label{eq:A-set-def}
	\begin{aligned}
		A(x_1,\cdots,x_k):=&\bigg\{x_1,\cdots,x_k\text{ is a nearest-neighbor path on } N_j,\\&\qquad  \forall 1\leq i\leq k-1 \text{ the line segment between }x_{i-1}, x_i\text{ is inside } {\mathcal C_\infty}\bigg\}.
	\end{aligned}
\end{equation}
If $\tilde{\ell}=j$, then one can write for some nearest neighbor path $0=x_0,x_1,\cdots,x_k$ on $N_j$ ($1\leq k\leq 3^d j$) such that the line segment between $x_{i}$ and $x_{i+1}$ is inside $\mathcal C_\infty$ for all $0\leq i\leq k-1$,\begin{align*}
	\int_{0\leadsto \mathfrak{v}_e}|G(\cdot,\omega)|\cdot \d\mathbf{r}&\leq \sum_{i=0}^{k-1}\int_{0}^1|G(t(x_{i+1}-x_i),\tau_{x_i}\omega)|\d t\\
	&\leq 2\sum_{i=0}^{k-1}\sum_{|e|=1}\int_0^{1/2} |G(te,\tau_{x_i}\omega)|\d t
	 \,\, \leq\,\,  2(3^d j)\sum_{x\in \mathcal{C}_\infty\cap N_j:|x|\leq 3^dj}\sum_{|e|=1}\int_0^{1/2} |G(te,\tau_{x}\omega)|\d t.
\end{align*}
Therefore, \begin{equation}\label{eq:eq7}
	\begin{aligned}
		\E_0|V_G(\mathfrak v_e,\cdot)|&=\sum_{j=1}^{\infty}\E_0\bigg[\bigg|\int_{0\leadsto \mathfrak{v}_e}G(\cdot,\omega)\cdot \d\mathbf{r}\bigg|,\tilde{\ell}=j\bigg]\\
	&\leq 2\sum_{j=1}^{\infty}\sum_{|e|=1}\sum_{x\in N_j:|x|\leq 3^d j}(3^d j)\int_0^{1/2}\E_0\bigg[|G(te,\tau_{x}\omega)|,x\leadsto x+e \subset \mathcal{C}_\infty,\tilde{\ell}=j\bigg]\d t\\
	&\leq 2\sum_{j=1}^{\infty}\sum_{|e|=1}\sum_{x\in N_j:|x|\leq 3^d j}(3^d j)\int_0^{1/2}\E_0\bigg[|G(te,\tau_{x}\omega)|^{1+\delta},x\leadsto x+e \subset \mathcal{C}_\infty\bigg]^{1/(1+\delta)}\P_0(\tilde{\ell}=j)^{\frac{\delta}{1+\delta}}\d t.
	\end{aligned}
\end{equation}
	Since $G\in \mathcal{G}_\delta$, then for any $x\in \R^d$,
	\begin{equation*}
		\E_0\left[|G(x,\omega)|^{1+\delta},x\in  \mathcal{C}_\infty\right]\leq \|G\|_{L^{1+\delta}(\P_0)}.
	\end{equation*}
 As a consequence, \eqref{eq:eq7} can be bounded by 	
 \begin{equation*}
	C(d)\|G\|_{L^{1+\delta}(\P_0)}\sum_{j=1}^{\infty}j^2\P_0(\tilde{\ell}=j)^{\frac{\delta}{1+\delta}}\leq C(d,\delta,\P_0)\|G\|_{L^{1+\delta}(\P_0)}
	\end{equation*}
	due to \eqref{eq:ell-tilde-exp-mom}. This finishes the proof of the proposition.

\end{proof}
Now we are ready to complete the

\noindent{\bf Proof of Lemma \ref{lemma3-lemma-minmax3}.} We have already proved \eqref{eq:unifbound-L-alpha}. On the other hand, the proof of \eqref{eq:unifbound-div} follows from the first inequality in \eqref{eq:eq17} (note that for this part we are only using the weak convergence of $G_n$ towards $G$ in $L^{1+\delta}(\P_0)$ and in distribution, which have been established in Lemma \ref{lemma1-lemma-minmax3}). Also, the closed loop property \eqref{eq:closed} was shown in Lemma \ref{lemma2-lemma-minmax3}. Thus it remain to check that $G$ satisfies the zero induced mean property $\E_0[ V_G(\mathfrak v_e,\cdot)]=0$, recall \eqref{eq:meanzero}.

Let us fix a coordinate unit vector $e$, and recall the definitions of $\mathfrak{v}_e=n(\omega,e)e$ from \eqref{def-n}, that of $\tilde{\ell}(\omega)=\tilde{\d}_\omega(0,\mathfrak{v}_e)$  from \eqref{eq:ell-tilde-def} and of the sets $A(x_1,\cdots,x_k)$ from \eqref{eq:A-set-def}.
Choose $\tilde{A}(x_1,\cdots,x_k)\subset A(x_1,\cdots,x_k)$ so that 
\begin{equation*}
	\{\tilde{\ell}=j\}=\bigsqcup_{k=1}^{3^dj}\bigsqcup_{x_1,\cdots,x_k}\{\tilde{\ell}=j\cap \tilde{A}(x_1,\cdots,x_k)\},
\end{equation*} 
where $\bigsqcup$ represents disjoint union. Next, for any $R>0$ 
define  \begin{equation}\label{eq:FM_def}
	\eta_R:=\E_0[V_{G}(\mathfrak{v}_e,\omega),\tilde{\ell} \leq R].
\end{equation}
 By dominated convergence theorem, the required identity $\E_0[V_g(\mathfrak v_e,\omega)]=0$ follows 
 once we show  that $\eta_R\to 0$ as $R\to \infty$. For this purpose, we further claim that
 \begin{equation}\label{eq:mean_zero_to_show}
 \begin{aligned}
	\eta_R =\lim_{n\to\infty}\E_0\bigg[\int_{0\leadsto \mathfrak{v}_e}G_{n}(\cdot,\omega)\cdot \d \mathbf{r},\tilde{\ell} \leq R\bigg] 
	=-\lim_{n\to\infty}\E_0\bigg[\int_{0\leadsto \mathfrak{v}_e}G_{n}(\cdot,\omega)\cdot \d \mathbf{r},\tilde{\ell} > R\bigg].
\end{aligned}
\end{equation} 
We observe that the second equality in \eqref{eq:mean_zero_to_show} follows from the fact that 
$$
\E_0\bigg[\int_{0\leadsto \mathfrak{v}_e}G_{n}(\cdot,\omega)\cdot \d \mathbf{r}\bigg]=\E_0\big[g_{n}(\sigma_e \omega)-g_{n}(\omega)\big]=0,
$$ 
because $\sigma_e$ is measure-preserving under $\P_0$ (recall Proposition \ref{prop 1}) and for each fixed $n$, $g_{n}$ is bounded and continuous. Thus, the only nontrivial claim is the first equality in \eqref{eq:mean_zero_to_show}. We decompose $\eta_R$ as 
(below, $x_0:=0$)
\begin{align*}
	\eta_R&=\sum_{j=1}^{R}\sum_{k=1}^{3^dj}\sum_{x_1,\cdots, x_k\in N_j}\E_0\bigg[\int_{0\leadsto \mathfrak{v}_e}G(\cdot,\omega)\cdot \d \mathbf{r}, \tilde{\ell}=j,\tilde{A}(x_1,\cdots, x_k)\bigg]\\
	&=\sum_{j=1}^{R}\sum_{k=1}^{3^dj}\sum_{x_1,\cdots, x_k\in N_j}\sum_{i=1}^{k}\E_0\bigg[\int_{x_{i-1}\leadsto x_i}G(\cdot,\omega)\cdot \d \mathbf{r},\tilde{\ell}=j, \tilde{A}(x_1,\cdots,x_k)\bigg].
\end{align*}
On $\tilde{A}(x_1,\cdots, x_j)$, we can always choose the straight line between these two points as a curve. Using that $G_{n}$ converges to $G$ weakly in $L^{1+\delta}(\P_0)$ (cf. Lemma \ref{lemma1-lemma-minmax3}) 
we deduce that \begin{equation*}
	\lim_{n\to\infty}\E_0\bigg[\int_{x_{i-1}\leadsto x_i}G_{n}(\cdot,\omega)\cdot \d \mathbf{r},\tilde{\ell}=j,\tilde{A}(x_1,\cdots,x_j)\bigg]
=\E_0\bigg[\int_{x_{i-1}\leadsto x_i}G(\cdot,\omega)\cdot \d \mathbf{r},\tilde{\ell}=j, \tilde{A}(x_1,\cdots,x_j)\bigg].
\end{equation*}
Therefore, \begin{equation*}
	\eta_R=\sum_{j=1}^{R}\sum_{k=1}^{3^dj}\sum_{x_1,\cdots, x_k\in N_j}\sum_{i=1}^{k}\lim_{n\to\infty}\E_0\bigg[\int_{x_{i-1}\leadsto x_i}G_{n}(\cdot,\omega)\cdot \d \mathbf{r},\tilde{\ell}=j, \tilde{A}(x_1,\cdots,x_j)\bigg].
\end{equation*}
Finally, we can exchange the limit with the sum over $x_1,\cdots, x_k$ by noting that \begin{equation*}
	\E_0\bigg[\int_{x_{i-1}\leadsto x_i}G_{n}(\cdot,\omega)\cdot \d \mathbf{r}, \tilde{\ell}=j, \tilde{A}(x_1,\cdots,x_j)\bigg]
\end{equation*} is uniformly bounded because $\sup_n \|G_{n}\|_{L^{1+\delta}(\P_0)}<\infty$. This shows \eqref{eq:mean_zero_to_show}. To conclude proving that $\eta_R\to 0$ as $R\to\infty$, we use \eqref{eq:mean_zero_to_show} to estimate $|\eta_R|$ as \begin{align*}
	|\eta_R|&=\lim_{n\to\infty}\bigg\lvert\E_0\bigg[\int_{0\leadsto \mathfrak{v}_e}G_{n}(\cdot,\omega)\cdot \d \mathbf{r},\tilde{\ell} > R\bigg]\bigg\rvert\leq \limsup_{n\to\infty}\sum_{j=R}^{\infty}\E_0\bigg[\int_{0\leadsto \mathfrak{v}_e}|G_{n}(\cdot,\omega)|\cdot \d \mathbf{r},\tilde{\ell} =j\bigg].\\
	\end{align*}
Now, following the arguments exactly as in the proof of Proposition \ref{zero expectation} and using that $\sup_n\|G_n\|_{L^{1+\delta}(\P_0)}<\infty$, we can show that the last display is bounded above by $C_1\e^{-C_2 R}$ for some constants $C_1,C_2>0$, implying that $|\eta_R|\to 0$, which in turn completes the proof that $G$ satisfies the induced mean zero property. Thus Lemma \ref{lemma3-lemma-minmax3} and therefore Proposition \ref{prop-minmax3} are proved. 
\qed

\section{Upper bound.} \label{sec-ub}

In this section we will prove the equivalence of $\overline H$ and $\overline\Lambda$ (see Theorem \ref{thm-equivalence} below) and conclude the proof of Theorem \ref{thm} in Section \ref{subsec ProofThmLDP}.  To this end, we recall the two variational formula 
$$
\begin{aligned}
&\overline{H}(\theta)=
\sup_{(b,\phi)\in \mathcal{E}} \bigg(\int \phi d\P_0\left[\frac{1}{2}\mathrm{div}(a\theta)+\langle \theta,b\rangle_a- L(b,\omega))\right],\\
& \overline{\Lambda}(\theta):=\inf_{G\in \mathcal{G}_{\delta}}\left(\mathrm{ess\,sup}_{\P_0}\left[\frac{1}{2}\mathrm{div}(a(G+\theta))+ H(G+\theta)\right]\right).
\end{aligned}
$$
and note that in Theorem \ref{thm-lb-geq-ub} we already proved the lower bound $\overline H(\cdot) \geq \overline\Lambda(\cdot)$. The reversed inequality 
\begin{equation}\label{eq-ub-geq-lb}
\overline\Lambda(\cdot) \geq \overline H(\cdot) 
\end{equation}
will immediately yield 
\begin{theorem}\label{thm-equivalence}
Assume \ref{assump:est-erg}-\ref{assump:fkg}, \ref{f1}, \ref{f1'}, \ref{f2} and \ref{f3}. Then for any $\theta\in\R^d$,
$$
\overline H(\theta)= \overline\Lambda(\theta).
$$
\end{theorem}

The inequality \eqref{eq-ub-geq-lb} will follow from the proposition below:

\begin{prop}
\label{prop: up-bound-lin-dat}
Assume \ref{assump:est-erg}-\ref{assump:fkg}, \ref{f1} and \ref{f2}. Let $u_{\eps,\theta}$ be the variational representation \eqref{eq:u-eps-def}  with initial condition $f(x)=\langle \theta,x\rangle$.
Then 
\begin{equation}
 \label{eq:up-bound-lin-dat}
  \limsup_{\eps\to  0}u_{\eps,\theta}(t,0,\omega)\ \leq\  
 t\, \overline{\Lambda}(\theta)\quad  \P_0\text{-}a.s.,
\end{equation}
where $\overline\Lambda$ is defined in \eqref{eq:Lambda_def}.
\end{prop}

Assuming the above fact, we can conclude

\begin{proof}[{\bf Proof of Theorem \ref{thm-equivalence} (assuming Proposition \ref{prop: up-bound-lin-dat}):}]
Combining Lemma \ref{lemma:lower-bound-lemma1}  for the particular case $f(x)=\langle\theta,x\rangle$ with 
Proposition \ref{prop: up-bound-lin-dat} we conclude that $\overline{H}(\theta)\leq \overline{\Lambda}(\theta)$. The reversed bound has been already shown in Theorem \ref{thm-lb-geq-ub}, which proves Theorem \ref{thm-equivalence}.
\end{proof}

\subsection{Proof of Proposition \ref{prop: up-bound-lin-dat}}\label{subsec proof ub lin f}
	Let us first sketch the main idea of the proof, for which we will apply Theorem \ref{thm:sublinear}. To simplify notation, for a fixed $\theta\in \R^d$,  we will simply write 
	$$
	u_{\eps}(t,x,\omega)= u_{\eps,\theta}(t,x,\omega).
	$$
	We now recall Lemma \ref{lemma:control_restriction}, which implies that for a fixed $t>0$,  
	\begin{equation}\label{ub-0}
	u_\eps(t,0,\omega)=	\eps \sup_{c\in \mathbf C_T^*} E^{P_{0}^{c,\omega}}\bigg[\langle \theta,X_{t/\eps}\rangle-\int_0^{t/\eps}L(X_s,c(s))\d s\bigg].
\end{equation}
with the class $\mathbf C^\star_T$ also defined in Lemma \ref{lemma:control_restriction}. Next, let us fix any $G\in \mathcal{G}_{\delta}$ as defined in \eqref{def-class-G}, with $V_G(x,\omega):=\int_{0\leadsto x} G(\cdot,\omega)\,\,\cdot\d\mathbf r$  as defined in \eqref{def-V}, and set 
\begin{equation}\label{hG}
h_G(x)= h_G(x,\omega):= \langle \theta,x\rangle+ V_G(x,\omega)
\end{equation}
 for a fixed $\omega\in \Omega_0$. If $V_G$ were smooth enough,  $\nabla h_G=\theta+G$ and by It\^{o}'s formula applied to the function $h_G(x)$, 
 \begin{equation}\label{Ito hG}
 \begin{aligned}
	\langle \theta,X_{t/\eps}\rangle + V_G(X_{t/\eps},\omega)&=\int_0^{t/\eps}(\theta+G(X_s))\sigma(X_s)\d\mathcal{B}_s+\frac{1}{2}\int_{0}^{t/\eps}\mathrm{div}(a(\theta+G))(X_s)\d s\\&+\int_0^{t/\eps}\langle c(s),\theta+G(X_s)\rangle_a \d s.
\end{aligned}
\end{equation}
 For any fixed $c\in \mathbf C^\star_T$, we subtract $\int_0^{t/\eps} L(X_s, c(s)) \d s$ from both sides in the above display and take expectation w.r.t. $P^{c,\omega}_0$ to obtain
  \begin{equation}\label{eq:eq35}
 \begin{aligned}
	&E^{P_{0}^{c,\omega}}\bigg[\langle \theta,X_{t/\eps}\rangle-\int_0^{t/\eps}L(X_s,c(s))\d s\bigg]=-E^{P_0^{c,\omega}}[V_G(X_{t/\eps},\omega)]	\\
	&\qquad +E^{P_0^{c,\omega}}\bigg[\int_0^{t/\eps}\bigg(\frac{1}{2}\mathrm{div}\big(a(X_s)(\theta+G(X_s))\big)+\langle c(s),\theta+G(X_s)\rangle_a-L(X_s,c(s))\bigg)\d s\bigg]\\
	&\leq -E^{P_0^{c,\omega}}\big[V_G(X_{t/\eps},\omega)\big]+ E^{P_0^{c,\omega}}\bigg[\int_0^{t/\eps}\frac{1}{2}\mathrm{div}\big(a(X_s)(\theta+G(X_s))\big)+H(X_s,\theta+G(X_s))\d s\bigg]\\
	&\leq -E^{P_0^{c,\omega}}\big[V_G(X_{t/\eps},\omega)\big]+\frac{t}{\eps}~\mathrm{ess\,sup}_{\P_0}\bigg[\frac{1}{2}\mathrm{div}(a(G+\theta))+ H(G+\theta)\bigg].\end{aligned}
 \end{equation}
 In the first upper bound of the above display we used that, for any $p,q,x \in \R^d$, $[\langle q,p\rangle_a- L(x,q)]\leq \sup_{q\in\R^d}[\langle q,p\rangle_a- L(x,q)]= H(x,p)$. For the second upper bound, we used that we have a function $F(x)=\frac 12 \mathrm{div}\big(a(x)(\theta+ G(x))\big)+H(x,\theta+G(x))$ that satisfies, for every $x$ and $\omega\in \Omega_0$, $F(x)=F(x,\omega)=F(\tau_x\omega)$ and we can consequently use the bound $\sup_{x,\omega} |F(x)|\leq \|F\|_{L^\infty(\P_0)}$.

Thus, combining \eqref{ub-0} and \eqref{eq:eq35} we would have 
 \begin{equation}\label{upbound}
 \begin{aligned}
	u_\eps(t,0,\omega)&\leq -\eps \inf_{c\in \mathbf C^*_T}E^{P_0^{c,\omega}}[V_G(X_{t/\eps},\omega)]+t~\mathrm{ess\,sup}_{\P_0}\left[\frac{1}{2}\mathrm{div}(a(G+\theta))+ H(G+\theta)\right].
	\end{aligned}
    \end{equation}
If $G$ were bounded (i.e., assuming that $G\in \mathcal G_\infty$), we could apply Theorem \ref{thm:sublinear} and deduce that $\P_0$-a.s., for all $r>0$ there exists some $c_r=c_r(\omega)$ such that for all $x\in {\mathcal C_\infty}$, $|V_G(x,\omega)|\leq r|x| + c_r$. Hence, recalling the definition of $\overline\Lambda(\cdot)$ from \eqref{eq:Lambda_def}, we would be led to 
$$
u_\eps(t,0,\omega)\leq t\overline{\Lambda}(\theta)+\eps c_r+ r \sup_{c\in \mathbf C^*_T}E^{P_0^{c,\omega}}[|\eps X_{t/\eps}|].
$$
By  Lemma \ref{lemma:control_restriction} and the inequalities \eqref{eq:ineq_c_norm}-\eqref{eq:ineq_c_norm2}, one can deduce (see \eqref{eq:unif_bound_xt} below for details) that $E^{P_0^{c,\omega}}[|\eps X_{t/\eps}|]$ is uniformly bounded over $0<\eps\leq 1$ and $c\in \mathbf C_T^*$. Thus, one simply let first $\eps\to 0$ and then $r\to 0$ to conclude the proof. 

However, a priori $G\in \mathcal G_{\delta}$ is neither smooth enough nor bounded. Nevertheless, we can mollify $G$ to get a smooth and bounded version, so that we can apply the same reasoning as above. We provide the details now.

Let $\rho$ be any spherically symmetric smooth mollifier having support in the unit ball and such that $\int_{\R^d}\rho(y)\d y=1$. For any $\eta>0$, we set 
\begin{equation}\label{eq:mollified-G}
	G_\eta(\omega):=\int_{\R^d}G(\tau_{\eta y}\omega)\rho(y)\d y, \qquad\mbox{and}\quad G_\eta(x,\omega)= G_\eta(\tau_x\omega)
\end{equation}
to be the convolution. Cleraly, for any $\omega\in \Omega_0$, $G_\eta(\cdot,\omega)\to G (\cdot,\omega)$ pointwise as $\eta\to 0$. Similarly, define 
\begin{equation}\label{def Veta}
V_{\eta,G}(x,\omega):=\int_{\R^d}V_{G}(x+\eta y,\omega)\rho(y)\d y, \quad\mbox{so that}\quad \nabla V_{\eta,G}=G_\eta
\end{equation}
Next, by Young's inequality, 
\begin{equation}\label{G eta bounded}
\|G_\eta\|_{L^\infty(\P_0)}\leq C_\eta <\infty. 
\end{equation} 
As a consequence, for any $x\in \mathcal C_\infty$ and any path $0\leadsto x$ inside $\mathcal C_\infty$,
\begin{equation}\label{eq:eq18} 
	V_{\eta,G}(x,\omega)-V_{\eta,G}(0,\omega)=\int_{0\leadsto x}G_\eta(\omega,r)\cdot \d \mathbf{r},
\end{equation}
so that the line integral is independent of the path $0\leadsto x$. Moreover, by Proposition \ref{zero expectation}, it also holds that 
\begin{equation}\label{Veta V}
\E_0\big[\big|V_{\eta,G}(\mathfrak v_e,\cdot)- V_G(\mathfrak v_e,\cdot)\big|\big]\to 0 \qquad\mbox{as}\,\, \eta\to 0. 
\end{equation}
Then, the constant vector $\xi_\eta:= \E_0[G- G_\eta]\in \R^d$ satisfies $|\xi_\eta|\to 0$ as $\eta\to 0$ and \eqref{def Veta}, \eqref{G eta bounded} and \eqref{Veta V} imply that we can replace $G_\eta$ by $\hat G_\eta:= G_\eta-\xi_\eta$ to obtain a smooth, bounded element $\hat G_\eta\in \mathcal{G}_\infty$ such that, together with \eqref{eq:eq18}, Theorem \ref{thm:sublinear} applies for each fixed $\eta$ to the function $x\mapsto V_{\eta,G}(x,\omega)- V_{\eta,G}(0,\omega)- \langle \xi_\eta,x\rangle$. Similarly, in \eqref{hG}, we can replace $\theta$ by $\theta- \xi_\eta$ and apply Ito's formula to the function $h_{\hat G_\eta}(x)= \langle \theta - \xi_\eta, x\rangle + V_{\eta,G}(x,\omega)$ and proceed as in \eqref{Ito hG} and \eqref{eq:eq35} to get 
 \begin{align*}
 	&E^{P_{0}^{c,\omega}}\left[\langle \theta,\eps X_{t/\eps}\rangle-\eps\int_0^{t/\eps}L(X_s,c(s))\d s\right]\\
 	&\leq -\eps E^{P_0^{c,\omega}}\big[V_{\eta,G}(X_{t/\eps},\omega)-V_{\eta,G}(0,\omega)-\langle \xi_\eta,X_{t/\eps}\rangle\big]\\
 	&\qquad+\eps E^{P_0^{c,\omega}}\left[\int_0^{t/\eps}\frac{1}{2}\mathrm{div}\left(a(X_s))(\theta+G_\eta(X_s)-\xi_\eta\right)+H(X_s,\theta+G_\eta(X_s)-\xi_\eta)\d s\right].
 \end{align*}
The first term in the upper bound above can be can be bounded by applying Theorem \ref{thm:sublinear} for each fixed $\eta$ to the function $x\mapsto V_{\eta,G}(x,\omega)- V_{\eta,G}(0,\omega)- \langle \xi_\eta,x\rangle$ and by proceeding exactly as described in the paragraph below \eqref{upbound}. To handle the second and third expectations, we use the convexity of $H$ and Jensen's inequality (recall that $\int_{\R^d} \rho=1$) and subsequently invoke the $L^\infty(\P_0)$ bound as explained above \eqref{upbound}. This leads to   
\begin{equation}\label{eq:eq16}
\begin{aligned}
	&\eps E^{P_0^{c,\omega}}\bigg[\int_0^{t/\eps}\int_{\R^d}\frac{1}{2}\bigg[\mathrm{div}(a(X_s))(\theta+G(X_s+\eta y)-\xi_\eta)\bigg]\rho(y)\d y\d s\bigg] \\
	&\qquad +\eps E^{P_0^{c,\omega}}\bigg[\int_0^{t/\eps}\int_{\R^d}\big[H\left(X_s,\theta+G(X_s+\eta y)-\xi_\eta\right) \big]\rho(y)\d y\d s\bigg]\\
	&\leq t ~\mathrm{ess\,sup}_{\P_0}\bigg[\frac{1}{2}\mathrm{div}\big(a(G+\theta-\xi_\eta)\big)+ H(G+\theta-\xi_\eta)\bigg].
\end{aligned}	
\end{equation}
Letting $\eps\to 0$, then $\eta\to 0$ and using the continuity of the
map \begin{equation*} \theta\mapsto
  \mathrm{ess\,sup}_{\P_0}\left[\frac{1}{2}\mathrm{div}(a(G+\theta))+
    H(G+\theta)\right]\end{equation*} we conclude the proof of the proposition.

\subsection{Upper bound: proof of Theorem \ref{thm}}\label{subsec ProofThmLDP}

In this section, we will prove the following corresponding upper bound. Combined with the corresponding lower bound shown in Theorem \ref{thm lower bound}, it will then complete the proof of Theorem \ref{thm}.

\begin{theorem}[Upper bound for $u_\eps$]
\label{thm-upper-bound}
Let $u_{\eps}(t,x)$ be defined as in
\eqref{eq:u-eps} and $u_{\hom}$ as in
\eqref{eq:lax-oleinik}. If \ref{assump:est-erg}-\ref{assump:fkg}, \ref{f1},\ref{f1'} and \ref{f2}-\ref{f4} hold, then $\P_0$-a.s., for all
$T>0$ and $R\geq 1$,
\begin{equation}\label{eq:upper-bound} 
\limsup_{\eps\to 0}\sup_{0\leq t\leq T}\sup_{x \in D(\eps, R)} 
\big( u_{\eps}(t,x,\omega) {-} u_{\hom}(t,x)\big)\leq 0  \quad \text{ with }
D(\eps, R)= \eps \mathcal C_\infty \cap B_R(0),.
\end{equation}
\end{theorem}
The proof of Theorem \ref{thm-upper-bound} will be shown in Section
\ref{sec-general-initial} below.

  \noindent{\bf Proof of Theorem \ref{thm}:} The proof follows by combining the lower bound from Theorem \ref{thm lower bound} and the upper bound from Theorem \ref{thm-upper-bound}. Indeed, we first prove Theorem \ref{thm} for the case $p=1$, namely for all  $R,T>0$, 
\begin{equation}\label{eq0}
\limsup_{\eps\to 0} \sup_{0\leq t \leq T} \frac 1 {|D(\eps,R)|} \int_{D(\eps,R)}   |u_\eps (t,x) - \uu(t,x)| \:\d x  = 0, 
\end{equation}
for $\P_0$-a.e.\ $\omega\in \Omega_0$. For this purpose, 
we exploit Theorem \ref{thm-upper-bound}.  Indeed, \eqref{eq:upper-bound} means 
\begin{equation}\label{Beps}
u_\eps(t,x,\omega)-\uu(t,x) \leq B_\eps(\omega) \text{ for all } (t,x)\in [0,T]{\times}
D(\eps,R),
\end{equation}
where $0\leq B_\eps(\omega) \to 0$ as $\eps \to 0$. Thus, we have
\[
|u_\eps{-} \uu| = \big| (B_\eps {+} \uu  {-}u_\eps) - B_\eps\big| \leq 2 B_\eps
+ \uu-u_\eps. 
\]
and conclude 
\[
\sup_{0\leq t\leq T}\frac1{|D(\eps,R)|} \int_{D(\eps,R)}   |u_\eps{-} \uu| 
\:\d x  \leq 2B_\eps  + \sup_{0\leq t\leq T} \frac1{|D(\eps,R)|} 
\int_{D(\eps,R)} \big(\uu {-}u_\eps \big) \:\d x 
\]
Thus, it remains to show that $\P_0$-a.s., for all $R\geq 1$ and $T>0$, 
\begin{equation}
  \label{eq:L1-conv-reduction1}
 \limsup_{\eps\to 0} \sup_{0\leq t\leq T} \frac1{|D(\eps,R)|} 
\int_{D(\eps,R)} \big(\uu(t,x) {-}u_\eps(t,x) \big) \:\d x \leq 0, 
\end{equation}
which will now follow from Theorem \ref{thm lower bound}. Indeed, if we let 
$$
V_\eps(t,x,\omega)= B_\eps(\omega) + \uu(t,x)- u_\eps(t,x,\omega)
$$
then by \eqref{Beps}, $B_\eps\to 0$ and $V_\eps\geq 0$ uniformly over $(t,x)$ and $\P_0$-a.s. If $Q_1(0)$ is the unit cube in $\R^d$ around $0$, to show \eqref{eq:L1-conv-reduction1} it suffices to show that $\P_0$-a.s., 
\begin{equation}
  \label{eq:L1-conv-reduction2}
 \limsup_{\eps\to 0} \sup_{0\leq t\leq T} \frac1{|Q_1(0)\cap \eps\mathcal C_\infty|} 
\int_{Q_1(0)\cap \eps\mathcal C_\infty} V_\eps(t,x) \d x \leq 0.
\end{equation}
We decompose 
$$
Q_1(0)= \cup_{j\in\{1,\dots,N\}^d} Q^{j,N} \qquad \mbox{where} \,\, Q^{j,N}=\frac 1 N(j+ Q_1(0))
$$
is the cube of side-length $1/N$ centered around $j/N\in \R^d$. Let $B_{r_N}(j/N)$ be the smallest ball of radius $r_N=C_d/N$ around $j/N$ contaning the cube $Q^{j,N}$. Next, note that $\frac{|Q_1(0)\cap \varepsilon \mathcal C_\infty|}{|Q_1(0)|}=\frac{1}{|Q_1(0)|}
\int_{Q_1(0)}
\mathbf 1_{\{x\in \varepsilon \mathcal C_\infty\}}\,\d x
=
\frac{1}{|Q_1(0)|}
\int_{Q_1(0)}
\mathbf 1_{\{\tau_{x/\varepsilon}\omega\in\{0\in\mathcal C_\infty\}\}}\, \d x$ and by the spatial ergodic theorem under $\P$, the right-hand side converges as $\eps\to 0$
$\P$-almost surely, and hence $\P_0$-almost surely, to $\E^\P\big[\mathbf 1_{\{0\in\mathcal C_\infty\}}\big]
=\P[0\in\mathcal C_\infty]$. Hence, $\P_0$-a.s., $|Q_1(0)\cap \varepsilon \mathcal C_\infty|
\longrightarrow |Q_1(0)|\,\P[0\in\mathcal C_\infty]>0$. Thus, 
$$
\begin{aligned}
&\limsup_{\eps\to 0} \sup_{0\leq t\leq T} \frac1{|Q_1(0)\cap \eps\mathcal C_\infty|} 
\int_{Q_1(0)\cap \eps\mathcal C_\infty} V_\eps(t,x) \d x \\
&\leq C \limsup_{\eps\to 0} \sup_{t\in [0,T]} \sum_{j\in\{1,\dots,N\}^d} \int_{Q^{j,N}\cap \eps \mathcal C_\infty} V_\eps(t,x) \d x \\
&\leq C \limsup_{\eps\to 0} \sup_{t\in [0,T]} \sum_{j\in\{1,\dots,N\}^d} \int_{B_{r_N}(j/N)\cap \eps \mathcal C_\infty} V_\eps(t,x) \d x \\
& \leq C |B_{r_N}| \limsup_{\eps\to 0} \sup_{t\in [0,T]} \sum_{j\in\{1,\dots,N\}^d} \frac 1 {|B_{r_N}(j/N) \cap \eps \mathcal C_\infty|}\int_{B_{r_N}(j/N)\cap \eps \mathcal C_\infty} V_\eps(t,x) \d x \\
& \leq C_d r_N^d N^d g(r_N)= C^\prime_d g(r_N) \to 0 \quad\mbox{as $N\to\infty$,}
\end{aligned}
$$
by Theorem \ref{thm lower bound}, where the second upper bound above we used that $V_\eps\geq 0$ and $Q^{j,N}\subset B_{r_N}(j/N)$. This proves \eqref{eq:L1-conv-reduction2} and therefore Theorem \ref{thm} for $p=1$. To extend it for all $p\geq 1$ we again exploit the bound $|u_\eps(t,x,\omega)-\uu(t,x)|
\leq C(1{+}R {+}T )$ (recall \eqref{eq:ueps.lower.upper} and the argument underneath) and the convergence in $L^p$ follow immediately. This completes the proof of Theorem \ref{thm}.
\qed 

%{\color{blue} The remaining part of this subsection can be removed now.}

\begin{remark}\label{remark proof LDP}
As mentioned in Remark \ref{remark LDP}, let $u_\eps$ be the solution of the Hamilton-Jacobi-Bellman equation \eqref{eq-HJB'} for the particular choice \eqref{Hb} and initial condition $f(x)=\langle \theta,x \rangle$. Then 
\begin{equation}\label{eq:v}
\begin{aligned}
&v(t,x):= \exp\bigg\{\frac{ u_\eps(\eps t, \eps x)}\eps\bigg\} \quad\mbox{solves} \quad \frac{\partial}{\partial t} v(t,x)= (\mathcal L^{\ssup{b,\omega}} v)(t,x), \qquad v(0,x)= \e^{\langle \theta,x \rangle}, \quad\mbox{where}\\
&\qquad \mathcal L^{\ssup{b,\omega}}= \mathrm{div}\big(a(\cdot,\omega)\nabla\cdot\,) 
  + \langle b(\cdot,\omega), \nabla\cdot \, \rangle_a
\end{aligned}
\end{equation} 
is the generator of the $\R^d$-valued diffusion $X_t$. By Feynman-Kac formula, 
we have $v(t,x)= E^{b,\omega}_x[\exp\{\langle \theta,X(t)\rangle\}]$ with
$E^{b,\omega}_x$ denoting expectation with respect to the diffusion with
generator $\mathcal L^{\ssup{b,\omega}}$ starting at $x\in \R^d$. Since 
$u_\eps(t,0) = \frac1\eps \log v(t/\eps, 0)$, we have 
$\lim_{\eps\to 0} \eps u_\eps(1,0)= \lim_{t\to\infty} \frac 1 t \log v(t,0)$,
and the quenched LDP mentioned in Remark \ref{remark LDP} follows from Theorem \ref{thm}. Actually the claim follows from combining Lemma \ref{lemma:lower-bound-lemma1}, Proposition \ref{prop: up-bound-lin-dat} and Theorem \ref{thm-equivalence}, which only assume the uniform continuity of the initial condition.\qed 
\end{remark}

\subsubsection{\bf Proof of Theorem \ref{thm-upper-bound}}
\label{sec-general-initial}

For  any fixed $t,\eps ,x$ and $\omega$,  
\begin{align*}
  u_\eps(t,x,\omega)-u_{\hom}(t,x) 
  &=\sup_{c\in \mathbf C_T^*}\bigg(E^{P_{x/\eps}^{c,\omega}}\bigg[f(\eps X_{t/\eps})-\eps \int_0^{t/\eps}L(X_s,c(s))\d s\bigg]\bigg)
  -\sup_{y\in \R^d}\bigg(f(y)-t\mathcal{I}\big(\frac{y-x}{t}\big)\bigg)\\
  &\leq \sup_{c\in \mathbf C_T^*}E^{P_{x/\eps}^{c,\omega}}\bigg[t\mathcal{I}\bigg(\frac{\eps X_{t/\eps}-x}{t}\bigg)-\eps \int_{0}^{t/\eps}L(X_s,c(s))\d s\bigg]\\
  &=\sup_{c\in \mathbf C_T^*}E^{P_{x/\eps}^{c,\omega}}\bigg[\sup_{\theta\in \R^d}\langle \theta, \eps X_{t/\eps}-x\rangle-t\overline{H}(\theta)-\eps \int_{0}^{t/\eps}L(X_s,c(s))\d s\bigg].
\end{align*}
Since by Lemma \ref{lemma:control_restriction} and \eqref{eq:ineq_c_norm}-\eqref{eq:ineq_c_norm2}
\begin{equation}
 \label{eq:unif_bound_xt}
 S:=\sup_{c\in \mathbf C^*_T}\sup_{0<\eps \leq 1}\sup_{0\leq t\leq
   T}\sup_{x\in D(\eps,R) }E^{P_{x/\eps}^{c,\omega}}\bigg[|\eps
 X_{t/\eps}|+\bigg\lvert\eps
 \int_{0}^{t/\eps}L(X_s,c(s))\d
 x\bigg\rvert\bigg]<\infty 
\end{equation} 
and $\overline{H}$ also satisfies the estimates of \eqref{eq:H1-H} with the Euclidean norm, we deduce
that it is enough to show that for any fixed $\theta\in \R^d$, 
\begin{equation}
  \label{eq:upper_bound_to_show} 
  \limsup_{\eps\to 0}\sup_{c\in
    \mathbf C_T^*}\sup_{0\leq t\leq T}\sup_{x\in D(\eps,R) }E^{P_{x/\eps}^{c,\omega}}\bigg[\langle \theta, \eps
  X_{t/\eps}-x\rangle-t\overline{H}(\theta)-\eps
  \int_{0}^{t/\eps}L(X_s,c(s))\d s\bigg]\leq 0.
\end{equation}
Following as in the proof of Proposition \ref{prop: up-bound-lin-dat}, for any
$G\in \mathcal{G}_{\delta}$ and $\eta>0$, we apply It\^{o}'s formula to
$\theta+ V_{\eta,G}-\xi_\eta$ with $|\xi_\eta|\to 0$ as $\eta\to 0$,
obtaining
\begin{align}
  &E^{P_{x/\eps}^{c,\omega}}\bigg[\langle \theta,\eps X_{t/\eps}-x\rangle-\eps \int_0^{t/\eps}L(X_s,c(s))\d s\bigg]-t\overline{H}(\theta) \nonumber\\
  &\leq -\eps E^{P_{x/\eps}^{c,\omega}}\big[V_{\eta,G}(X_{t/\eps},\omega)-V_{\eta,G}(x/\eps,\omega)-\langle \xi_\eta, X_{t/\eps}\rangle\big] \label{term1}\\
  &\qquad+\eps E^{P_{x/\eps}^{c,\omega}}\bigg[\int_0^{t/\eps}\frac{1}{2}\mathrm{div}\big(a(X_s))(\theta+G_\eta(X_s)-\xi_{\eta}\big)+H(X_s,\theta+G_\eta(X_s)-\xi_\eta)\d s\bigg]-t\overline{H}(\theta)\label{term3}.
\end{align}
To bound the first term \eqref{term1}, we recall that, thanks to Theorem
\ref{thm:sublinear}, $\P_0$-a.s., for all $\tau>0$ there is some
$C_\tau=C_\tau(\omega)$ such that for all $x\in {\mathcal C_\infty}$,
$|V_{\eta,G}(x,\omega)|\leq \tau|x|+C_\tau$. Then the first expectation
\eqref{term1} is bounded above by 
\begin{equation*} 2 \eps C_\tau +
  (\tau+|\xi_\eta|) E^{P^{c,\omega}_{x/\eps}}[|\eps X_{t/\eps}|+|x|].
\end{equation*}
By \eqref{eq:unif_bound_xt}, we deduce that 
\begin{align*} &\limsup_{\tau\to
    0}\limsup_{\eps\to 0}\sup_{c\in \mathbf C_T^*}\sup_{0\leq t\leq
    T}\sup_{x\in D(\eps,R) }\bigg(-\eps
  E^{P_{x/\eps}^{c,\omega}}[V_{\eta,G}(X_{t/\eps},\omega)-V_{\eta,G}(x/\eps,\omega)-\langle
  \xi_\eta,X_{t/\eps}\rangle]\bigg) \leq S|\xi_\eta|.
\end{align*}
We bound the second term as in \eqref{eq:eq16}, implying that \eqref{term3} is bounded above by
\begin{align*}
 t ~\mathrm{ess \, sup}_{\P_0}\bigg[\frac{1}{2}\mathrm{div}\big(a(G+\theta-\xi_{\eta})\big)+ H(G+\theta-\xi_{\eta})\bigg]-t\overline{H}(\theta).
\end{align*}
By Theorem \ref{thm-equivalence}, for any $\eps'>0$, there is some
$G\in \mathcal{G}_{\delta}$ so that the last display is bounded
by $\eps'+t
  \overline{H}(\theta-\xi_\eta)-t\overline{H}(\theta)$,
so that the final bound is
$S|\xi_\eta|+\eps'+t \overline{H}(\theta-\xi_\eta)-t\overline{H}(\theta)$. As
$\overline{H}$ is continuous, letting first $\eps\to 0$, then $\eps'\to 0$
and finally $\eta\to 0$, we deduce \eqref{eq:upper-bound}, thus proving Theorem \ref{thm-upper-bound}.
\qed

\section{Sublinear growth: Proof of Theorem \ref{thm:sublinear}}\label{sec-proof-sublinear}
 We recall the definition of the class $\mathcal G_\delta$ and $\mathcal G_\infty$ from Section \ref{classG}. The proof of Theorem \ref{thm:sublinear} will be carried out in the following three main steps: 
 
\noindent (A) In the first step we will show that, if $G \in \mathcal G_\delta$, then $V_G$ has only sublinear 
growth along any of the coordinate directions. This will be shown in Theorem \ref{sublinearity along coordinate directions} in Section \ref{sec-coordinate}.

\noindent (B) Next, we will  
provide a control on the growth of $V_G$ on growing balls on $\mathcal C_\infty$ ``on average":  Proposition \ref{thm:weaksublinear} in Section \ref{sec:weaksublinear} will show that, if $G\in \mathcal G_\infty$, then for any $\eps>0$ and $\P_0$-a.s., 
$$
\lambda_d\big\{x: \mathcal C_\infty(\omega): |x|\leq r, \frac{|V_G(x,\omega)|}r \geq \eps\big\} = o(r^d)\qquad\mbox{ as }r\uparrow \infty.
$$ 

\noindent (C) Using the two steps above, the proof of Theorem \ref{thm:sublinear} will be completed in Section \ref{sec-proof-thm:sublinear}. In this step, as well as in the preceding steps above, 
the geometry of $\mathcal C_\infty$ and \ref{assump:est-erg}-\ref{assump:fkg} will play a crucial role. Here and in the sequel, $\lambda_d$ denotes Lebesgue measure on $\R^d$.

\subsection{Controlling directional growth.}\label{sec-coordinate}

We fix a unit coordinate vector $e$ and for $\omega\in\Omega_0$, define the {\it successive arrivals}  $(n_k(\omega))_{k\in\N}$
of the cluster recursively as follows:
Recall \eqref{def-n} and define 
	$$
	n_1(\omega)= n(\omega,e), \quad\mbox{and for $k\geq 1$ we set}\quad 
	n_{k+1}(\omega):=\min\{l\in\N:l>n_k(\omega), le\in{\mathcal C_\infty}(\omega)\}.
	$$

\begin{theorem}\label{sublinearity along coordinate directions}	
Recall the corrector $V_G$ from Definition \eqref{def-V} and 
let $e$ be any unit coordinate vector. If $G\in \mathcal{G}_\delta$, then  $\P_0$-a.s.,
	$$
	\lim_{k\to\infty}\frac{|V_G(n_k(\omega)e,\omega)|}{k}=0.
	$$
\end{theorem}

\begin{proof}
For each $k\in\N$, set $x_0=0$ and $x_j=n_{j}e$ for $1\leq j\leq k$. We
choose a path $0\leadsto x_{k}$ from 0 to $x_{k}$ contained in
${\mathcal C_\infty}(\omega)$ such that, for some $0=t_0<t_1<...<t_k=1$ and
$r:[0,1]\rightarrow (x_0\leadsto x_{k})$, it holds $r(t_j)=x_j$. Then by the
definition of $V_G$ in \eqref{def-V},
\begin{align*}
V_G(n_ke,\omega)&=\int_{x_0\leadsto x_{k}} G( r,\omega)\d r=\sum_{j=0}^{k-1}\int_{x_j\leadsto x_{j+1}} G( r,\omega)\d r=\sum_{j=0}^{k-1}\int_{0\leadsto (x_{j+1}-x_j)} G( r,\tau_{x_j}\omega)\d r\\
&=\sum_{j=0}^{k-1}V_G(x_{j+1}-x_j,\tau_{x_j}\omega)=\sum_{j=1}^{k-1}V_G((n_{j}(\omega)-n_{j-1}(\omega))e,\tau_{n_{j-1}e}\omega)\\
&=\sum_{j=0}^{k-1}V_G(n_1(\sigma_e^j(\omega))e,\sigma_e^j(\omega)).%=V_G(ue,\omega)+\sum_{j=0}^{k-1}f\circ\sigma_e^j(\omega).
\end{align*}
where we recall from \eqref{def-induced} the definition of the induced shift $\sigma_e: \Omega_0 \to \Omega_0$. We define the function 
\begin{align}
\label{f sum for psi}
F(\omega)=V_G(n(\omega,e)e,\omega) \qquad\mbox{so that}\qquad 
V_G(n_k(\omega),e)e,\omega)=\sum_{j=0}^{k-1} F \circ\sigma_e^j(\omega).
\end{align}
Proposition \ref{prop 1} implies that the induced shift $\sigma_e$ is $\P_0$-preserving and ergodic. Furthermore, Proposition \ref{zero expectation} implies that $F\in L^1(\P_0)$. %$G$ is uniformly bounded, which implies $V_G(ue,\omega)<\infty$ and also . 
Then by Birkhoff's Ergodic Theorem,
\begin{align}
\label{convergence of f sum}
\lim_{k\to\infty}\frac{\sum_{j=0}^{k-1}F\circ\sigma_e^j(\omega)}{k}=\bE_0[V_G(n(\omega,e)e,\omega)]=0,
\end{align}
where the last equality follows from the induced mean-zero property \eqref{eq:meanzero} of $G\in \mathcal{G}_\delta$. 	
\end{proof}

\begin{cor}\label{cor:sublinearity_along_coordinates}
Let  $G\in \mathcal{G}_\infty$. Then for any unit coordinate vector $e$ and $\P_0$-a.s., 
\begin{equation*}
	\lim_{s\to\infty}\1\{se\in {\mathcal C_\infty}(\omega)\}\frac{|V_G(se,\omega)|}{s}=0.
\end{equation*}
\end{cor}
\begin{proof}
	If $se\in {\mathcal C_\infty}(\omega)$, then there exists $k\geq 0$ such that $n_k(\omega)\leq s < n_{k+1}(\omega)$. Note that $k\nearrow \infty$ as $s\nearrow\infty$. Then we have \begin{equation*}
		\frac{|V_G(se,\omega)|}{s}\leq \frac{|V_G(n_{k}(\omega)e,\omega)|}{n_k(\omega)}+ \frac{|V_G((s-n_k(\omega))e,\tau_{n_k(\omega)e}\omega)|}{n_k(\omega)}.
	\end{equation*}
	By the ergodic theorem (as in the proof of Theorem \ref{sublinearity along coordinate directions}) and  \ref{assump:exp-dec-dist-indshift}, $\lim_{k\to\infty}\frac{n_k(\omega)}{k}= \bE_0[n_1]<\infty$ $\P_0$-a.s. This fact, together with Theorem \ref{sublinearity along coordinate directions}, allow us to deduce that the first term in the sum above goes to zero as $s\to \infty$. It remains to bound the second term. Note that it suffices to show that \begin{equation*}
		\lim_{k\to\infty}\sup_{n_k(\omega)\leq s\leq n_{k+1}(\omega)}\1\{se\in {\mathcal C_\infty}(\omega)\}\frac{|V_G((s-n_k(\omega))e,\tau_{n_k(\omega)e},\omega)|}{k}=0\quad \P_0\text{-a.s.}
	\end{equation*}
	Since $G\in \mathcal{G}_\infty$, it is enough to prove that \begin{equation}
		\lim_{k\to\infty}\sup_{n_k\leq s\leq n_{k+1}}\1\{se\in {\mathcal C_\infty}(\omega)\}\frac{\d_{\omega}(n_k(\omega)e,se)}{k}=0~\P_0\text{-}a.s.
	\end{equation}
	By the Borel-Cantelli lemma, it suffices to verify that for any $\eps>0$, \begin{equation}\label{eq:borel-cantelli-sum}
		\sum_{k=1}^{\infty}\P_0\left(\sup_{n_k(\omega)\leq s\leq n_{k+1}(\omega)}\1\{se\in {\mathcal C_\infty}(\omega)\}\d_{\omega}(n_k(\omega)e,se)>k\eps\right)<\infty.
	\end{equation}
	Since $\P_0$ is invariant under $\tau_{n_ke}$, the sum above is 
	\begin{equation*}
		\sum_{k=1}^{\infty}\P_0\left(\sup_{0\leq s\leq n_{1}(\omega)}\1\{se\in {\mathcal C_\infty}(\omega)\}\d_{\omega}(0,se)>k\eps\right)<\infty \qquad\mbox{for each $\eps>0$ by Lemma \ref{lemma 4.4}.}
	\end{equation*}
	\end{proof}

\subsection{Controlling density of growth.}\label{sec:weaksublinear}

The main result of this section is the following result: 
\begin{prop}\label{thm:weaksublinear}

	Let $d\geq 2$ and $G\in \mathcal{G}_\infty$. Then for all $\eps>0$ and $\P_0$-almost all $\omega$,
	\begin{align}\label{equation in theorem 5.4}
	\limsup_{r\rightarrow\infty}\frac{1}{(2r)^d}\int_{x\in{\mathcal C_\infty}(\omega),|x|\leq r}\1\{|V_G(x,\omega)|\geq \eps r\}\d x=0.
	\end{align}
\end{prop}

The proof of Proposition \ref{thm:weaksublinear} consists of three main steps.

\noindent{\bf Step 1:} We start this section with a definition: 
	Given $K>0$ and $\eps>0$, we say that a point $x\in\R^d$ belongs to $\mathscr G_{K,\eps}(\omega)$ for $\omega\in\Omega$ if $x\in{\mathcal C_\infty}(\omega)$ and
	\begin{align}
	\label{good points}
	|V_G(x+te,\omega)-V_G(x,\omega)|\leq K + \eps |t|
	\end{align}
	for each $t\in \R$, and $e$ is a unit coordinate vector such that $x+te\in {\mathcal C_\infty}(\omega)$.
	 %We will use $\mathscr G_{K,\eps}=\mathscr G_{K,\eps}(\omega)$ to denote the set of $K,\eps$-good points in configuration $\omega$.
	We will use the following consequence of Corollary \ref{cor:sublinearity_along_coordinates}
	in the sequel: for every $\eps>0$, $\P(0\in {\mathcal C_\infty})=\lim_{K\to\infty}\P(0\in \mathscr{G}_{K,\eps})$. 
For $k\in\{1,...,d\}$, let us also define
\begin{align}
\label{definition of Lambda}
\Lambda_r^k=\{x\in\R^k:|x|_\infty\leq r\},
\end{align}
which is the $k$-dimensional section of the $d$-dimensional box $\{x\in\R^d:|x|_\infty\leq r\}$, and set 
\begin{equation}\label{definition of rho}
	\begin{aligned}
\varrho_{k,\eps}(\omega)&:=\limsup_{r\rightarrow\infty}\inf_{y\in{\mathcal C_\infty}(\omega)\cap\Lambda_r^1}\frac{1}{|\Lambda_r^k|}\int_{x\in{\mathcal C_\infty}(\omega)\cap\Lambda_r^k}\1\{|V_G(x,\omega)-V_G(y,\omega)|\geq \eps r\}\d x,\quad 
\varrho_k(\omega)&:=\lim_{\eps\searrow 0}\varrho_{k,\eps}(\omega).\\
\end{aligned}
\end{equation}

\begin{lemma}
	\label{lemma 5.5}
	Let $1\leq k<d$. If $\varrho_k=0$ $\P$-almost surely, then also $\varrho_{k+1}=0$ $\P$-almost surely.
\end{lemma}

\noindent{\bf Step 2 (Proof of Lemma \ref{lemma 5.5}).} For $k\leq d$, we consider the $k$-dimensional Lebesgue measure on $\R^k$ and we call it $\lambda_k$. We assume that  $\P$-a.s. $\varrho_1=0$. In particular, for each $\eps>0$ and large enough $r$, there is some set $\Delta\subset {\mathcal C_\infty}\cap \Lambda_r^1$ satisfying  \begin{align*}
	\lambda_1(\Lambda_r^1\cap{\mathcal C_\infty}\setminus\Delta) &\leq \eps\lambda_1(\Lambda_r^1),\qquad 
	|V_G(x,\omega)-V_G(y,\omega)|\leq \eps r\quad \forall x,y\in\Delta.
\end{align*}
 Moreover, for $K>0$ large enough (but deterministic), replacing $\Delta$ by  $\Delta\cap \mathscr{G}_{K,\eps}$ grant us the following properties for large $r$ :
 \begin{equation}\label{i-iv}
 \begin{aligned}
 &\mathrm{(i)}\,\,  \lambda_1(\Lambda_r^1\cap{\mathcal C_\infty}\setminus\Delta) \leq \eps\lambda_1(\Lambda_r^1), \qquad
\mathrm{(ii)} \,\, |V_G(x,\omega)-V_G(y,\omega)|\leq \eps r\quad x,y\in\Delta, \\
& \mathrm{(iii)}\,\, \Delta\subset \mathscr{G}_{K,\eps}\quad \qquad\qquad\qquad\mbox{and}\quad
\mathrm{(iv)}\,\, \Delta\cap \Lambda_r^1\neq \emptyset. 
 \end{aligned}
 \end{equation}
 This is a consequence of the fact that  $\lim_{K\to\infty}\P(0\in {\mathcal C_\infty}\setminus \mathscr{G}_{K,\eps})=0$, and the ergodic theorem. We stress that even though these conditions are easily satisfied in dimension one, the construction will allow us to obtain the same properties in larger dimensions. In particular, we want that the ``base" $\Delta$ is contained in each successive step, so that (iv) in \eqref{i-iv} will be always be valid.
 
 Next, for $L\in \N$ and $r>0$, define \begin{equation}\label{eq:Xi}
 	\Xi_{L,r}(\omega):=\{x\in \Lambda_r^1: \#\{0\leq i\leq L-1: x+ie_2\in {\mathcal C_\infty}(\omega\}>0\}.
 \end{equation} We claim that for each $\delta>0$, there exists some $L=L(\delta)$ (deterministic) that satisfies $\P$-a.s. $\lambda(\Xi_{L,r})\geq (1-\delta)\lambda(\Lambda_r^1)$ for large $r$ (which may depend on $\omega$). Indeed, by the ergodic theorem,  the following equality holds $\P$-a.s. for all $L\in \N$:\begin{equation}\label{eq:eq20}
	\lim_{r\to \infty}\frac{\lambda(\Xi_{L,r})(\omega)}{\lambda(\Lambda_r^1)}=\P(\#\{0\leq i\leq L: ie_2\in {\mathcal C_\infty}(\omega\}>0).
\end{equation}
Since \begin{equation*}
	\lim_{L\to \infty}\frac{1}{L}\#\{i\in \{0,\cdots, L-1\}: ie_2\in {\mathcal C_\infty}(\omega)\}=\P(0\in {\mathcal C_\infty})>0 ~\P\text{-a.s.}
\end{equation*}
as $L\to \infty$, the probability on the right in \eqref{eq:eq20} converges to $1$, so the claim holds. For fixed $L$, choose $K>0$ large enough so that $\P$-a.s., for all $i=0,\cdots, L-1$ the conditions (i)-(iv) in \eqref{i-iv} will hold  some $\Delta_i\subset \tau_{ie_2}(\Lambda_r^1)$ (replacing $\Lambda_r^1$ with $\tau_{ie_2}(\Lambda_r^1)$ in (i) and (iv), for $r$ large enough. Next, we define for $r>0$ (and setting $\Delta_0:=\Delta$) \begin{equation}
	\Lambda=\Lambda_r:=\{x\in \Lambda_r^2\cap {\mathcal C_\infty}: \exists ~0\leq i\leq L-1,~ (y,t)\in [-r,r]^2 \text{ such that } x=ye_1+te_2\text{ and } ye_1+ ie_2\in \Delta_i\}.
\end{equation} 
In words, $\Lambda$ represents the points in $x\in \Lambda_r^2$ which have some $\tilde{x}\in \Delta_i$ that shares the same projection over $\R e_1$. Note that $\Delta\subset \Lambda$, so in particular, $\Lambda\cap \Lambda_r^1\neq \emptyset$ for large $r$. We show that the density $\Lambda$ is close to $1$. More precisely, if $x\in (\Lambda_r^2\cap {\mathcal C_\infty})\setminus \Lambda$, then $x=ye_1+te_2$ for some $(y,t)\in [-r,r]^2$, and either $ye_1\notin \Xi_{L,r}$ or $ye_1\in \Xi_{L,r}$ and $ye_1+ ie_2\in {\mathcal C_\infty}\setminus \Delta_i$ for all $i=0,\cdots, L-1$. Therefore, for large enough $r$,\begin{equation}\label{eq:eq21}
	\begin{aligned}
	\frac{\lambda_2((\Lambda_r^2\cap {\mathcal C_\infty})\setminus \Lambda)}{\lambda_2(\Lambda_r^2)}&\leq\frac{1}{2r}\int_{-r}^r\1\{ye_1\in \Lambda_r^1\setminus \Xi_{L,r}\}\d y+ \sum_{i=0}^{L-1}\frac{1}{2r}\int_{-r}^r\1\{ye_1+se_2\in \Lambda_r^1\setminus \Delta_i\}\d y\\
	&=\frac{1}{2r}\left[\lambda_1(\Lambda_r^1\setminus \Xi_{L,r})+\sum_{i=0}^{L-1}\lambda_1((\Lambda_1^r\cap {\mathcal C_\infty})\setminus \Delta_i)\right]\leq L\eps + \delta.
\end{aligned} 
\end{equation}
At this point, we choose $\eps$ and $\delta$. Let $\eps, \delta>0$ small enough so that $L\eps+\delta<\frac{1}{2}\P(0\in {\mathcal C_\infty})^2$. By the FKG-inequality in \ref{assump:fkg} 
(note that $\{x\in {\mathcal C_\infty}\}$ is an increasing event), for every $x,y\in \R^d$ we have  

\begin{equation*}
	\P(x\in{\mathcal C_\infty}(\omega),y\in{\mathcal C_\infty}(\omega))\geq \P(x\in{\mathcal C_\infty}(\omega))\P(y\in{\mathcal C_\infty}(\omega))=\P(0\in {\mathcal C_\infty})^2.
\end{equation*}

Moreover, for $K$ large enough, by the ergodic theorem we have for any $s,t \in \{0,\cdots, L-1\}$\begin{equation}\label{eq:eq22}
	\lim_{r\to \infty}\frac{1}{\lambda_1(\Lambda_r^1)}\lambda_1(x\in \Lambda_r^1 : x+se_2\in \mathscr{G}_{K,\eps}, x+ te_2\in \mathscr{G}_{K,\eps} )= \P(se_2\in \mathscr{G}_{K,\eps}, te_2\in \mathscr{G}_{K,\eps})>L\eps + \delta.
\end{equation}
Thus, for large enough $r$, for every $s,t\in \{0,\cdots, L-1\}$, the density of points $x\in \Lambda_r^1$ such that 
$x+se_2\in \Delta_s$ and $x+te_2\in \Delta_t$ is positive.  
To finish the proof, we verify that for each $u,v\in \Lambda$, $|V(u,\omega)-V(v,\omega)|\leq 7\eps r$ for large $r$
such that all the above holds (in particular, $\eqref{eq:eq21},\eqref{eq:eq22}$). Indeed, if $u=x_1e_1+y_1e_2$ and $v=x_2e_1+y_2e_2 \in \Lambda$, then there are $s, t\in \{0,\cdots, L-1\}$ such that if $u':=x_1e_1+se_2$ and $v':=x_2e_1+ te_2$, then $u',v'\in \mathscr{G}_{K,\eps}(\omega)$ (for $K=K(\omega)$ independent on $r$ that satisfies the conditions listed above). Moreover, by \eqref{eq:eq22}, there exists some $x_3\in \Lambda_r^1$ satisfying $u'', v''\in \mathscr{G}_{K,\eps}$, where $u'':=x_3e_1+ se_2$ and $v''=x_3e_1+ te_2$. Putting all together, we have \begin{align*}
	|V_G(u,\omega)-V_G(v,\omega)|&\leq |V_G(u,\omega)-V_G(u',\omega)|+|V_G(u',\omega)-V_G(u'',\omega)|+|V_G(u'',\omega)-V_G(v'',\omega)|+\\
	&\quad|V_G(v'',\omega)-V_G(v',\omega)|+ |V_G(v',\omega)-V_G(v,\omega)|\\
	&\leq K+\eps|x_2-s|+K+\eps|x_1-x_3|+K+\eps |s-t|+ K+ \eps|x_2-x_3|+K+\eps|y_2-t|\\
	&\leq 5K+3\eps L+6\eps r\leq 7\eps r
\end{align*}
for large enough $r$. In conclusion, by the last computation, the fact that $\Lambda\cap \Lambda_r^1\neq \emptyset$ and \eqref{eq:eq21}, $\varrho_{2,7\eps}\leq L\eps +\delta$.  By letting  first $\eps \searrow 0$ and then $\delta \searrow 0$, we deduce that $\varrho_{2}=0$ $\P_0$-a.s. 

We can use the same construction to go to higher dimensions. More precisely, the element $\Lambda$ for dimension $\rho$ becomes the element $\Delta$ in dimension $\rho+1$. The base case guarantees that properties (i)-(iv) in \eqref{i-iv} that appear at the beginning of the proof remain true for $\rho>1$. This finishes the proof of Lemma \ref{lemma 5.5}.	\qed
	
\medskip

\noindent{\bf Step 3 (Proof of Proposition \ref{thm:weaksublinear}).} This follows from Corollary \ref{cor:sublinearity_along_coordinates} and Lemma \ref{lemma 5.5}. Indeed, since 
\begin{align*}
&\inf_y\lambda_1\bigg(\bigg\{x\in{\mathcal C_\infty}\cap\Lambda_r^1:|V_G(x,\omega)-V_G(y,\omega)|\geq \eps r\bigg\}\bigg) 
\leq \lambda_1\bigg(\bigg\{x\in{\mathcal C_\infty}\cap\Lambda_r^1:|V_G(x,\omega)|\geq \eps r-|V_G(0,\omega)|\bigg\}\bigg), 
\end{align*}
and by Corollary \ref{cor:sublinearity_along_coordinates}, it holds $\varrho_1=0$ for $\P_0$-almost every $\omega$. By changing over to appropriate shifts, we also have $\varrho_1=0$ for $\P$-almost every $\omega$. 
%More precisely, if $\omega\notin \Omega_0$, choose $z$ so that $\tau_z\omega\in \Omega_0$ and note that $\varrho_1(\tau_z\omega)\geq \varrho_1(\omega)$.  
We use Lemma \ref{lemma 5.5} repeatedly, which shows that $\varrho_d=0$  $\P$-a.s. and thus, $\P_0$-a.s. Again by Corollary \ref{cor:sublinearity_along_coordinates}, there exists $r_0=r_0(\omega)$ with $\P_0(r_0<\infty)=1$ such that $|V_G(y,\omega)|\leq\eps r/2$ for any $r\geq r_0$ and any $y\in\Lambda_r^1\cap{\mathcal C_\infty}(\omega)$. Therefore, 
$$
\begin{aligned}
\lambda_d\big(\{x\in{\mathcal C_\infty}\cap\Lambda_r^d:|V_G(x,\omega)|\geq \eps r\}\big)
&\leq \,\, \inf_y\lambda_d\big(\{x\in{\mathcal C_\infty}\cap\Lambda_r^d:|V_G(x,\omega)-V_G(y,\omega)|\geq \eps r-|V_G(y,\omega)|\}\big)\\
&\leq \inf_y\lambda_d\big(\{x\in{\mathcal C_\infty}\cap\Lambda_r^d:|V_G(x,\omega)-V_G(y,\omega)|\geq \eps r/2\}\big),
\end{aligned}  
$$
and \eqref{equation in theorem 5.4} holds for any $\eps>0$. This finishes the proof of Proposition \ref{thm:weaksublinear}.
\qed

\subsection{Proof of Theorem \ref{thm:sublinear}.}\label{sec-proof-thm:sublinear}
We will prove an equivalent version of Theorem \ref{thm:sublinear}, namely:

\begin{theorem}\label{thm:sublinearv2}
Fix $d\geq 2$ and $G\in \mathcal{G}_\infty$. Then $\P_0$-a.e. $\omega\in \Omega_0$, 
$$
\lim_{r\to \infty }\sup_{x\in D\left(1,\frac{1}{r}\right)} \frac{|V_G(x,\omega) |}r=0.$$
\end{theorem}

Some preliminary lemmas will be required for the proof of the above result. Before that, let us set some notation that will be useful in the sequel. We will be interested in considering sets on $\R^d\times \R^d$, so we endow this space with the standard product Lebesgue measure, which we denote by $\lambda_d^{\otimes 2}$. The section on the ``first" coordinate of a measurable set $A\subset \R^d\times \R^d$ is 
\begin{equation}\label{def-Ax}
A^{\ssup x}:=\big\{y\in \R^d: (x,y)\in A\big\} \quad\forall x\in \R^d.
\end{equation}
 Given $a\in (0,1)$ and $r,\delta,\rho>0$, we also define 
\begin{equation}\label{def-C-D-E}
\begin{aligned}
	&C_r(a):=\{(x,y)\in \R^d\times \R^d:a r<|x-y|_{\infty}<r\},\\
	&D(\rho):=\{(x,y)\in \R^d\times \R^d: \d_\omega(x,y)\geq \rho|x-y|_\infty; x,y\in {\mathcal C_\infty}\},
	&E(r):=({\mathcal C_\infty})^2\cap ([-r,r]^d)^2.
\end{aligned}
\end{equation}

\begin{lemma}\label{lemma:sublinear_prelim1}
For any $a\in (0,1)$, there exists a constant $\rho=\rho(a,d)$ such that for all $\delta >0$, $\P_0$-a.s. for large enough $n\in \N$, for every $x,y\in {\mathcal C_\infty}\cap [-n,n]^d$ satisfying $a\delta n<|x-y|_\infty<\delta n$, we have $\d_\omega(x,y)\leq \rho|x-y|_\infty$.
\end{lemma}
\begin{proof}
	We recall the definition of the two-fold Palm distribution from \eqref{eq:multidim-campbell-eq} and Assumption \ref{assump:chem-dist}. Next, fix any $\delta'>\delta$ and $\rho<c_0$ (as in \eqref{eq:chem_dist_ineq}), and choose $0<a'<a$ such that for $a'\delta'<a\delta$, so that for  $n\in \N$ large enough, $\delta n+ 1 \leq \delta' n$ and $a\delta n-1>a'\delta'n$. Then we have 
    %By \ref{assump:chem-dist} and \eqref{eq:multidim-campbell-eq}, we have
	\begin{align*}
		&\P_0\bigg(\exists x,y\in {\mathcal C_\infty}\cap [-n,n]^d, \,\, a\delta n<|x-y|_\infty<n\delta, \,\, \d_\omega(x,y)\geq \rho|x-y|_\infty\bigg)\\
		&\leq \frac 1 {\P(0\in {\mathcal C_\infty})} 
		\P\bigg(\exists x\neq y\in \mathcal C_\infty(\omega)\cap \left[-(n+1),n+1\right]^d,\,\, a'\delta n<|x-y|_\infty <n\delta, \\ 
		&\qquad\qquad\qquad\qquad\qquad \d_\omega(x,y)\geq \rho|x-y|_\infty ; \,\, 0,x,y\in {\mathcal C_\infty}\bigg) \\
		&=\frac 1 {\P(0\in {\mathcal C_\infty})} \E\bigg[ \sum_{x\neq y\in \mathcal C_\infty(\omega)} \1\bigg\{(x,y)\in C_{\delta' n}(a')\cap [-(n+1),n+1]^{2d}, \d_\omega(x,y)\geq \rho |x-y|_\infty;0,x,y\in \mathcal C_\infty\bigg\}\bigg]  \\
		&= \frac{\zeta^2}{\P(0\in{\mathcal C_\infty})}\int_{C_{\delta' n}(a')\cap [-(n+1),n+1]^{2d}} \lambda_d^{\otimes 2}(\d x,\d y)\,\,
		 \P^{x,y} \bigg(0,x,y\in {\mathcal C_\infty}, \d_\omega(x,y)\geq \rho|x-y|_\infty\bigg )
		%&\qquad\qquad\qquad\qquad\qquad 
		\leq C\e^{-C' n},
	\end{align*}
	for some $C=c(a,d,\rho), C'=C'(a,d,\rho)>0$. Indeed, in the first identity of the above display, we used the definition of $C_{\delta^\prime n}(a^\prime)$ from \eqref{def-C-D-E}; in the subsequent identity we used the definition of $\zeta$ from \ref{assump:intensity} and that of the two-fold Palm distribution $\P^{x,y}$ from \eqref{eq:multidim-campbell-eq}; and in the last upper bound above we invoked Assumption \ref{assump:chem-dist}. The claim of the lemma now follows from the Borel-Cantelli lemma.
	\end{proof}
\begin{lemma}\label{lemma:density_box}
	Let $C\subset \R^d$ be any box of the type $[a_1,b_1]\times [a_2,b_2]\times \cdots [a_d,b_d]$. Then $\P$-a.s., $$ \lim_{r\to\infty}\frac {\lambda_d(\mathcal{C}_{\infty}\cap rC)}{\lambda_d(rC)}=p_\infty.$$ 
\end{lemma}
\begin{proof}
	This is an application of the ergodic theorem \cite[Theorem 10.14]{K02}.
\end{proof}
We are now ready to prove Theorem \ref{thm:sublinearv2} which will also prove Theorem \ref{thm:sublinear}. 
\begin{proof}[\bf Proof of Theorem \ref{thm:sublinearv2}]
We consider some $\ell=\ell(d,\mathbb{P})\in \N$ satisfying \begin{equation}\label{eq:eq0}
	p_\infty> \frac{1}{2^{d(\ell-1)}}.
\end{equation}

We claim the proof is complete once we show the following: in a  measurable set $A$ such that $\mathbb{P}_0(A)=1$, for all $\eps>0$ and $\omega\in A$, there exists some $r_0=r_0(\omega)$ such that if $r\geq r_0$, for all $x\in [-r,r]^d\cap {\mathcal C_\infty}$  with $|V_G(x,\omega)|_{\infty}>\eps r$, 
\begin{equation}\label{eq:eq15}
\lambda_d\left[\left(E(r)\cap C_{\delta r}(2^{-\ell})\right)^x \cap \{|V_G(\cdot,\omega)|_\infty\leq \eps r\}\right]>0
\end{equation} for some $\delta=\delta(d,\mathbb{P},\eps)$ that vanishes as $\eps\to 0$ (recall the notation \eqref{def-Ax} and \eqref{def-C-D-E}). 
Indeed, for any $x,y\in {\mathcal C_\infty}$,  
\begin{equation}\label{eq:eq1}
	|V_G(x,\omega)-V_G(y,\omega)|_{\infty}\leq \d_\omega(x,y)\mathrm{ess \, sup}_{\P_0} |G(x,\omega)|_{\infty}.
\end{equation}
For a fixed $x\in {\mathcal C_\infty}\cap  [-r,r]^d$, if $|V_G(x,\omega)|_{\infty}\leq \eps r$ for all $r\geq r_0$, there is nothing else to do. Otherwise, choose $r_1(\omega)$ large enough so that Lemma \ref{lemma:sublinear_prelim1} is true for $r\geq r_1$ and $a=2^{-\ell}$ (of course the lemma is still true if we replace $n\in \N$ by $r\in \R$). Now, if $r\geq r_0\vee r_1$, by \eqref{eq:eq15}, for every $x\in [-r,r]^d\cap {\mathcal C_\infty}$ satisfying $|V(x,\omega)|_\infty > \eps r$, we find some $y\in [-r,r]^d\cap {\mathcal C_\infty}$ such that $2^{-\ell}\delta r<|x-y|_\infty\leq \delta r$ and $|V(y,\omega)|\leq \eps r$. In particular, $\d_\omega(x,y)\leq \rho |x-y|_\infty\leq \rho \delta r$. Hence, by \eqref{eq:eq1}, we deduce that 
\begin{equation}\label{eq:eq13}
	\begin{aligned}
		|V_G(x,\omega)|_{\infty}\leq |V_G(y,\omega)|_{\infty}+|V_G(x,\omega)-V_G(y,\omega)|_{\infty}
		&\leq \eps r+ \d_\omega(x,y) ~\mathrm{ess\,sup}_{\P_0} |G(x,\omega)|_{\infty}\\
		&\leq \eps r + \delta \rho r~\mathrm{ess\,sup}_{\P_0} |G(x,\omega)|_{\infty}.
	\end{aligned}
	\end{equation}
 Since $\delta\to 0$ as $\eps \to 0$, this finishes the proof, once we complete the

\noindent{\bf Proof of \eqref{eq:eq15}:} By Theorem \ref{thm:weaksublinear}, there is a measurable set $A_1$ with $\mathbb{P}_0(A_1)=1$, so that for all $\omega\in A_1$,\begin{equation}\label{eq:eq2}
	\limsup_{r\to\infty}\frac{1}{r^d}\int_{ {\mathcal C_\infty}\cap[-r,r]^d}\1_{\{|V_G(x,\omega)|_{\infty}>\eps r\}}\d x
=0.
\end{equation}
 
On the other hand, by Lemma \ref{lemma:density_box} we know that for a fixed
box $C=[a_1,b_1]\times [a_2,b_2]\times\dots\times [a_d,b_d]$, there exists a
measurable set $A_C$ satisfying $\mathbb{P}(A_C)=1$ and 
\begin{equation}
\label{eq:eq3a}
	\lim_{r\to\infty}\frac{\lambda_d({\mathcal C_\infty}\cap  rC)}{\lambda_d(rC)}=p_\infty~\text{ for all } \omega \in A_C.
\end{equation}
Choose any $\kappa=\kappa(d,\mathbb{P})\in (0,1)$ (which exists due to \eqref{eq:eq0}) and $c=(d,\mathbb{P})>0$ satisfying 
\begin{align}
	 &1-\kappa >\frac{1}{p_\infty 2^{d(\ell-1)}}\qquad\mbox{and}%\label{eq:eq11},\\
	 \quad c\left(1-\kappa-\frac{1}{p_\infty 2^{d(\ell-1)}}\right)>1\label{eq:eq12}.
\end{align}  
Next, for each $\eps>0$, let 
\begin{equation}
\label{eq:eq3b}
\delta:=\left(\frac{\eps c}{p_\infty}\right)^{1/d}.	
\end{equation}
We can cover $[-1,1]^d$ with finitely many cubes $C_1,\cdots, C_m\subset
[-1,1]^d$ of side $\delta$. In particular, for every $x\neq y$ in the same box
we will have $|x-y|_{\infty}<\delta$. By Lemma \ref{lemma:density_box} applied
to these boxes, we deduce that there exists a measurable set $A_2$ with
$\mathbb{P}_0(A_2)=1$ such that for all $\omega\in A_2$ and $1\leq i\leq m$ we
have 
\begin{equation}\label{eq:eq4}
	\lim_{r\to\infty}\frac{\lambda_d({\mathcal C_\infty}\cap
          rC_i)}{\lambda_d(rC_i)}=p_\infty.v 
\end{equation}
Let $A:=A_1\cap A_2$. Then for every $\omega\in A$ there exists some $r_0=r_0(\omega)$ such that for all $r\geq r_0$ and $1\leq i\leq m$, 
\begin{equation}\label{eq:eq5}
	\begin{aligned}
		&\lambda_d\big(\mathcal{C}_{\infty}\cap[-r,r]^d\cap \{|V_G(\cdot,\omega)|_{\infty}>\eps r\}\big)<\eps r^d,\qquad\mbox{and}\\
		&\lambda_d({\mathcal C_\infty}\cap rC_i) \geq p_\infty(1-\kappa)\lambda_d(rC_i)  =r^d \delta^d p_\infty(1-\kappa) =\eps c(1-\kappa) r^d.
	\end{aligned}
\end{equation}
For every fixed $x\in [-r,r]^d\cap {\mathcal C_\infty}$ that satisfies $|V_G(x,\omega)|_{\infty}>\eps r$, we have $x\in rC_i$ for some $1\leq i\leq m$, so that $x\in {\mathcal C_\infty}\cap rC_i$. We decompose $\lambda_d({\mathcal C_\infty}\cap rC_i)$ as 

\begin{equation*}
\begin{aligned}
	\lambda_d \big ({\mathcal C_\infty}\cap rC_i\big)&=\lambda_d\big({\mathcal C_\infty}\cap rC_i\cap \{|V_G(\cdot,\omega)|_{\infty}>\eps r\}\big)\\
	&\qquad\qquad+ \lambda_d\big({\mathcal C_\infty}\cap rC_i\cap \{|V_G(\cdot,\omega)|_{\infty}\leq \eps r\}\big).
\end{aligned}
\end{equation*}
By \eqref{eq:eq5}, and noting that $rC_i\subset [-r,r]^d$, we know that $\lambda_d({\mathcal C_\infty}\cap rC_i\cap \{|V_G(\cdot,\omega)|_{\infty}>\eps r\})< \eps r^d$. This  inequality, combined with the equality above allow us to deduce that \begin{equation}\label{eq:eq6}
	\lambda_d\big({\mathcal C_\infty}\cap rC_i\cap \{|V_G(\cdot,\omega)|_{\infty}\leq \eps r\}\big)\geq \eps r^d\left(c(1-\kappa)-1\right)>0,\end{equation}
	and the last inequality holds since $c>\frac{1}{1-\kappa}$ by \eqref{eq:eq12}. Next, we decompose the Lebesgue measure of $D':={\mathcal C_\infty}\cap rC_i\cap \{|V_G(\cdot,\omega)|_{\infty}\leq\eps r\}$ as 
\begin{equation}\label{eq:eq9}
\lambda_d(D')= \lambda_d\big(D'\cap B^{\infty}_{r\delta/2^\ell}(x)\big)+\lambda_d\big(D' \cap B^{\infty}_{r\delta/2^{\ell}}(x)^c\big),
\end{equation}
where $B^{\infty}_{r\delta/2^{\ell}}(x)$ is the ball centered at $x$ of radius $r\delta/2^{\ell}$ with respect to the $|\cdot|_\infty$ norm, which is a cube of side $r\delta/2^{\ell-1}$. We conclude that \begin{equation}\label{eq:eq8}
	\lambda_d\big(D'\cap B^{\infty}_{r\delta/2^\ell}(x) \big )\leq (r\delta/2^{\ell-1})^d=\frac{\eps c r^d }{p_\infty 2^{d(\ell-1)}},\qquad\mbox{so that}
	\end{equation} 
 \begin{equation}\label{eq:eq10}
	\lambda_d(D'\cap  B^{\infty}_{r\delta/2^{\ell}}(x)^c)\geq \eps r^d\left(c(1-\kappa)-1-\frac{c}{2^{d(\ell-1)}p_\infty}\right).
	\end{equation}
by combining \eqref{eq:eq6}, \eqref{eq:eq9} and \eqref{eq:eq8}. By the choice of $c$ in \eqref{eq:eq12}, we deduce that 
\begin{equation}\label{positive-measure}
c(1-\kappa)-1-\frac{c}{2^{d(\ell-1)}p_\infty}>0,\qquad\mbox{thus}\qquad  \lambda_d\big (D'\cap  B^{\infty}_{r\delta/2^{\ell}}(x)^c\big)>0.
\end{equation}
Finally, recalling the notation of $A^{\ssup x}$ from \eqref{def-Ax} and that of
$E(r)$ from \eqref{def-C-D-E}, 
\[
  \big\{ D'\cap B^{\infty}_{r\delta/2^{\ell}}(x)^c\big\} \subset \big(E(r)\cap
  C_{\delta r}(2^{-\ell})\big)^{\ssup x} \cap
  \big\{|V_G(\cdot,\omega)|_\infty\leq \eps r\big\},
\]
so by \eqref{positive-measure}, the claim \eqref{eq:eq15} follows, proving Theorem \ref{thm:sublinearv2}. Therefore also Theorem  
\ref{thm:sublinear} is proved.
\end{proof}

\appendix

\noindent{\bf Acknowledgement.} The research of the third author
is supported by Deutsche Forschungsgemeinschaft (DFG, German research
foundation), under Germany's Excellence Strategy EXC 2044-390685587,
Mathematics M\"unster: Dynamics-Geometry-Structure.

\end{document}